\documentclass[12pt]{amsart}

\newtheorem{theorem}{Theorem}[section]
\newtheorem{lemma}[theorem]{Lemma}
\newtheorem{proposition}[theorem]{Proposition}
\theoremstyle{definition}
\newtheorem{definition}[theorem]{Definition}

\theoremstyle{remark}
\newtheorem{remark}[theorem]{Remark}

\numberwithin{equation}{section}

\usepackage{tikz-cd}
\usepackage{graphicx,geometry}%
\usepackage{multirow}%
\usepackage{amsmath,amssymb,amsfonts}%
\usepackage{amsthm}%
\usepackage{mathrsfs}%
\usepackage[title]{appendix}%
\usepackage{xcolor}%
\usepackage{textcomp}%
\usepackage{manyfoot}%
\usepackage{booktabs}%
\usepackage{algorithm}%
\usepackage{algorithmicx}%
\usepackage{algpseudocode}%
\usepackage{listings}%

\allowdisplaybreaks[1]
\usepackage{epstopdf}
\usepackage{booktabs}
\usepackage{graphicx}
\usepackage{subfigure}

\usepackage{mathtools}
\def\bu{\boldsymbol{u}}
\def\d{\operatorname{d}}

\def\v{\boldsymbol{v}}
\def \w {\boldsymbol{w}}
\def\vh{\boldsymbol{v}_h}

\def\div{\operatorname{div}}
\def\curl{\operatorname{curl}}

\def\wh{\boldsymbol{w}_h}

\def \Vrk{\boldsymbol{V}_{r-1, k+1}(K)}

\begin{document}
	\begin{center}
		\begin{LARGE}
			{A grad-curl conforming virtual element method for a grad-curl problem linking the 3D quad-curl problem and Stokes system}
		\end{LARGE}
	\end{center}
	
	\title{}
	\author{Xiaojing Dong$^1$}
	\address{$^1$Hunan Key Laboratory for Computation and Simulation in Science and Engineering, Key Laboratory of Intelligent Computing \& Information Processing of Ministry of Education, School of Mathematics and Computational Science, Xiangtan University, Er huan Road, Xiangtan,  411105, Hunan, P.R. China}
	\email{dongxiaojing99@xtu.edu.cn}
	\thanks{The research was supported by the National Natural Science Foundation of China (Nos: 12071404, 11971410, 12071402, 12261131501), Young Elite Scientist Sponsorship Program by CAST (No: 2020QNRC001), the Science and Technology Innovation Program of Hunan Province (No: 2024RC3158), Key Project of Scientific Research Project of Hunan Provincial Department of Education (No: 22A0136), Postgraduate Scientific Research Innovation Project of Hunan Province (No: LXBZZ2024112), the Project of Scientific Research Fund of the Hunan Provincial Science and Technology Department (No.2023GK2029, No.2024ZL5017), and Program for Science and Technology Innovative Research Team in Higher Educational Institutions of Hunan Province of China.}
	
	\author{Yibing Han$^{1\dagger}$}
	\thanks{$^\dagger$ Corresponding author.}
	
	\email{202331510114@smail.xtu.edu.cn}
	
	\author{Yunqing Huang$^1$}
	\email{hangyq@xtu.edu.cn}
	
	\subjclass[2020]{Primary 65N30, 65N15}

	\keywords{$\boldsymbol{H}(\operatorname*{grad-curl})$-conforming; virtual elements; quad-curl problem; vector potential formulation;  Stokes complex; polyhedral meshes}
	
	\begin{abstract}
		Based on the Stokes complex with vanishing boundary conditions and its dual complex, we reinterpret a grad-curl problem arising from the quad-curl problem as a new vector potential formulation of the three-dimensional Stokes system.
		By extending the analysis to the corresponding boundary value problems and the accompanying trace complex, we construct a novel $\boldsymbol{H}(\operatorname{grad-curl})$-conforming virtual element space with arbitrary approximation order that satisfies the exactness of the associated discrete Stokes complex. In the lowest-order case, three degrees of freedom are assigned to each vertex and one to each edge. For the grad-curl problem, we rigorously establish the interpolation error estimates, the stability of discrete bilinear forms, and the convergence of the proposed element on polyhedral meshes. As a discrete vector potential formulation of the Stokes problem, the resulting system is pressure-decoupled and symmetric positive definite. Some numerical examples are presented to verify the theoretical results.
	\end{abstract}

	\maketitle

	\section{Introduction}\label{sec1}
	
	\label{sec:into}
	Let $\Omega \subset \mathbb{R}^3$ be a contractible Lipschitz polyhedron with boundary $\Gamma$ and unit outward normal $\boldsymbol{n}$. We consider the following de Rham complex with enhanced smoothness and homogeneous boundary conditions:
	\begin{equation}\label{StokesComplex}
		0 \stackrel{}{\longrightarrow} H_0^{1}(\Omega) \stackrel{\nabla}{\longrightarrow} \boldsymbol{V}_0(\Omega) \stackrel{\nabla \times}{\longrightarrow}\boldsymbol{H}_0^1(\Omega) \stackrel{\nabla \cdot }{\longrightarrow} L_0^2(\Omega) \longrightarrow 0,
	\end{equation}
	where $$\boldsymbol{V}_0( \Omega) := \{\boldsymbol{v} \in \boldsymbol{L}^2(\Omega) : \nabla \times \boldsymbol{v} \in \boldsymbol{H}_0^1(\Omega) \text{ and } \v \times \boldsymbol{n}=0 \text{ on } \Gamma\}$$ 
	denotes the $\boldsymbol{H}(\operatorname{grad-curl})$-conforming space. This complex is termed the Stokes complex because its last two spaces provide the theoretical foundation for discretizing the Stokes equations with no-slip boundary conditions \cite{John2017}: find the velocity $\bu$ and pressure $p$ such that
	\begin{equation}
		\begin{aligned}\label{StokesProblem}
			-\nu \Delta\boldsymbol{u} + \nabla p &= \boldsymbol{f} \quad \text{in }\Omega,\\
			\nabla\cdot \boldsymbol{u} &=0 \quad \text{in }\Omega, \\
			\boldsymbol{u} &= 0 \quad \text{on }\Gamma,
		\end{aligned}
	\end{equation}
	where $\nu>0$ is the dynamic viscosity, $\boldsymbol{f} \in \boldsymbol{H}^{-1}(\Omega)$ is the external body force.
	The divergence constraint $\nabla\cdot \bu=0$ enforces fluid incompressibility, equivalent to mass conservation.
	Significant research has been devoted to the construction of stable, divergence-free velocity-pressure pairs that satisfy a discrete version of the Stokes complex \eqref{StokesComplex} or its variants, as seen in \cite{VEMforStokes,Chen2024,Falk2013,ZZMQuadCurl,Neilan2015}.  
	\par 
	Another important application of the Stokes complex \eqref{StokesComplex} utilizes its first two spaces to solve the following quad-curl problem \cite{ZZMQuadCurl,HuangXHQuadCurl}: given vector fields $\boldsymbol{j}$, find $\boldsymbol{\psi}$ such that
	\begin{equation}\label{primalForm}
		\begin{aligned}
			(\nabla\times)^4  \boldsymbol{\psi} &=\boldsymbol{j} \quad\text {in } \Omega, \\
			\nabla \cdot \,  \boldsymbol{\psi} &=0 \quad\text{in } \Omega, \\
			\boldsymbol{\psi} \times \boldsymbol{n} &=0 \quad\text{on } \Gamma, \\
			\nabla \times \boldsymbol{\psi} \times \boldsymbol{n} &=0 \quad\text {on } \Gamma.
		\end{aligned}
	\end{equation}
	The quad-curl operator $(\nabla\times)^4$ plays a pivotal role in diverse applications,  including inverse electromagnetic scattering in non-homogeneous media \cite{Cakoni2010,Cakoni2007,Monk2012} and magnetohydrodynamics (MHD) \cite{Zheng2011}. We also refer the reader to other effective numerical methods for this problem \cite{BrennerCavanaughSung2024,ChenG2021,ChenHuangZhang2025,Hu2020,WangWangZhang2023,ZJQQuadCurl}.
	\par 
	In this paper, we present several new Helmholtz-type decompositions for the dual spaces of $\boldsymbol{V}_0(\Omega)$, $\boldsymbol{H}_0^1(\Omega)$, and $\boldsymbol{X}(\Omega)$, where $\boldsymbol{X}(\Omega)$ denotes the $\boldsymbol{L}^2$-orthogonal complement in $\boldsymbol{V}_0(\Omega)$ defined as
	\begin{equation*}
		\boldsymbol{X}(\Omega) := \{\boldsymbol{v} \in \boldsymbol{V}_0(\Omega): (\boldsymbol{v}, \nabla q) = 0, \ \forall q \in H_0^1(\Omega)\}.
	\end{equation*}
	These results are derived by constructing two commutative diagrams that connect short complexes with their dual counterparts \cite{ChenLong2018}, thereby establishing a fundamental link between the quad-curl and Stokes problems. Specifically, given $\boldsymbol{j}\in \boldsymbol{X}'(\Omega)$, the dual space of $\boldsymbol{X}(\Omega)$ satisfying $\boldsymbol{X}'(\Omega)=\boldsymbol{H}^{-2}(\Omega) \cap \ker(\nabla\cdot)$, the exactness of the dual complex guarantees the existence of $\boldsymbol{f}_0 \in \boldsymbol{H}^{-1}(\Omega)$ such that $\boldsymbol{j} = \nabla \times \boldsymbol{f}_0$. The corresponding weak formulation of \eqref{primalForm} is then given by: find $\boldsymbol{\psi}\in\boldsymbol{X}(\Omega)$ such that
	\begin{equation}\label{GradCurlVar}
		(\nabla \nabla\times \boldsymbol{\psi}, \nabla\nabla\times \boldsymbol{\phi})=\langle\boldsymbol{f}_0,\nabla\times \boldsymbol{\phi}\rangle, \quad \forall\boldsymbol{\phi}\in \boldsymbol{X}(\Omega).
	\end{equation}
	To distinguish it from the primal weak form of the quad-curl problem \eqref{primalForm} with the right-hand side $\langle \boldsymbol{j}, \boldsymbol{\phi}\rangle$, 
	we refer to the formulation \eqref{GradCurlVar}
	as \emph{grad-curl problem}.  Moreover, since both $\boldsymbol{f}$ and $\boldsymbol{f}_0$ belong to $\boldsymbol{H}^{-1}(\Omega)$, the identification $\boldsymbol{j} = \frac{1}{\nu} \nabla \times  \boldsymbol{f} $ allows us to interpret the grad-curl problem \eqref{GradCurlVar} as a novel vector potential formulation for the Stokes problem \eqref{StokesProblem}.  This implies that the unique solution $\bu$ to \eqref{StokesProblem} admits a divergence-free vector potential given by the unique solution $\boldsymbol{\psi}$ to \eqref{primalForm}, such that $$\boldsymbol{u} = \nabla \times \boldsymbol{\psi}.$$
	Moreover, we define a suitable trace complex that enables the grad–curl problem \eqref{GradCurlVar} with non-homogeneous boundary conditions to serve as a vector potential formulation for the Stokes problem \eqref{StokesProblem} endowed with slip boundary conditions.
	\par 
	The virtual element method (VEM) \cite{VEM2013,VEM2014} is a natural extension of the finite element method (FEM) to polygonal and polyhedral meshes. It retains the conformity of FEM while offering greater flexibility and ease of design. This is achieved by constructing the basis functions as solutions to local boundary value problems on each element, thereby avoiding explicit polynomial representations while allowing a straightforward extension to higher orders; see \cite{VEM2013,VEMfordivcurl, Antonietti2014, VEMforStokes,VEMforMaxwell,ZJQQuadCurl,VEMforMHD}.
	In two dimensions, we refer to the VEMs for the quad-curl problem \cite{ZJQQuadCurl} and the Stokes stream function formulation \cite{Antonietti2014,Mora2025}; however, it should be noted that these address two entirely unrelated continuous problems. For the three-dimensional case, \cite{VEMforStokes} has first proposed a 
	$\boldsymbol{H}(\operatorname*{grad-curl})$-conforming virtual element space based on a non-homogeneous vector biharmonic problem. This problem serves as a classical vector potential formulation for the Stokes system, as presented in \cite[Chapter I, Section 5.3]{GiraultRaviart1986}.  However, this space lacks polynomial completeness, making it unsuitable for discretizing the quad-curl problem \eqref{primalForm} and the grad-curl problem \eqref{GradCurlVar}. 
	\par 
	To address this limitation, we develop three new families of 
	$\boldsymbol{H}(\operatorname{grad-curl})$-conforming virtual element spaces $\boldsymbol{V}_{r-1,k+1}(\Omega)$, based on the grad-curl boundary value problem with parameter choices $r = k$, $k+1$, and $k+2$.
	This generalizes the discrete complex from \cite{VEMforStokes} to arbitrary approximation orders $r,$ $ k \ge 1$, defined as follows:
	\begin{equation}\label{IntroDiscomplex}
		\mathbb{R} \stackrel{\subset}{\longrightarrow} U_r(\Omega)\stackrel{\nabla}{\longrightarrow} \boldsymbol{V}_{r-1,k+1}(\Omega) \stackrel{\nabla \times }{\longrightarrow} \boldsymbol{W}_{k}(\Omega) \stackrel{\nabla \cdot }{\longrightarrow} Q_{k-1}(\Omega)\longrightarrow 0.
	\end{equation} 
	Furthermore, the exactness of the sequence from $\boldsymbol{V}_{r-1,k+1}(K)$ to $\boldsymbol{W}_{k}(K)$ is established by identifying the grad-curl boundary value problem as the vector-potential formulation of the non-homogeneous Stokes system.
	These families generalize the $\boldsymbol{H}(\operatorname{grad-curl})$-conforming finite elements in \cite{ZZMQuadCurl} to general polyhedral meshes. In particular, the lowest-order case ($r=k=1$) recovers the simplest finite element structure on simplex meshes, with degrees of freedom solely at the vertices and edges. Using interpolation operators, we establish a commutative diagram linking the discrete complex \eqref{IntroDiscomplex} to the continuous one. Furthermore, we derive interpolation error estimates by leveraging existing results from $\boldsymbol{H}(\operatorname{curl})$-conforming virtual element spaces \cite{VEMforGener} and vector $\boldsymbol{H}^1$-conforming spaces \cite{VEMforStokes, VEMforMHD}.
	\par
	To discretize the grad–curl problem \eqref{GradCurlVar}, we construct a discrete bilinear form on the space $\boldsymbol{V}_{r-1,k+1}(K)$. Its stability is established through a trace inequality for the functions in $\boldsymbol{V}_{r-1,k+1}(\Omega)$, leading to a scaled $\boldsymbol{H}(\operatorname{grad-curl})$ upper bound. We further prove the well-posedness of the resulting discrete grad–curl problem and identify it as a discrete vector potential VEM for the Stokes problem \eqref{StokesProblem}. 
	Convergence is established in the $\boldsymbol{H}(\operatorname{grad-curl})$-norm, as well as in the $\boldsymbol{L}^2$-norm and the $\boldsymbol{H}^1$-seminorm involving the curl operator. The latter convergence rate aligns with that of the vector $\boldsymbol{H}^1$-conforming VEM \cite{VEMforStokes} applied to the Stokes problem \eqref{StokesProblem}. 
	In addition, we introduce a reduced version of the discrete complex \eqref{IntroDiscomplex} to compare the number of degrees of freedom between the velocity-pressure formulation from \cite{VEMforStokes} and our vector potential formulation. From a theoretical perspective, we also discuss the distinct advantages of our approach, such as its symmetric positive definite structure and inherent pressure-decoupling.
	\par
	The remainder of this paper is organized as follows. Section 2 presents some notation and preliminaries.
	Section 3 demonstrates how the grad-curl problem provides a link between the quad-curl problem and the Stokes system.
	Section 4 constructs the exact virtual element complex and establishes interpolation error estimates. Section 5 establishes the well-posedness of the discrete problem and derives optimal error estimates. Section 6 validates the theory through numerical experiments.
	
	\section{Preliminaries}	
	\subsection{Notation}
	We consider a sequence of conforming polyhedral meshes $\{\mathcal{T}_h\}_h$ partitioning $\Omega$ into polyhedra $K$, where $h := \max_{K \in \mathcal{T}_h} h_K$ denotes the maximum element diameter. For each element $K \in \mathcal{T}_h$, let $\partial K$ denote its boundary with unit outward normal $\boldsymbol{n}_{\partial K}$. Each edge $e$ of $K$ is associated with a unit tangent $\boldsymbol{t}_e$. On any face $f \in \partial K$, the restriction of $\boldsymbol{n}_{\partial K}$ is denoted by $\boldsymbol{n}_f$. The boundary $\partial f$  is equipped with the unit outward normal $\boldsymbol{n}_{\partial f}$ and the unit tangent $\boldsymbol{t}_{\partial f}$, oriented counter-clockwise with respect to $\boldsymbol{n}_{\partial f}$. For a general polyhedron $K$, let $N_v$, $N_e$, and $N_f$ be the number of its vertices, edges, and faces, respectively.
	\par 
	For any geometric entity $G$ (element, face, edge, or domain $\Omega$) with the unit outward normal  $\boldsymbol{n}_{\partial G}$,
	we define $h_G$ as the diameter of $G$, $|G|$ as its measure, and $\boldsymbol{b}_G$ as its barycenter. The mesh family $\{\mathcal{T}_h\}_h$ is assumed to satisfy the following regularity conditions uniformly in $h$ (see \cite{VEMforGener,VEMforStokes,VEMforMaxwell}):  \par
	\textbf{(A1)} $K$ is star-shaped with respect to a ball of radius $\geq \mu h_K$.  \par
	\textbf{(A2)} each face $f$ of $K$ is star-shaped with respect to a disk of radius $\geq \mu h_K$.  \par
	\textbf{(A3)} each edge $e$ of $K$ has a length of at least $\mu h_K$. \par 
	Here, $\mu > 0$ is a fixed constant independent of $h$. Throughout this paper, we use $a \lesssim b$ (resp. $a \gtrsim b$) to denote $a \leq C b$ (resp. $a \geq C b$), where $C > 0$ is a generic constant independent of $h$.
	\par 
	For any  $s\ge 0$, let $H^s(G)$ and $H^s_0(G)$ denote the standard Sobolev spaces over $G$, equipped with the norm $\|\cdot\|_{s,G}$ and the seminorm $|\cdot|_{s,G}$. We identify $H^0(G)$ with $L^2(G)$, where the norm $\|\cdot\|_G$ and the inner product $(\cdot, \cdot)_G$ are used. 
	The duality pairing is denoted by $\langle \cdot, \cdot \rangle_G$, and the corresponding dual space $H^{-s}(G)$ is defined, equipped with the norm $\|\cdot\|_{-s, G}$.
	When $G = \Omega$, the subscript $G$ is omitted. 
	For $k\in \mathbb{N}_0$, let $P_k(G)$ denote the space of polynomials of degree at most $k$ on the geometric entity $G$, with $P_{-1}(G)={0}$.
	Vector-valued Sobolev spaces are denoted by $\boldsymbol{H}^m(G; \mathbb{R}^d)$, $\boldsymbol{H}_0^m(G; \mathbb{R}^d)$, $\boldsymbol{L}^2(G; \mathbb{R}^d)$ and $\boldsymbol{P}_k(G;\mathbb{R}^d)$ for $d=2, 3$. When the ambient dimension $d$ is clear, we employ the simplified notation by writing, for example, $\boldsymbol{L}^2(\Omega)=\boldsymbol{L}^2(\Omega;\mathbb{R}^3)$, $\boldsymbol{L}^2(\partial K)=\boldsymbol{L}^2(\partial K;\mathbb{R}^3)$, and $\boldsymbol{L}^2(f)=\boldsymbol{L}^2(f;\mathbb{R}^2)$. 
	Standard differential operators ($\nabla$, $\Delta$, $\nabla \times$,  $\nabla \cdot$) and 
	Sobolev spaces $\boldsymbol{H}(\operatorname{curl};G),\boldsymbol{H}_0(\operatorname{curl};G), 	\boldsymbol{H}(\operatorname{div};G)$ and $	\boldsymbol{H}_0(\operatorname{div};G)$ retain their 	conventional meanings. 
	In particular, for a face $f \in \partial K$ situated in the $(x_1,x_2)$-plane, we define the two-dimensional vector curl for a scalar function $q$ by
	\begin{equation*}
		\overrightarrow{\nabla}_f\times  q= (\partial_{x_2}q, -\partial_{x_1}q)
	\end{equation*}
	and the scalar curl for a vector function $\v=(v_1(x_1,x_2),v_2(x_1,x_2))^T$ by
	\begin{equation*}
		\nabla_f \times \v =  \partial_{x_1}v_2- \partial_{x_2} v_1.
	\end{equation*}
	We also introduce the following functional spaces:
	\begin{gather*}
		\boldsymbol{V}(G):=\{\boldsymbol{v}\in \boldsymbol{L}^2(G;\mathbb{R}^3):\nabla\times \boldsymbol{v}\in \boldsymbol{H}^1(G;\mathbb{R}^3)\},\\
		\boldsymbol{V}_{0}(G):=\{ \boldsymbol{v}\in \boldsymbol{V}(G): \boldsymbol{u}\times \boldsymbol{n}_{\partial G} \text{ and } \nabla \times \boldsymbol{v} =0 
		\text{ on } \partial G \},
	\end{gather*}
	equipped with the graph norm
	\begin{equation*}
		\|\boldsymbol{v}\|^2_{\boldsymbol{V}(G)}=\|\boldsymbol{v}\|^2_G+\|\nabla\times\boldsymbol{v}\|_{1,G}^2.
	\end{equation*}   
	Note that $\boldsymbol{V}_0(G)$ is the closure of $\boldsymbol{C}_0^{\infty}(G)$ in the $\|\cdot\|_{\boldsymbol{V}(G)}$-norm, which follows from a standard argument analogous to that in \cite[Theorem 2.3]{QuadCurlXuLiWei}. Furthermore, incorporating the divergence-free condition, we define
	\begin{gather*}
		\boldsymbol{H}(\operatorname{div}^0;G):=\{ \v \in \boldsymbol{L}^2(G):\, (\v,\nabla q) = 0, \forall q\in H^1_0(G                                                                                                                                )\}, \\
		\boldsymbol{X}(G):=\boldsymbol{V}_0(G) \cap \boldsymbol{H}(\operatorname{div}^0;G).
	\end{gather*}
	The spaces $\boldsymbol{X}(\Omega)$ and $\boldsymbol{H}_0(\operatorname{curl;\Omega})\cap \boldsymbol{H}(\operatorname{div}^0;\Omega)$ satisfy the following Friedrichs inequality \cite[Remark 2.16 and Proposition 3.7]{Amrouche1998}.
	\begin{lemma}
		If $\Omega$ is a Lipschitz polyhedron with a connected boundary, there exists $s>\frac{1}{2}$ and positive constant $C_F$ such that
		\begin{equation}\label{FreIneq}
			\|\v\|_s\le C_F \|\nabla\times \v\|,\quad \forall \v \in  	\boldsymbol{X}(\Omega) \text{ or } \boldsymbol{H}_0(\curl;\Omega)\cap\boldsymbol{H}(\div^0;\Omega).
		\end{equation}
		The regularity exponent $s$ is typically $\frac{1}{2}$ for the general Lipschitz domains and increases to $s = 1$ when $\Omega$ is convex.
	\end{lemma}
	\subsection{The dual complex}
	Based on the generalized Helmholtz decomposition established in \cite{Chen2018}, this subsection characterizes the dual spaces of $\boldsymbol{H}_0^1(\Omega)$, $\boldsymbol{V}_0(\Omega)$, and $\boldsymbol{X}(\Omega)$, denoted by $\boldsymbol{H}^{-1}(\Omega)$, $\boldsymbol{V}'(\Omega)$, and $\boldsymbol{X}'(\Omega)$, respectively.
	\par 
	Following the definitions in \cite{Chen2018}, consider a Hilbert space $X$ with inner product $(\cdot,\cdot)_X$ and its dual space $X'$. The Riesz representation theorem establishes an isomorphism $J_X: X \to X'$, defined for any $w \in X$ by
	\begin{equation}
		\langle J_X w, v \rangle = (w, v)_X ,\quad \forall v \in X.
	\end{equation}
	The dual space $X'$ is equipped with the inner product
	\begin{equation*}
		(w',v')_{X'}:=((J_{X})^{-1}w',(J_X)^{-1}v')_X=\langle (J_{X})^{-1}w',v'\rangle=\langle w',(J_{X})^{-1}v'\rangle ,
	\end{equation*}
	which induces the corresponding norm on $X'$.
	Let $U$, $V$, and $W$ be Hilbert spaces.  Consider the short exact sequence:
	\begin{equation}\label{UVWcomplex}
		0  \stackrel{}{\longrightarrow} U \stackrel{\d_1}{\longrightarrow} V \stackrel{\d_2}{\longrightarrow} W.
	\end{equation}
	It follows from \cite[Remark 2.15]{Pauly2016} that the dual complex of \eqref{UVWcomplex} is exact:
	\begin{equation}\label{dualUVW}
		0 \longrightarrow W'/\ker(d_2') \stackrel{d_2'}{\longrightarrow} V' \stackrel{d_1'}{\longrightarrow} U' \longrightarrow 0.
	\end{equation}
	The exactness of \eqref{UVWcomplex} and \eqref{dualUVW} implies that, specifically,
	\begin{equation*}
		\operatorname{ker}(\d_1) = 0,\quad\ker(\d_2) = \operatorname{img}(\d_1),\quad \ker(\d_1')=\operatorname{img}(\d_2'), \quad\operatorname{img}(\d_1')=U'.
	\end{equation*}
	It is then apparent that  $d_1' J_V d_1$ defines an isomorphism from $U$ to $U'$. Setting $J_U = d_1' J_V d_1$ yields the following commutative diagram
	\begin{equation}\label{UVWcom1}
		\begin{tikzcd}
			W' \arrow[r,"d_2'"] 
			& V' \arrow[r,"d_1'"] 
			& U' \arrow[r] & 0 \\
			& V \arrow[u,"J_V"] 
			& U\arrow[l,"d_1"'] \arrow[u,"J_U" ].
		\end{tikzcd}
	\end{equation} 
	\begin{lemma}\label{HelDec1}
		\cite[Corollary 2.4]{Chen2018}
		Suppose that the short Hilbert sequence \eqref{UVWcomplex} is exact. Then we have the $(\cdot,\cdot)_{V'}$-orthogonal Helmholtz decomposition
		\begin{equation}\label{VdualDecomp}
			V'= d_2'(W'/\ker(d_2')) \oplus^{\bot} J_Vd_1 U.
		\end{equation}
	\end{lemma}
	The definition of the operator $J_V$ in the decomposition \eqref{VdualDecomp} depends on the inner product of the Hilbert space $V$.
	Sometimes we do not know the space $V$ explicitly, nor do we necessarily need to know it; accordingly, we present a more general result, namely that $J_V$ reduces to the identity operator on $\d_1 U$.
	\begin{lemma}\label{HelDec2}
		\cite[Corollary 2.5]{Chen2018}
		Suppose that we have a short exact Hilbert sequence \eqref{UVWcomplex} and  another short exact sequence
		\begin{equation*}
			W'\stackrel{d_2'}{\longrightarrow} \hat{V}\stackrel{d_1'}{\longrightarrow} U'\longrightarrow 0
		\end{equation*}
		with the following commutative diagram:
		\[
		\begin{tikzcd}
			W' \arrow[r,"d_2'"] 
			& \hat{V}\arrow[r,"d_1'"] 
			& U' \arrow[r] & 0 \\
			& V \arrow[u,"I"] 
			& U\arrow[l,"d_1"'] \arrow[u,"J_U" ],
		\end{tikzcd}
		\]
		where $I$ denotes an embedding operator. Then 
		\begin{equation*}
			\hat{V}=V'=\d_2'(W'/\ker(\
			d_2')) \oplus \d_1U.
		\end{equation*}
	\end{lemma}
	We now leverage the two lemmas to characterize the dual spaces $\boldsymbol{H}^{-1}(\Omega)$ and $\boldsymbol{V}'(\Omega)$, respectively. It is known that the Stokes complex \eqref{StokesComplex} is exact on contractible domains \cite{Arnold2018}. Additionally, we introduce the following complex for any $s\in \mathbb{R}$:
	\begin{equation}\label{Realscomplex}
		\mathbb{R} \stackrel{\subset}{\longrightarrow} H^{s+3}(\Omega)\stackrel{\nabla}{\longrightarrow} \boldsymbol{H}^{s+2}(\Omega) \stackrel{\nabla \times }{\longrightarrow} \boldsymbol{H}^{s+1}(\Omega)  \stackrel{\nabla \cdot }{\longrightarrow} H^{s}(\Omega)\longrightarrow 0.
	\end{equation}    
	This complex is exact on bounded domains that are starlike with respect to a ball \cite[p. 301]{Costabel2010}.
	\par
	On the Stokes complex \eqref{StokesComplex}, reducing the space $\boldsymbol{V}_0(\Omega)$ to its orthogonal complement $\boldsymbol{X}(\Omega)$ yields a short exact sequence
	\begin{equation*}
		0 \stackrel{}{\longrightarrow}\boldsymbol{X}(\Omega)\stackrel{\nabla \times}{\longrightarrow}\boldsymbol{H}_0^1(\Omega) \stackrel{\nabla \cdot }{\longrightarrow} L_0^2(\Omega) \longrightarrow 0,
	\end{equation*}
	which, combined with \cite[Remark 2.15]{Pauly2016}, implies the exactness of the dual complex 
	\begin{equation}\label{DualShortStokes}
		0\stackrel{}{\longrightarrow} L_0^2(\Omega)\stackrel{\nabla}{\longrightarrow} \boldsymbol{H}^{-1}(\Omega)\stackrel{\nabla \times}{\longrightarrow}\boldsymbol{X}'(\Omega)\longrightarrow 0.
	\end{equation}
	From the commutative diagram \eqref{UVWcom1} and the isomorphism $\Delta: \boldsymbol{H}_0^1(\Omega) \to \boldsymbol{H}^{-1}(\Omega)$, it follows that the operator $\nabla\times \Delta\nabla\times: \boldsymbol{X}(\Omega) \to \boldsymbol{X}'(\Omega)$ is an isomorphism, leading to the commutative diagram:
	\[
	\begin{tikzcd}
		L_0^2(\Omega) \arrow[r,"{\nabla }"] 
		& \boldsymbol{H}^{-1}(\Omega) \arrow[r,"{\nabla\times}"] 
		& \boldsymbol{X}'(\Omega) \arrow[r] & 0 \\
		& \boldsymbol{H}^1_0(\Omega) \arrow[u,"\Delta"] 
		& \boldsymbol{X}(\Omega) \arrow[l,"{\nabla \times }"'] \arrow[u,"{\nabla \times \Delta \nabla \times}" ].
	\end{tikzcd}
	\]
	By  Lemma \ref{HelDec1},  we obtain the $\boldsymbol{H}^{-1}$-orthogonal Helmholtz decomposition:
	\begin{equation}\label{H1DualDec}
		\boldsymbol{H}^{-1}(\Omega)=\nabla L_0^2(\Omega)\oplus^{\bot} \Delta(\nabla\times\boldsymbol{X}(\Omega)).
	\end{equation}
	From the exactness of \eqref{DualShortStokes} and \eqref{Realscomplex} with $s=-3$, we have 
	\begin{equation}\label{Xdual}
		\boldsymbol{X}'(\Omega)=
		\nabla\times \boldsymbol{H}^{-1}(\Omega)=\boldsymbol{H}^{-2}(\Omega)\cap\ker(\nabla\cdot).
	\end{equation}           
	
	\par 
	According to \cite[Remark 2.15 ]{Pauly2016}, the short exact sequence
	\begin{equation}
		0 \stackrel{}{\longrightarrow} H_0^{1}(\Omega) \stackrel{\nabla}{\longrightarrow} \boldsymbol{V}_0(\Omega) \stackrel{\nabla \times}{\longrightarrow}\boldsymbol{H}_0^1(\Omega),
	\end{equation}
	induces the exact dual complex:
	\begin{equation*}
		\boldsymbol{H}^{-1}(\Omega) \stackrel{\nabla \times}{\longrightarrow}
		\boldsymbol{V}'(\Omega)\stackrel{\nabla \cdot}{\longrightarrow}
		H^{-1}(\Omega) \stackrel{}{\longrightarrow}  0.
	\end{equation*}
	To characterize the dual space $\boldsymbol{V}'(\Omega)$, we construct another exact sequence of dual spaces. To this end, we define
	\begin{equation*}
		\boldsymbol{H}^{-2}(\div;\Omega):=\{\v\in\boldsymbol{H}^{-2}(\Omega): \nabla\cdot \v \in H^{-1}(\Omega)\},
	\end{equation*}
	equipped with the norm
	\begin{equation*}
		\|\v\|^2_{\boldsymbol{H}^{-2}(\div;\Omega)}:=\|\v\|^2_{-2}+\|\nabla\cdot \v\|^2_{-1}.
	\end{equation*}
	\begin{lemma}
		The complex
		\begin{equation*}
			\boldsymbol{H}^{-1}(\Omega)\stackrel{\nabla \times}{\longrightarrow}\boldsymbol{H}^{-2}(\div;\Omega) \stackrel{\nabla \cdot }{\longrightarrow} H^{-1}(\Omega) \longrightarrow 0
		\end{equation*}
		is exact.
	\end{lemma}
	\begin{proof}
		Setting $s=-3$ in the complex \eqref{Realscomplex} gives $\boldsymbol{H}^{-2}(\Omega)\cap\ker(\nabla\cdot)=\nabla\times \boldsymbol{H}^{-1}(\Omega)$. Apparently $\boldsymbol{H}^{-2}(\div;\Omega)\cap \ker(\nabla\cdot)= \boldsymbol{H}^{-2}(\Omega)\cap \ker(\nabla\cdot)$, establishing the exactness of the former complex.
		\par
		Furthermore, since $\Delta: H_0^1(\Omega) \to H^{-1}(\Omega)$ is  an isomorphism,  we have 
		\begin{equation*}
			H^{-1}(\Omega)= \nabla \cdot \nabla H_0^1(\Omega) \subset \nabla\cdot \boldsymbol{H}^{-2}(\div;\Omega).
		\end{equation*}
		This inclusion, together with the definition $\nabla\cdot \boldsymbol{H}^{-2}(\div;\Omega) \subset H^{-1}(\Omega)$, yields
		$$\nabla\cdot \boldsymbol{H}^{-2}(\div;\Omega) = H^{-1}(\Omega).$$
	\end{proof}
	Thus, we construct the commutative diagram     
	\[
	\begin{tikzcd}
		\boldsymbol{H}^{-1}(\Omega) \arrow[r,"{\nabla\times }"] 
		&\boldsymbol{H}^{-2}(\div;\Omega)\arrow[r,"{\nabla\cdot}"] 
		& H^{-1}(\Omega) \arrow[r] & 0 \\
		& \boldsymbol{V}_0(\Omega) \arrow[u,"I"] 
		& H_0^1(\Omega) \arrow[l,"{\nabla }"'] \arrow[u,"{\Delta}" ].
	\end{tikzcd}
	\]      
	By Lemma \ref{HelDec2}, we characterize the dual space $\boldsymbol{V}'(\Omega)$ as
	\begin{align}\label{DualVDec}
		\boldsymbol{V}'(\Omega)=\boldsymbol{H}^{-2}(\div;\Omega)=\nabla H_0^1(\Omega)\oplus \nabla\times \boldsymbol{H}^{-1}(\Omega) =\nabla H_0^1(\Omega) \oplus \boldsymbol{X}'(\Omega).
	\end{align}
	As established in \cite{Chen2018}, $J_{\boldsymbol{V}(\Omega)}$ acts as the identity operator on $\nabla H_0^1(\Omega)$. This implies that the decomposition \eqref{DualVDec} is orthogonal with respect to the inner product $(\cdot,\cdot)_{\boldsymbol{V}'(\Omega)}$, but not with respect to the $\boldsymbol{H}^{-l}$ inner product for $l=0,1,2$.
	\subsection{The trace complex}
	This subsection is devoted to defining the trace space of $\boldsymbol{V}(\Omega)$.
	Recall the classical trace operators \cite{Buffa2003} for a smooth scalar function $v: \Omega\to \mathbb{R}$ and a smooth vector field $\v:\Omega\to\mathbb{R}^3$:
	\begin{gather*}
		\gamma_0(v)(\boldsymbol{x}_0):=\lim_{\boldsymbol{x}\to \boldsymbol{x}_0} v(\boldsymbol{x}_0),\quad 
		\boldsymbol{\gamma}_0(\bu)(\boldsymbol{x}_0):=\lim_{\boldsymbol{x}\to\boldsymbol{x}_0}\v(\boldsymbol{x}),\\\boldsymbol{\gamma}_{\tau}(\v)(\boldsymbol{x}_0) := \lim_{\boldsymbol{x}\to\boldsymbol{x}_0} \boldsymbol{n}_{\boldsymbol{x}_0}\times\
		(\boldsymbol{v}(\boldsymbol{x})\times\boldsymbol{n}_{\boldsymbol{x}_0}),\quad
		\boldsymbol{\gamma}_{n}(\v)(\boldsymbol{x}_0):=\lim_{\boldsymbol{x}\to \boldsymbol{x}_0}\v(\boldsymbol{x})\cdot\boldsymbol{n}_{\boldsymbol{x}_0},
	\end{gather*}
	where $\boldsymbol{x}_0\in\Gamma$, $\boldsymbol{x}\in\Omega$, and $\boldsymbol{n}_{\boldsymbol{x}_0}$ denotes the unit outward normal to $\Gamma$ at  $\boldsymbol{x}_0$. It is known that for a Lipschitz domain $\Omega$, these trace operators admit continuous and surjective extensions to the corresponding Sobolev spaces in the weak setting:
	\begin{align}
		\label{Trace0}
		\gamma_0 &: H^s(\Omega) \rightarrow H^{s-1/2}(\Gamma) \text{ for } \frac{1}{2}<s<\frac{3}{2} \quad \text{\cite[Theorem 3.38]{McLean2000}},  \\
		\boldsymbol{\gamma}_{\tau} &: \mathbf{H}(\text{curl},\Omega) \rightarrow \mathbf{H}^{-1/2}(\text{curl}_{\Gamma}, \Gamma) \quad \text{\cite[Theorem 4.1]{Buffa2002}},  \\
		\boldsymbol{\gamma}_n &: \mathbf{H}(\text{div},\Omega) \rightarrow H^{-1/2}(\Gamma) \quad \text{\cite[Theorem 2.5, Corollary 2.8]{GiraultRaviart1986}}. 
	\end{align}
	The trace space $H^s(\Gamma)$ for $0 \le s \le 1$ is defined via localization and pullback under charts \cite[Chapter 3]{McLean2000}, with the definition extended by duality to $-1 \le s < 0$ via $H^{-s}(\Gamma)$.
	This approach, which utilizes local charts, also facilitates the definition of key surface differential operators. These include the surface gradient $\nabla_{\Gamma}$, the surface divergence $\nabla_{\Gamma}\cdot$, the surface vector curl $\overrightarrow{\nabla}_{\Gamma} \times$, and the surface scalar curl $\nabla_{\Gamma} \times$, as detailed in \cite{Buffa2002,Buffa2003}.
	The trace space of $\boldsymbol{H}(\curl;\Omega)$ is given by 
	\begin{equation*}
		\boldsymbol{H}^{-\frac{1}{2}}(\curl_{\Gamma};\Gamma) =\{\v\in \boldsymbol{H}_{\tau}^{-\frac{1}{2}}(\Gamma): \nabla_{\Gamma}\times\boldsymbol{v}\in H^{-\frac{1}{2}}(\Gamma)\},
	\end{equation*}
	where 
	\begin{equation*}
		\boldsymbol{H}_{\tau}^{-\frac{1}{2}}(\Gamma):=\{\v\in\boldsymbol{H}^{-\frac{1}{2}}(\Gamma): \boldsymbol{v}\cdot \boldsymbol{n}=0 \text{ almost everywhere on }\Gamma\}
	\end{equation*}
	denotes the dual of the range of the tangential trace operator applied to $\boldsymbol{H}^1(\Omega)$.
	The actions of these operators are summarized in the following commutative diagram:
	\begin{equation}\label{DeRhamTraceCom}
		\begin{tikzcd}
			& 
			H^1(\Omega) \arrow[r, "\nabla"] \arrow[d, "\gamma_0"] & 
			\boldsymbol{H}(\curl; \Omega)  \arrow[r, "\nabla \times"] \arrow[d, "\boldsymbol{\gamma}_{\tau}"] &
			\boldsymbol{H}( \div;\Omega) \arrow[d, "\boldsymbol{\gamma}_n"] & \\
			& 
			H^{\frac{1}{2}}(\Gamma) \arrow[r, "\nabla_{\Gamma}"] & 
			\boldsymbol{H}^{-\frac{1}{2}}(\operatorname{curl}_{\Gamma};\Gamma) \arrow[r, "\nabla_{\Gamma}\times "] & 
			H^{-\frac{1}{2}}(\Gamma) & 
		\end{tikzcd}
	\end{equation}
	On contractible Lipschitz domains, both the de Rham complex and its trace complex in \eqref{DeRhamTraceCom} are exact \cite{Arnold2018}.
	\par
	We begin by defining the trace space of $\boldsymbol{V}(\Omega)$. Consider the space
	\begin{equation*}
		\boldsymbol{H}^{-\frac{1}{2}}(\operatorname{grad-curl}_{\Gamma};\Gamma) =\{\v\in \boldsymbol{H}_{\tau}^{-\frac{1}{2}}(\Gamma): \nabla_{\Gamma}\times\boldsymbol{v}\in H^{\frac{1}{2}}(\Gamma)\}.
	\end{equation*}
	Following \cite[Theorem 4.1]{Buffa2002}, $\boldsymbol{\gamma}_{\tau}: \boldsymbol{V}(\Omega) \to  \boldsymbol{H}^{-\frac{1}{2}}(\operatorname{grad-curl}_{\Gamma};\Gamma)$ is continuous by 
	\begin{align*}
		\nabla_{\Gamma}\times \boldsymbol{\gamma}_{\tau}(\v) = (\nabla\times \boldsymbol{v})\cdot \boldsymbol{n} \in \boldsymbol{H}^{\frac{1}{2}}(\Gamma), \\   \|\nabla_{\Gamma}\times  \boldsymbol{\gamma}_{\tau}(\v)\|_{\frac{1}{2},\Gamma}\ \le \|\boldsymbol{\gamma}_0(\nabla\times \v)\|_{\frac{1}{2},\Gamma}\lesssim \|\v\|_{\boldsymbol{V}(\Omega)}
	\end{align*}
	Then, for all functions in $\boldsymbol{V}(\Omega)$, we define the continuous operator $\boldsymbol{\gamma}_{\curl}: \boldsymbol{V}(\Omega)\to \boldsymbol{H}^{\frac{1}{2}}(\Gamma)$
	\begin{equation*}
		\boldsymbol{\gamma}_{\curl}(\v):= (\nabla\times\v)|_{\Gamma}=\boldsymbol{\gamma}_0(\nabla\times\v).
	\end{equation*}
	The relationships between these trace operators are captured by the following commutative diagram:
	\begin{equation}\label{EnhaceDeRhamTraceCom}
		\begin{tikzcd}
			& 
			H^1(\Omega) \arrow[r, "\nabla"] \arrow[dd, "\gamma_0"'] & 
			\boldsymbol{V}( \Omega)  \arrow[r, "\nabla \times"] \arrow[dd, "\boldsymbol{\gamma}_{\tau}"'] \arrow[rd, dashed, "\boldsymbol{\gamma}_{\curl}"'] & 
			\boldsymbol{H}^1( \Omega) \arrow[d, "\boldsymbol{\gamma}_0"] \arrow[dd, bend left=50, dashed, "\gamma_{\boldsymbol{n}}"] & \\
			& & & \boldsymbol{H}^{\frac{1}{2}}(\Gamma) \arrow[d, "\cdot \boldsymbol{n}"] & \\
			& 
			H^{\frac{1}{2}}(\Gamma) \arrow[r, "\nabla_{\Gamma}"] & 
			\boldsymbol{H}^{-\frac{1}{2}}(\operatorname{grad-curl}_{\Gamma};\Gamma) \arrow[r, "\nabla_{\Gamma} \times"] & 
			H^{\frac{1}{2}}(\Gamma) &.
		\end{tikzcd}
	\end{equation}
	Note that for any $\boldsymbol{v} \in \boldsymbol{H}^{-\frac{1}{2}}(\operatorname{grad-curl}_{\Gamma};\Gamma) \cap \ker(\nabla_{\Gamma}\times)\subset \boldsymbol{H}^{-\frac{1}{2}}(\curl_{\Gamma};\Gamma)\cap \ker(\nabla_{\Gamma}\times)$, there exists $q \in H^{\frac{1}{2}}(\Gamma)$ such that $\nabla_{\Gamma} q = \boldsymbol{v}$, which implies that the trace complex in \eqref{EnhaceDeRhamTraceCom} is also exact. 
	Finally, we define the trace space of $\boldsymbol{V}(\Omega)$:
	\begin{equation*}
		\boldsymbol{Y}(\Gamma):=\{(\boldsymbol{h}, \boldsymbol{g})\in  \boldsymbol{H}^{-\frac{1}{2}}(\operatorname{grad-curl}_{\Gamma};\Gamma)\times \boldsymbol{H}^{\frac{1}{2}}(\Gamma): \nabla_{\Gamma}\times \boldsymbol{h}=\boldsymbol{g}\cdot \boldsymbol{n}\}.
	\end{equation*}
	Since $\Gamma$ is a closed connected surface, the surface Stokes theorem implies that for any $(\boldsymbol{h},\boldsymbol{g}) \in \boldsymbol{Y}(\Gamma)$,
	\begin{equation}\label{hgCom}
		\int_{\Gamma} \boldsymbol{g} \cdot \boldsymbol{n}\d S= \int_{\Gamma} \nabla_{\Gamma}\times \boldsymbol{h} \d S = 0.
	\end{equation}
	\begin{theorem}\label{SurjForV}
		On a contractible Lipschitz domain $\Omega$, the map $\v \to \{ \boldsymbol{\gamma}_{\tau}(\v), \boldsymbol{\gamma}_{\curl}(\v)\} : \boldsymbol{V}(\Omega) \to \boldsymbol{Y}(\Gamma)$ is continuous and surjective.
	\end{theorem}
	\begin{proof}
		The continuity is immediate, as the continuous operators satisfy 
		\begin{equation*} 
			\nabla_{\Gamma}\times\boldsymbol{\gamma}_{\tau}(\boldsymbol{v})=\boldsymbol{\gamma}_{\curl}(\boldsymbol{v})\cdot\boldsymbol{n}
		\end{equation*} in view of the commutative diagram \eqref{EnhaceDeRhamTraceCom}.
		Let $(\boldsymbol{h},\boldsymbol{g})\in \boldsymbol{Y}(\Gamma)$. In view of the compatibility relation \eqref{hgCom}, we now consider the Stokes equations \eqref{StokesProblem} that take the non-homogeneous boundary conditions $\boldsymbol{u} = \boldsymbol{g}$ on $\Gamma$.
		Then, by \cite[Theorem 3.4]{GiraultRaviart1986}, there exists $\boldsymbol{\varphi}^{g}  \in \boldsymbol{V}(\Omega)$ satisfying 
		\begin{equation*}
			\nabla\times \boldsymbol{\varphi}^g=\boldsymbol{u} \quad \text{and} \quad \nabla\cdot \boldsymbol{\varphi}^g=0.
		\end{equation*}
		It follows from the definition of the  trace space $\boldsymbol{Y}(\Gamma)$ and the commutative diagram \eqref{EnhaceDeRhamTraceCom} that
		\begin{equation*}
			\nabla_{\Gamma}\times (\boldsymbol{h}-\boldsymbol{\gamma}_{\tau}(\boldsymbol{\varphi}^g))=\nabla_{\Gamma}\times \boldsymbol{h}-\boldsymbol{\gamma}_{\curl}(\boldsymbol{\varphi^g)}=\nabla_{\Gamma}\times \boldsymbol{h}-\boldsymbol{g}\cdot\boldsymbol{n}=0.
		\end{equation*}
		By the exactness of the trace complex in \eqref{EnhaceDeRhamTraceCom},  there exists $ q_0\in H^{\frac{1}{2}}(\Gamma)$ such that
		\begin{equation}\label{auxSurj}
			\nabla_{\Gamma} q_0 =  \boldsymbol{h}-\boldsymbol{\gamma}_{\tau}(\boldsymbol{\varphi}^g).
		\end{equation}
		The surjectivity of $\gamma_0$ ensures the existence of $q \in H^1(\Omega)$ satisfying $$\gamma_0(q)=q_0,$$
		which, combined with  the  commutativity  of \eqref{EnhaceDeRhamTraceCom} and \eqref{auxSurj}, yields $$\boldsymbol{\gamma}_{\tau}(\nabla q)=  \nabla_{\Gamma}\gamma_0(q)= \boldsymbol{h}-\boldsymbol{\gamma}_{\tau}(\boldsymbol{\varphi}^g). $$
		We therefore obtain $\boldsymbol{\varphi}^{\partial} = \nabla q + \boldsymbol{\varphi}^g \in \boldsymbol{V}(\Omega)$, satisfying the prescribed boundary conditions: 
		\begin{equation*}
			\boldsymbol{\gamma}_{\tau}(\boldsymbol{\boldsymbol{\varphi}}^{\partial})= \boldsymbol{\gamma}_{\tau}(\nabla q)+\boldsymbol{\gamma}_{\tau}(\boldsymbol{\varphi}^{g})=\boldsymbol{h} \text{ and }\boldsymbol{\gamma}_{\curl}(\boldsymbol{\varphi}^{\partial})=\boldsymbol{\gamma}_0(\nabla\times \boldsymbol{\varphi}^{\partial})=\boldsymbol{\gamma}_0(\nabla\times \boldsymbol{\varphi}^{g})=\boldsymbol{g}.
		\end{equation*}
		The proof is complete.
	\end{proof}
	
	\section{the grad-curl problem}
	With the preliminaries in place, we derive the grad-curl problem as the vector potential formulation of the Stokes problem \eqref{StokesProblem} and examine its connection to the quad-curl problem \eqref{primalForm} in this section. This derivation and analysis will be conducted for both homogeneous and non-homogeneous boundary cases.
	\subsection{Homogeneous boundary conditions}
	Let $\bu$ be the unique solution of the Stokes problem \eqref{StokesProblem}.
	As shown in \cite[Chapter I, section 5.3]{GiraultRaviart1986}, the velocity field $\bu \in \boldsymbol{H}_0^1(\Omega)$ can be expressed as 
	\begin{equation*}
		\bu=\nabla\times \boldsymbol{\psi}  \quad \text{with}\quad \nabla\cdot \boldsymbol{\psi}=0 \quad \text{in } \Omega.
	\end{equation*}
	We refer to the vector field $\boldsymbol{\psi}$  as the vector potential, which is uniquely determined to the biharmonic problem: find $\boldsymbol{\psi} \in \boldsymbol{\Psi}_0(\Omega)$ such that
	\begin{equation}\label{biharmonic problem}
		-\nu (\Delta \boldsymbol{\psi}, \Delta \boldsymbol{\phi}) = \langle\boldsymbol{f},\nabla\times \boldsymbol{\phi}\rangle, \quad \forall \boldsymbol{\phi} \in \boldsymbol{\Psi}_0(\Omega),
	\end{equation}
	where
	\begin{equation}\label{divH1space}
		\boldsymbol{\Psi}_0(\Omega):=\{\boldsymbol{v} \in\boldsymbol{V}_0(\Omega): \nabla\cdot \boldsymbol{v} \in H^1(\Omega) \text{ and } \int_{\Gamma} \boldsymbol{v} \cdot\boldsymbol{n} \d S =0\}.
	\end{equation}
	In fact,
	testing the first equation of \eqref{StokesProblem} with $\nabla \times \boldsymbol{\phi}$ for $\boldsymbol{\phi}\in \boldsymbol{\Psi}_0(\Omega)$ gives
	\begin{equation}\label{firstVar}
		-\nu\langle\Delta \nabla\times \boldsymbol{\psi},\nabla\times \boldsymbol{\phi}\rangle = \langle\boldsymbol{f},\nabla\times \boldsymbol{\phi}\rangle.
	\end{equation}
	Using the vector calculus identity
	\begin{equation}\label{vector calculus identity}
		\nabla\times \nabla \times \boldsymbol{\v} = -\Delta \boldsymbol{\v} +\nabla
		\nabla\cdot \boldsymbol{\v},
	\end{equation}
	the left-hand side of \eqref{firstVar} can be rewritten as
	\begin{equation}\label{variation}
		-\nu\langle\Delta \nabla\times \boldsymbol{\psi},
		\nabla \times \boldsymbol{\phi}\rangle =\nu(\nabla\times \nabla\times \boldsymbol{\psi},\nabla\times\nabla\times \boldsymbol{\phi})= \nu(\Delta \boldsymbol{\psi}, \Delta \boldsymbol{\phi}),
	\end{equation}
	which leads to \eqref{biharmonic problem}.
	This approach entirely relaxes the divergence-free condition, in exchange for imposing the extra divergence regularity $\nabla\cdot \boldsymbol{\psi} \in H^1(\Omega)$.  A major challenge in this setting is the construction of discrete subspaces that satisfies such strong continuity constraints.
	\par 
	Alternatively, we explicitly enforce the weak divergence-free condition by requiring $\boldsymbol{\psi} \in \boldsymbol{X}(\Omega)$. This requirement is consistent because any $\boldsymbol{\psi} \in \boldsymbol{\Psi}_0(\Omega)$ with $\nabla\cdot \boldsymbol{\psi}=0$ already lies in $\boldsymbol{X}(\Omega)$.
	Then the new vector potential formulation \eqref{GradCurlVar} can be derived by
	\begin{equation}\label{gradcurlVar}
		\nu(\nabla \nabla\times \boldsymbol{\psi}, \nabla\nabla\times \boldsymbol{\phi})=-\nu\langle\Delta \nabla \times \boldsymbol{\psi},\nabla\times  \boldsymbol{\phi}\rangle =\langle \boldsymbol{f},\nabla\times \boldsymbol{\phi}\rangle, \quad \forall \boldsymbol{\varphi}\in \boldsymbol{X}(\Omega).
	\end{equation}
	Hence, we have established  that the vector potential $\boldsymbol{\psi}$ is a solution of \eqref{gradcurlVar}.
	We identify this formulation as the grad-curl problem. Introducing a Lagrange multiplier, we give the following definition, which will be the target of our numerical discretization.
	\begin{definition}
		Given $\boldsymbol{f}\in \boldsymbol{H}^{s-1}(\Omega)$ with $s\ge 0$, find $(\boldsymbol{\psi},\lambda) \in \boldsymbol{V}_0(\Omega)\times H_0^1(\Omega)$ satisfies
		\begin{equation}\label{VectorPotentialProblem}
			\begin{aligned}
				(\nabla\nabla\times \boldsymbol{\psi},\nabla\nabla\times \boldsymbol{\phi})+(\nabla\lambda,\boldsymbol{\phi})&= \frac{1}{\nu} \langle\boldsymbol{f},\nabla\times \boldsymbol{\phi}\rangle,\quad \forall \boldsymbol{\phi}\in \boldsymbol{V}_0(\Omega),\\
				(\boldsymbol{\psi},\nabla q) &=0,\quad \forall q\in H_0^1(\Omega).
			\end{aligned}
		\end{equation}
	\end{definition}
	The Lagrange multiplier $\lambda$ vanishes identically, established by substituting 
	$\boldsymbol{\phi} = \nabla \lambda$
	into the first equation of \eqref{VectorPotentialProblem} and using the Poincaré inequality.
	\begin{remark}\label{Regu}
		The inf-sup condition holds:
		\begin{equation*}
			\sup_{\v\in\boldsymbol{V}_0(\Omega)/\{0\}}\frac{(\v,\nabla \lambda)}{\|\v\|_{\boldsymbol{V}(\Omega)}}\ge \frac{(\nabla\lambda,\nabla\lambda)}{\|\nabla\lambda\|} \gtrsim C_P \|\lambda\|_1,
		\end{equation*}
		where $C_p>0$ depends on the Poincaré inequality.  Combined with the Friedrichs inequality \eqref{FreIneq}, the grad-curl problem \eqref{VectorPotentialProblem} has a unique solution $\boldsymbol{\psi} \in \boldsymbol{X}(\Omega)$ satisfying 
		\begin{equation}
			\|\boldsymbol{\psi}\|_{s} \le C_F \|\nabla\times \boldsymbol{\psi}\|\le C_F C_P |\nabla\times\boldsymbol{\psi}|_1 \lesssim C_F C_P \frac{1}{\nu} \|\boldsymbol{f}\|_{-1}.
		\end{equation}
		Furthermore,
		$\nabla\times \boldsymbol{\psi} = \bu $ is the unique solution of the Stokes problem \eqref{StokesProblem} with $\boldsymbol{f}\in \boldsymbol{H}^{s-1}(\Omega)$. The regularity result of the Stokes problem  established in \cite{Dauge1989} applies: there exists $s>\frac{1}{2}$, such that 
		\begin{equation*}
			\nabla \times \boldsymbol{\psi} \in \boldsymbol{H}^{s+1}(\Omega).
		\end{equation*}
		Applying the integration by parts to the right-hand side of \eqref{VectorPotentialProblem} and employing the density argument, we obtain the following equivalent form:
		\begin{equation}\label{curlDeltacurlProblem2}
			-\nu\nabla\times \Delta\nabla\times \boldsymbol{\psi}= \nabla\times \boldsymbol{f} \text{ in } \nabla\times \boldsymbol{H}^{s-1}(\Omega).
		\end{equation}
	\end{remark}
	On the other hand, the weak form of the quad-curl problem \eqref{primalForm}  reads: given $\boldsymbol{j}\in\boldsymbol{H}^{s-2}(\Omega)\cap\ker(\nabla\cdot)$ with $s\ge 0$,
	find $\boldsymbol{\psi}\in \boldsymbol{H}_0(\curl^2;\Omega)\cap\boldsymbol{H}^0(\div;\Omega)$ such that 
	\begin{equation}\label{curl4var}
		(\nabla\times \nabla\times\boldsymbol{\psi},\nabla\times \nabla\times \boldsymbol{\phi})=\langle\boldsymbol{j},\boldsymbol{\phi}\rangle,
	\end{equation}
	where
	\begin{equation*}
		\boldsymbol{H}_0(\curl^2;\Omega):=\{\v\in\boldsymbol{H}_0(\curl;\Omega);\nabla\times\v\in\boldsymbol{H}_0(\curl;\Omega)\}.
	\end{equation*}
	From  \cite[Theorem 2.5]{Amrouche1998}, on a Lipschitz domain $\Omega$, there holds $$\boldsymbol{H}_0(\curl;\Omega)\cap \boldsymbol{H}_0(\div;\Omega) =\boldsymbol{H}_0^1(\Omega),$$  which, combined with $\nabla\times \boldsymbol{H}_0(\curl;\Omega)\subset\boldsymbol{H}_0(\div;\Omega)$, yields $$\boldsymbol{V}_0(\Omega)=\boldsymbol{H}_0(\curl^2;\Omega).$$ Consequently, by applying the identity \eqref{vector calculus identity}, the weak form \eqref{curl4var} can be reformulated as seeking $\boldsymbol{\psi} \in \boldsymbol{X}(\Omega)$ such that
	\begin{equation}\label{gradcurlproblem}
		(\nabla\nabla\times \boldsymbol{\psi},\nabla\nabla\times\boldsymbol{\phi})=\langle\boldsymbol{j},\boldsymbol{\phi}\rangle,
	\end{equation}
	which is equivalent to the $-$curl$\Delta$curl problem:
	\begin{equation}\label{curlDeltacurlProblem1}
		-\nabla\times \Delta\nabla\times \boldsymbol{\psi} = \boldsymbol{j} \text{ in } \boldsymbol{H}^{s-2}(\Omega)\cap\ker(\nabla\cdot).
	\end{equation}
	In view of \eqref{Xdual}, the source term in the least regular case ($s=0$) satisfies $\boldsymbol{j} \in \boldsymbol{H}^{-2}(\Omega) \cap \ker(\nabla\cdot) = \nabla\times \boldsymbol{H}^{-1}(\Omega) = \boldsymbol{X}'(\Omega)$, which implies the well-posedness of the problem \eqref{curlDeltacurlProblem1}. Furthermore, the Friedrichs inequality \eqref{FreIneq} and the Poincaré inequality yield the estimate:
	\begin{equation*}
		\|\boldsymbol{\psi}\|_s \le C_F \|\nabla\times \boldsymbol{\psi}\| \le C_F C_P |\nabla\times \boldsymbol{\psi}|_1 \lesssim C_F C_{P} \|\boldsymbol{j}\|_{-2}.
	\end{equation*}
	In contrast, for smoother data, specifically, when $s=1$ with $\boldsymbol{j} \in \boldsymbol{H}^{-1}(\Omega) \cap \ker(\nabla\cdot)$, and when $s=2$ with $\boldsymbol{j} \in \boldsymbol{H}^0(\div;\Omega)$ the quad-curl problem \eqref{primalForm} can be decoupled into two Maxwell systems and a Stokes system; see \cite{Chen2018,HuangXHQuadCurl}.
	\begin{remark}\label{jequCurlf}
		The exactness of the complex \eqref{Realscomplex} ensures that the right-hand sides of problems \eqref{curlDeltacurlProblem1} and \eqref{curlDeltacurlProblem2} are compatible. It follows that for any $\boldsymbol{j} \in \boldsymbol{H}^{s-2}(\Omega) \cap \ker(\nabla \cdot)$, one can find $\boldsymbol{f}_0 \in \boldsymbol{H}^{s-1}(\Omega)$ satisfying $\boldsymbol{j} = \nabla \times \boldsymbol{f}_0$. Therefore, the grad-curl problem \eqref{VectorPotentialProblem} and the quad-curl problem \eqref{primalForm} are equivalent under the relation $\boldsymbol{j} = \frac{1}{\nu} \nabla \times \boldsymbol{f}$.
	\end{remark}
	Then we have proved the following theorem.
	\begin{theorem}\label{homStokesThe}
		Assume that $\Omega$ is a contractible Lipschitz domain.
		If $\boldsymbol{j}=\frac{1}{\nu}\nabla\times \boldsymbol{f}$ in $\boldsymbol{H}^{-2}(\Omega)\cap\ker(\nabla\cdot)$, 
		the quad-curl problem \eqref{primalForm} has a unique solution $\boldsymbol{\psi}$, whose curl $\nabla\times \boldsymbol{\psi}$ solves the Stokes problem \eqref{StokesProblem} uniquely.
	\end{theorem}
	\begin{remark}
		To conclude, the relationships among these problems are formalized through the dual space decompositions in \eqref{H1DualDec} and \eqref{DualVDec}. From this perspective, the primal Stokes problem \eqref{StokesProblem} corresponds to the decomposition
		\begin{equation*}
			\nabla L_0^2(\Omega)\oplus \Delta (\boldsymbol{H}_0^1(\Omega)\cap\ker(\nabla\cdot))= \boldsymbol{H}^{-1}(\Omega) 
		\end{equation*}
		Applying the curl operator to both sides of this identity and using the relation $\nabla \times \boldsymbol{X}(\Omega) = \boldsymbol{H}_0^1(\Omega) \cap \ker(\nabla\cdot)$, we obtain the grad-curl problem \eqref{VectorPotentialProblem}:
		\begin{equation*}
			\nabla\times(\Delta \nabla\times\boldsymbol{X}(\Omega)) =\nabla\times \boldsymbol{H}^{-1}(\Omega)
		\end{equation*}
		Given data in $\boldsymbol{H}^{-2}(\Omega) \cap \ker(\nabla\cdot)$ that is not explicitly expressed as $\nabla \times \boldsymbol{H}^{-1}(\Omega)$, the corresponding formulation is the quad-curl problem \eqref{primalForm}:
		\begin{equation*}
			(\nabla\times)^4((\boldsymbol{H}_0(\curl^2;\Omega)\cap\boldsymbol{H}^0(\div;\Omega))=-\nabla \times \Delta\nabla\times \boldsymbol{X}(\Omega)=\boldsymbol{X}'(\Omega) =\boldsymbol{H}^{-2}(\Omega)\cap\ker(\nabla\cdot)
		\end{equation*}
		Furthermore, by incorporating a non-zero divergence in \eqref{primalForm},  we arrive at the extended decomposition:
		\begin{equation*}
			(\nabla\times)^4(\boldsymbol{H}_0(\curl^2;\Omega))\oplus\nabla H_0^1(\Omega)=-\nabla \times \Delta\nabla\times \boldsymbol{V}_0(\Omega)\oplus\nabla H_0^1(\Omega)=\boldsymbol{V}'(\Omega).
		\end{equation*}
	\end{remark}
	\begin{remark}
		The above analysis can be generalized to domains with non-trivial topology (meaning that $\Omega$ is not necessarily simply connected, and its boundary $\Gamma$ is not necessarily connected). In such settings, the divergence-free solutions $\boldsymbol{\psi}$ and $\boldsymbol{u}$ are determined only up to the finite-dimensional harmonic function spaces arising from the topology. For further details and related numerical treatments, we refer to \cite{Amrouche1998,BrennerCavanaughSung2024}.
	\end{remark}
	
	\subsection{Non-homogeneous boundary conditions}
	This subsection further considers the vector potential problem of the Stokes problem with non-homogeneous boundary conditions, which will help us to construct the $\boldsymbol{H}(\operatorname{grad-curl})$-conforming virtual element space to satisfy the associated discrete complex.
	\par 
	Let $\bu^{\text{non}}$ be the unique solution of the Stokes problem \eqref{StokesProblem} with non-homogeneous boundary conditions 
	\begin{equation}\label{NonhomCond}
		\bu^{\text{non}}= \boldsymbol{g} \quad\text{on } \boldsymbol{H}^{\frac{1}{2}}(\Gamma),
	\end{equation}
	which satisfies the compatibility relation $$\int_{\Gamma} \boldsymbol{g} \cdot \boldsymbol{n}\d S=\int_\Omega \nabla\cdot \bu^{\text{non}} \d \Omega= 0.$$
	We now present the definition of the grad-curl problem with non-homogeneous boundary conditions and nonzero divergence. 
	\begin{definition}
		Given $j_d \in H^{-1}(\Omega)$, $\boldsymbol{f}\in \boldsymbol{H}^{-1}(\Omega)$, and $(\boldsymbol{h},\boldsymbol{g})\in\boldsymbol{Y}(\Gamma)$, find  $\boldsymbol{\psi}^{\text{non}} \in \boldsymbol{V}(\Omega)$ with $(\boldsymbol{\gamma}_{\tau}(\boldsymbol{\psi}^{\text{non}}),\boldsymbol{\gamma}_{\curl}(\boldsymbol{\psi}^{\text{non}}))=(\boldsymbol{h},\boldsymbol{g})$ such that 
		\begin{equation}\label{NonhomGradcurl}
			\begin{aligned}
				\nu(\nabla\nabla\times \boldsymbol{\psi}^{\text{non}},\nabla\nabla\times\boldsymbol{\phi})&=(\boldsymbol{f},\nabla\times\boldsymbol{\phi)},\quad \forall \boldsymbol{\phi}\in\boldsymbol{V}_0(\Omega),\\  (\boldsymbol{\psi}^{\text{non}},\nabla q)&=(j_d, q), \quad \forall q\in H_0^1(\Omega).
			\end{aligned}
		\end{equation}
	\end{definition}
	\begin{theorem}\label{NonHomThem}
		Assume that $\Omega$ is a contractible Lipschitz domain.
		The grad-curl problem \eqref{NonhomGradcurl}
		has a unique solution $\boldsymbol{\psi}^{\text{non}} \in \boldsymbol{V}(\Omega)$ satisfying
		\begin{equation}\label{NonhomEsti}
			\|\boldsymbol{\psi}^{\text{non}}\|_{\boldsymbol{V}(\Omega)}\lesssim \|j_d\|_{-1}+ \frac{1}{\nu}\|\boldsymbol{f}\|_{-1} +\|\boldsymbol{h}\|_{-\frac{1}{2},\Gamma} +\|\boldsymbol{g}\|_{\frac{1}{2},\Gamma}.
		\end{equation}
		Furthermore, $\nabla\times \boldsymbol{\psi}^{\text{non}}$ is the unique solution of the Stokes problem \eqref{StokesProblem} with the non-homogeneous boundary conditions \eqref{NonhomCond}.
	\end{theorem}	
	\begin{proof}
		By Theorem \ref{SurjForV}, there exists $\boldsymbol{\psi}^{\partial}\in \boldsymbol{V}(\Omega)$ with
		$\boldsymbol{\gamma}_{\tau}(\boldsymbol{\psi}^{\partial})=\boldsymbol{h}$, $ \boldsymbol{\gamma}_{\curl}(\boldsymbol{\psi}^{\partial})=\boldsymbol{g},$ 
		and
		\begin{equation}
			\|\boldsymbol{\psi}^{\partial}\|_{\boldsymbol{V}(\Omega)}\lesssim \|\boldsymbol{h}\|_{-\frac{1}{2},\Gamma} + \|\boldsymbol{g}\|_{\frac{1}{2},\Gamma}.
		\end{equation}
		Then the problem \eqref{NonhomGradcurl} is equivalent to finding $\boldsymbol{\psi}^0 \in \boldsymbol{X}(\Omega)$ and $\varphi \in H_0^1(\Omega)$ such that:
		\begin{equation}\label{problem1}
			(\nabla\nabla\times \boldsymbol{\psi}^0,\nabla\nabla\times \boldsymbol{\phi})= \frac{1}{\nu}(\boldsymbol{f},\nabla\times\boldsymbol{\phi)}-(\nabla\nabla\times \boldsymbol{\psi}^{\partial},\nabla\nabla\times\boldsymbol{\phi}),\quad \forall\boldsymbol{\phi}\in\boldsymbol{X}(\Omega),
		\end{equation}
		and
		\begin{equation}\label{problem2}
			(\nabla \varphi ,\nabla q)=(j_d ,q)-(\boldsymbol{\psi}^{\partial},\nabla q),\quad \forall q\in H_0^1(\Omega).
		\end{equation}
		The  Friedrichs inequality \eqref{FreIneq} and the Poincaré inequality imply that
		\begin{equation*}
			\|\boldsymbol{\psi}^0\|_{\boldsymbol{V}(\Omega})\lesssim |\nabla\times \boldsymbol{\psi}^0|_{1} \lesssim \frac{1}{\nu}\|\boldsymbol{f}\|_{-1} + |\nabla\times\boldsymbol{\psi}^{\partial}|_{1},
		\end{equation*}
		and 
		\begin{equation*}
			\|\nabla \phi\|_{\boldsymbol{V}(\Omega)}=\|\nabla\phi\|\lesssim\|j_d\|_{-1} + \|\boldsymbol{\psi}^{\partial}\|.
		\end{equation*}
		In conclusion,  $ \boldsymbol{\psi}^{\text{non}}:=\boldsymbol{\psi}^0+\nabla\varphi +\boldsymbol{\psi}^{\partial}$ is the unique solution of \eqref{NonhomGradcurl} and satisfies \eqref{NonhomEsti}.
		\par
		By analogy with the homogeneous case, the vector potential representation of the Stokes problem \eqref{StokesProblem} generalizes to non-homogeneous boundary conditions \eqref{NonhomCond}, leading to the relation:
		\begin{align*}
			\nu(\nabla\nabla\times \boldsymbol{\psi}^{\text{non}},\nabla\nabla\times \boldsymbol{\phi}) = \nu\langle-\Delta\nabla\times \boldsymbol{\psi}^{\text{non}},\nabla\times\boldsymbol{\phi}\rangle\\=(\boldsymbol{f},\nabla\times\boldsymbol{\phi})=\nu(\nabla\boldsymbol{u},\nabla\nabla\times\boldsymbol{\phi}),\quad \forall\boldsymbol{\phi}\in \boldsymbol{X}(\Omega).
		\end{align*}
		Since $\nabla\times\boldsymbol{X}(\Omega)=\boldsymbol{H}_0^1(\Omega)\cap\operatorname{ker}(\nabla\cdot)$, we have 
		\begin{equation*}
			\nabla\times \boldsymbol{\psi}^{\text{non}}= \bu^{\text{non}} \text{ in } \Omega \text{ and } \nabla\times \boldsymbol{\psi}^{\text{non}} = \boldsymbol{g} =\bu^{\text{non}} \text{ on } \Gamma.
		\end{equation*}
		The proof is complete.
	\end{proof}
	\begin{remark}\label{NonHomEqu}
		The grad-curl problem \eqref{NonhomGradcurl} is equivalent to the
		following $-$curl$\Delta$curl problem up to the right-hand side:
		find $\boldsymbol{\psi}^{\text{non}}\in \boldsymbol{V}(\Omega)$ such that 
		\begin{equation}\label{curlDeltacurlProblem}
			\begin{aligned}
				-\nabla \times \Delta (\nabla \times \boldsymbol{\psi}^{\text{non}}) &=\boldsymbol{j} \text { in } \Omega, \\
				\nabla \cdot \boldsymbol{\psi}^{\text{non}} &=j_d \text { in } \Omega, \\
				\boldsymbol{n}\times (\boldsymbol{\psi}^{\text{non}}\times \boldsymbol{n})&=\boldsymbol{h} \text { on } \Gamma, \\
				(\nabla \times\boldsymbol{\psi}^{\text{non}})  &=\boldsymbol{g} \text { on } \Gamma,
			\end{aligned}
		\end{equation}
		where $\boldsymbol{j}\in \boldsymbol{H}^{-2}(\Omega)\cap\ker(\nabla\cdot)$, $j_d\in H^{-1}(\Omega)$, and  $(\boldsymbol{h},\boldsymbol{g})\in \boldsymbol{Y}(\Gamma)$.
		Moreover, when $\boldsymbol{j} = \frac{1}{\nu} \nabla \times \boldsymbol{f}$, the $-$curl$\Delta$curl problem \eqref{curlDeltacurlProblem} coincides with the grad-curl problem \eqref{NonhomGradcurl}.
	\end{remark}
	\begin{remark}
		One can also prescribe the following non-homogeneous boundary conditions in the quad-curl problem \eqref{primalForm} to obtain an equivalent grad-curl problem \eqref{NonhomGradcurl}:
		\[
		\boldsymbol{\gamma}_{\tau}(\boldsymbol{\psi}^{\mathrm{non}}) = \boldsymbol{h} \text{ with } \nabla_{\Gamma}\times \boldsymbol{h}= \boldsymbol{g} \cdot \boldsymbol{n}  \text{ and } \nabla \times \boldsymbol{\psi}^{\mathrm{non}} \times \boldsymbol{n} = \boldsymbol{g} \times \boldsymbol{n} \text{ on } \Gamma.
		\]
	\end{remark}
	
	\section{Virtual element spaces}
	Starting from the continuous Stokes complex
	\begin{equation}\label{ContiComplex}
		\mathbb{R} \stackrel{\subset}{\longrightarrow} H^{1}(\Omega) \stackrel{\nabla}{\longrightarrow} \boldsymbol{V}(\Omega) \stackrel{\nabla \times}{\longrightarrow}\boldsymbol{H}^1(\Omega) \stackrel{\nabla \cdot }{\longrightarrow} L^2(\Omega) \longrightarrow 0,
	\end{equation}
	we construct a conforming discrete subcomplex for any integers $r,k\ge 1$:
	\begin{equation}\label{DisComplex}
		\mathbb{R} \stackrel{\subset}{\longrightarrow} U_r(\Omega)\stackrel{\nabla}{\longrightarrow} \boldsymbol{V}_{r-1,k+1}(\Omega) \stackrel{\nabla \times }{\longrightarrow} \boldsymbol{W}_{k}(\Omega) \stackrel{\nabla \cdot }{\longrightarrow} Q_{k-1}(\Omega)\longrightarrow 0.
	\end{equation}
	Here,  the $\boldsymbol{H}(\operatorname{grad-curl})$-conforming virtual element space $\boldsymbol{V}_{r-1,k+1}(\Omega) \subset \boldsymbol{V}(\Omega) $ will be carefully designed to satisfy $$\nabla\times \boldsymbol{V}_{r-1,k+1}(\Omega) = \boldsymbol{W}_k(\Omega)\cap\ker(\nabla\cdot).$$  
	We also establish a commutative diagram relating the continuous and discrete complexes and derive corresponding interpolation error estimates. \par 
	A key feature of the VEM is defining  suitable projection operators onto polynomial spaces. 
	On any given element or face $G$,  we define the $H^1$- seminorm projection $\Pi_k^{\nabla, G} : H^1(G) \to P_k(G)$ by
	\begin{equation*}
		\left\{
		\begin{aligned}
			&\int_G\nabla \Pi_k^{\nabla, G}v\cdot\nabla p_k \d G= \int_G \nabla v\cdot\nabla p_k \d G,  \forall p_k \in P_k(G),  \\
			&\int_{\partial G} \Pi_k^{\nabla, G}v \d S=\int_{\partial G} v \d S. 
		\end{aligned}
		\right.
	\end{equation*}
	The $L^2$- projection $\Pi_k^{0, G}$: $L^2(G) \to P_k(G)$ is given by
	\begin{equation*}
		\int_G\Pi^{0, G}_k v p_k \d G=\int_G v  p_k \d G,  \forall p_k\in P_k(G).
	\end{equation*}
	These projections can be naturally extended to vector fields as $\boldsymbol{\Pi}_{k}^{\nabla,G}: \boldsymbol{H}^1(G;\mathbb{R}^d)\to \boldsymbol{P}_k(G;\mathbb{R}^d)$ and $\boldsymbol{\Pi}_{k}^{0,G}: \boldsymbol{L}^2(G;\mathbb{R}^d)\to \boldsymbol{P}_k(G;\mathbb{R}^d)$.
	For vector polynomial  spaces,  we introduce the following useful decompositions \cite{VEMforMaxwell, VEMfordivcurl}:
	\begin{align}
		\boldsymbol{P}_k(K)&=\nabla P_{k+1}(K)\oplus \left(\boldsymbol{x}\times \boldsymbol{P}_{k-1}(K)\right),  \label{Pkdc1}\\
		\boldsymbol{P}_k(K)&=\nabla \times \boldsymbol{P}_{k+1}(K) \oplus \boldsymbol{x}P_{k-1}(K),  \label{Pkdc2}\\
		\boldsymbol{P}_k(f)& = \overrightarrow{\nabla}_{f}\times P_{k+1}(f) \oplus \boldsymbol{x}P_{k-1}(f)\label{Pkdcf}.
	\end{align}
	For each face $f \in \partial K$, the two-dimensional tangential component of a three-dimensional vector field $\boldsymbol{v}$ on $f$ is defined as $$\boldsymbol{v}_{\tau} = \boldsymbol{v}_f|_f,$$ where $\boldsymbol{v}_f := \boldsymbol{n}_f \times (\boldsymbol{v} \times \boldsymbol{n}_f)$.
	\subsection{Scalar $H^1$-conforming virtual element space}
	For each face $f \in \partial K$, we recall the face space \cite{VEM2013} with $r\ge 2$, $l\ge 1$ 
	\begin{equation}\label{Bf}
		\mathbb{B}_{l,r}(f):=\left\{v \in H^1(f): \Delta v \in P_{r-1}(f),  v|_e \in P_{l}(e), \forall e\in \partial f, v|_{\partial f} \in C^0(\partial f)\right\},
	\end{equation}
	equipped with the degrees of freedom
	\begin{flalign}
		&\	\bullet \mathbf{D}^1_{\mathbb{B}}:  \text{the values of } v \text{ at the vertices of } f, \label{DB1}&\\
		&\	\bullet \mathbf{D}^2_{\mathbb{B}}:  \text{the values of } v \text{ at }l-1 \text{ distinct points of } e, \label{DB2}&\\
		&\	\bullet \mathbf{D}^3_{\mathbb{B}}:  \text{the face moments } \frac{1}{|f|}\int_f \nabla v \cdot \boldsymbol{x}_f p_{r-1} \d f,  \forall p_{r-1} \in P_{r-1}(f) , \label{DB3}
	\end{flalign}
	For $r=1$,  we introduce the serendipity space from \cite{VEMforlowMaxwell}
	\begin{equation}\label{SerBf}
		\begin{aligned}
			\mathbb{B}_{l,1}(f):=\left\{v \in H^1(f): \Delta v \in P_0(f), v|_e \in P_{l}(e), \forall e\in \partial f, v|_{\partial f} \in C^0(\partial f),\right. \\
			\left. \int_f\nabla v \cdot \boldsymbol{x}_f \d f=0 \right\},
		\end{aligned}
	\end{equation}
	where $\boldsymbol{x}_f:= \boldsymbol{x}-\boldsymbol{b}_f$ satisfies $\int_f \boldsymbol{x}_f \d f=0$, thus ensuring $P_1(f)\subset \mathbb{B}_{1,l}(f)$.
	This serendipity space is viewed as the subspace of the primal space \eqref{Bf} with $r=1$, where the degrees of freedom \eqref{DB3} vanish.  
	\par
	On  the polyhedron $K$, the local space for $r\ge 1$ is defined as
	\begin{equation*}
		U_r(K):=\left\{v \in H^{1}(K) : \Delta v \in P_{r-2}(K), v|_f \in \mathbb{B}_{r,r}(f),\forall f \in \partial K, v|_{\partial K} \in C^{0}(\partial K)   \right\}
	\end{equation*}
	with the degrees of freedom
	\begin{flalign}
		&\	\bullet \mathbf{D}^1_{U}:  \text{the values of } v \text{ at the vertices of } K, \label{DU1}&\\
		&\	\bullet \mathbf{D}^2_{U}:  \text{the values of } v \text{ at }r-1 \text{ distinct points of } e, \label{DU2}&\\
		&\	\bullet \mathbf{D}^3_{U}:  \text{the face moments } \frac{1}{|f|}\int_f \nabla v \cdot \boldsymbol{x}_f p_{\widehat{r}-1} \d f,  \forall p_{\widehat{r}-1} \in P_{\widehat{r}-1}(f) , \label{DU3}&\\ 
		&\	\bullet \mathbf{D}^4_{U}:  \text{the volume moments } \frac{1}{|K|}\int_K v p_{r-2} \d f,  \forall p_{r-2} \in P_{r-2}(K).&	\label{DU4}
	\end{flalign}
	Here, $\widehat{r}(r)$ is defined by
	\begin{equation}\label{rwidehat}
		\widehat{r}(r) = 
		\begin{cases}
			0, & \text{if } r = 1, \\
			r, & \text{if } r \ge 2.
		\end{cases}
	\end{equation}
	The dimension of $U_{r}(K)$ is given by
	\begin{equation*}
		\dim(U_r(K))=N_v+(r-1)N_e+\dim(P_{\widehat{r}-1}(f))N_f+ \dim(P_{r-2}(K)).
	\end{equation*}
	The global $H^1$-conforming virtual element space is obtained by gluing local spaces:
	\begin{equation*}
		U_{r}(\Omega):=\{ v \in H^1( \Omega):v|_K \in U_{r}(K), \forall K \in \mathcal{T}_h\}.
	\end{equation*}
	\subsection{Vector $\boldsymbol{H}^1$-conforming virtual element space}
	We introduce two face enhancement spaces using the computable $H^1$-projection $\Pi^{\nabla}_{k,f}$ \cite{VEM2014} with $k\ge 1$:
	\begin{equation*}
		\begin{aligned}
			\widehat{\mathbb{B}}_{k}(f)&:=\left\{w \in\mathbb{B}_{k+2,k}(f) 
			: (w-\Pi_{k}^{\nabla, f} w, \widehat{p}_{k+1})_ f=0, \forall \widehat{p}_{k+1} \in P_{k+1}(f)/ P_{k-2}(f)\right\}, \\
			\overline{\mathbb{B}}_{k}(f)&:=\left\{w \in\mathbb{B}_{k+2,k}(f) 
			: (w-\Pi_{k}^{\nabla, f} w , \widehat{p}_{k+1})_f=0, \forall \widehat{p}_{k+1} \in P_{k+1}(f)/ P_{m}(f)\right\},
		\end{aligned}
	\end{equation*}
	where $m=\text{max}(0,k-2)$. Note that the ``super-enhanced'' constraints proposed  by \cite{VEMforStokes}, originally used to compute the \( \boldsymbol{L}^2 \)-projection in \( \boldsymbol{W}_k(K) \), are extended in this work to compute the \( \boldsymbol{L}^2 \)-projection in \( \boldsymbol{V}_{r-1,k+1}(K) \) (see Proposition \ref{L2prop}).
	The boundary space is defined as
	\begin{equation*}
		\boldsymbol{\mathbb{B}}_k(\partial K):=\{\boldsymbol{w}\in \boldsymbol{C}^0(\partial K): \boldsymbol{w}_{\tau}\in[\widehat{\mathbb{B}}_{k}(f)]^2,\boldsymbol{w} \cdot \boldsymbol{n}_f|_f\in \overline{\mathbb{B}}_{k}(f) ,\forall f \in \partial K \},
	\end{equation*}
	and its dimension is given by
	\begin{equation}\label{dimBk}
		\begin{aligned}
			\text{dim}\mathbb{B}_k(\partial K)
			=3N_v + 3(k-1)N_e+ (2\text{dim}(P_{k-2}(f))+\text{dim}(P_m(f)))N_f.
		\end{aligned}
	\end{equation}
	\par 
	The local vector $\boldsymbol{H}^1$-conforming  virtual element space \cite{VEMforMHD,VEMforStokes} on the polyhedron $K$ is defined as the enlarged space: 
	\begin{equation}\label{EnlargedW}
		\begin{aligned}
			\widetilde{\boldsymbol{W}}_{k}(K):= \bigg\{ \boldsymbol{w} \in
			\boldsymbol{H}^1(K): 
			&\w |_{\partial K} \in \boldsymbol{\mathbb{B}}_k(\partial K),\\
			&\left\{ \begin{aligned}
				&-\Delta \boldsymbol{w} + \nabla q \in \boldsymbol{x}_K \times \boldsymbol{P}_{k-1}(K)  \\
				&\quad \quad   \quad\text{ for some } q \in L^2_0(K), \\
				&\nabla \cdot  \boldsymbol{w} \in P_{k-1}(K)
			\end{aligned} \right. 
			\bigg\},
		\end{aligned}
	\end{equation}
	and the restricted space:
	\begin{equation}\label{WFinal}
		\begin{aligned}
			\boldsymbol{W}_k(K):=\{
			\boldsymbol{w}\in \widetilde{\boldsymbol{W}}_k(K): 
			&(\boldsymbol{w}-\boldsymbol{\Pi}_{k}^{\nabla, K}  \boldsymbol{w}, \boldsymbol{x}_K \times \widehat{\boldsymbol{p}}_{k-1})_K=0, \\
			&\quad \quad   \forall \widehat{\boldsymbol{p}}_{k-1} \in \boldsymbol{P}_{k-1}(K)/\boldsymbol{P}_{k-3}(K)\},
		\end{aligned}
	\end{equation}
	where $\boldsymbol{x}_K:=\boldsymbol{x}-\boldsymbol{b}_K$. In $\boldsymbol{W}_k(K)$ we have the degrees of freedom
	\begin{flalign}
		&\	\bullet \mathbf{D}^1_{\boldsymbol{W}}:  \text{the values of } \boldsymbol{w} \text{ at the vertices of } K, \label{DW1}&\\
		&\	\bullet \mathbf{D}^2_{\boldsymbol{W}}:  \text{the values of } \boldsymbol{w} \text{ at }k-1 \text{ distinct points of } e, \label{DW2}&\\
		&\	\bullet \mathbf{D}^3_{\boldsymbol{W}}:  \text{the  tangential face moments } 	\int_{f} \boldsymbol{w}_{\tau} \cdot \boldsymbol{p}_{k-2} \, \d f, \forall  \boldsymbol{p}_{k-2}\in \boldsymbol{P}_{k-2}(f), \label{DW3}&\\ 
		&\	\bullet \mathbf{D}^4_{\boldsymbol{W}}:  \text{the  normal face moments } 	\int_{f} \boldsymbol{w}\cdot \boldsymbol{n}_f p_{m}(f) \d f, \forall p_{m}\in P_{m}(f), \label{DW4}&\\ 
		&\	\bullet \mathbf{D}^5_{\boldsymbol{W}}:  \text{the volume moments } \int_{K} \boldsymbol{w}\cdot \boldsymbol{x}_K\times \boldsymbol{p}_{k-3} \d K, \forall p_{k-3} \in \boldsymbol{P}_{k-3}(K),&	\label{DW5}&\\
		&\	\bullet \mathbf{D}^6_{\boldsymbol{W}}:  \text{the volume moments } \int_K (\nabla \cdot \boldsymbol{w})\widehat{p}_{k-1} \d K, \forall \widehat{p}_{k-1}\in P_{k-1}(K)/P_{0}(K).& \label{DW6}
	\end{flalign}
	\begin{remark}
		In the enlarged space $\widetilde{\boldsymbol{W}}_k(K)$, it is additionally equipped with the degrees of freedom
		\begin{equation*}
			\bullet \widetilde{\mathbf{D}}^5_{W}: \text{the volume moments} \int_K \boldsymbol{w}\cdot \boldsymbol{x}_K\times \widehat{\boldsymbol{p}}_{k-1} \d K, \forall \widehat{\boldsymbol{p}}_{k-1} \in \boldsymbol{P}_{k-1}(K)/\boldsymbol{P}_{k-3}(K).
		\end{equation*}
		Since the $\boldsymbol{H}^1$-seminorm projection $\boldsymbol{\Pi}_k^{\nabla,K}$ defiend in $\boldsymbol{W}_k(K)$ is computable by the degrees of freedom \eqref{DW1}-\eqref{DW6}, see \cite[Proposition 5.1]{VEMforStokes}, we recognize $\boldsymbol{W}_k(K)$ as the subspace of  $\widetilde{\boldsymbol{W}}_k(K)$ with the $\widetilde{\mathbf{D}}^5_{W}$ restricted by the projection $\boldsymbol{\Pi}_k^{\nabla,K}$.
	\end{remark}
	The dimension of $\boldsymbol{W}_k(K)$ is given by
	\begin{equation}\label{dWk}
		\dim(\boldsymbol{W}_k(K))= \dim(\mathbb{B}_k(\partial K)+ \dim (\boldsymbol{x}_K\times \boldsymbol{P}_{k-2}(K))+ \dim(P_{k-1}(K))-1.
	\end{equation}
	From the decomposition \eqref{Pkdc1}, we obtain
	\begin{equation}\label{dimxPk}
		\begin{aligned}
			\text{dim}(\boldsymbol{x}_K \times \boldsymbol{P}_{k-3}(K))&=
			\text{dim}(\boldsymbol{P}_{k-2}(K))-\text{dim}(\nabla P_{k-1}(K))\\
			&=3\text{dim}(P_{k-2}(K))-\text{dim}(P_{k-1}(K))+1,
		\end{aligned}
	\end{equation}
	which, combined with \eqref{dWk} and \eqref{dimBk}, yields
	\begin{equation}
		\begin{aligned}
			\dim(\boldsymbol{W}_k(K))&=3N_v + 3(k-1)N_e+ (2\text{dim}(P_{k-2}(f))+\text{dim}(P_m(f)))N_f\\
			&\quad + 3\text{dim}(P_{k-2}(K)).
		\end{aligned}
	\end{equation}
	\par
	At last,
	we also define the global space by
	\begin{eqnarray*}
		\boldsymbol{W}_k(\Omega):=\{ \boldsymbol{w}\in \boldsymbol{H}^1(\Omega):\, \boldsymbol{w}|_{K}\in \boldsymbol{W}_k(K),  \forall K \in \mathcal{T}_h\}.
	\end{eqnarray*}

	\subsection{$\boldsymbol{H}(\operatorname{grad-curl})$-conforming virtual element space}
	The $\boldsymbol{H}(\operatorname{grad-curl})$-conforming virtual element space $\boldsymbol{V}_{r-1,k+1}(\Omega)$ is constructed in this subsection.
	We start from the face $\boldsymbol{H}(\operatorname{curl-curl})$-conforming virtual element space with $r\ge 2$, $k\ge 1$ described in \cite{ZJQQuadCurl}:
	\begin{equation*}
		\begin{aligned}
			\mathbb{E}_{r-1,k+1}(f):=\left\{\boldsymbol{v} \in \boldsymbol{L}^2(f): \nabla \cdot  \boldsymbol{v}\in P_{r-1}(f) ,  \nabla_{f}\times \boldsymbol{v} \in \overline{\mathbb{B}}_k(f), \right.\\
			\left.
			\left(\boldsymbol{v} \cdot \boldsymbol{t}_{\partial f}\right)|_e \in P_{r-1}(e),  \forall e\in \partial f
			\right\},
		\end{aligned}
	\end{equation*}
	and the corresponding Serendipity space for $r=1$:
	\begin{equation}\label{reducedE}
		\begin{aligned}
			\mathbb{E}_{0,k+1}(f):=\left\{\boldsymbol{v} \in \boldsymbol{L}^2(f): \nabla \cdot  \boldsymbol{v}\in P_{0}(f) ,  \nabla_{f}\times \boldsymbol{v} \in \overline{\mathbb{B}}_k(f),\left(\boldsymbol{v} \cdot \boldsymbol{t}_{\partial f}\right)|_e \in P_{0}(e), \right.\\
			\left.
			\forall e\in \partial f,\int_f \v \cdot \boldsymbol{x}_fp_0 =0, \forall p_0\in P_0(f)
			\right\}.
		\end{aligned}
	\end{equation}
	Here, the face integral constraints from \cite{VEMforlowMaxwell} are extended to define \eqref{reducedE} so that the inclusion $\nabla \mathbb{B}_{1,1}(f)\subset\mathbb{E}_{0,k+1}(f)$ holds. The associated degrees of freedom will be specified later in Proposition~\ref{DofProp}, 
	and will include the normal components of~\eqref{dofV1} and~\eqref{dofV2}, together with those in~\eqref{dofV3}, \eqref{dofV5}, and~\eqref{dofV6}.
	Gluing the face on the boundary yields 
	\begin{equation}
		\begin{aligned}\label{boundary space}
			\mathbb{E}_{r-1,k+1 }(\partial K):=\left\{\v \in \boldsymbol{L}^2(\partial K) : \v_{\tau}
			\in \mathbb{E}_{r-1,k+1}(f), \forall f \in \partial K, \right.\\
			\left. \v \cdot \boldsymbol{n}_{\partial K} =0 \text{ almost everywhere on }\partial K, \right. \\
			\left. \v_{f_1} \cdot \boldsymbol{t}_e =\v_{f_2}\cdot \boldsymbol{t}_e \text{ and } \nabla_{f_1}\times\v_{\tau}=\nabla_{f_2}\times\v_{\tau} \right.\\
			\left. \text{ on each edge } e\in \partial f_1\cap \partial f_2, \,  f_1,f_2 \subset \partial K
			\right\}.
		\end{aligned}
	\end{equation}
	Apparently, the continuity requirement defined in the tangential space \eqref{boundary space} ensures that $\nabla_{\partial K} \times \mathbb{E}_{r-1,k+1}(\partial K) \subset H^1(\partial K)$. 
	Furthermore, we have the following compatibility relation:
	\begin{equation}\label{compatibility}
		\nabla_{\partial K} \times \mathbb{E}_{r-1,k+1}(\partial K)= \mathbb{B}_k(\partial K) \cdot \boldsymbol{n}_{\partial K},
	\end{equation}
	which implies that $\mathbb{E}_{r-1,k+1}(\partial K) \times \mathbb{B}_k(\partial K)\subset \boldsymbol{Y}(\partial K).$  The dimension of the space $\mathbb{E}_{r-1,k+1}(\partial K)$ is
	\begin{equation}\label{dimEk}
		\begin{aligned}
			\text{dim}(\mathbb{E}_{r-1,k+1}(\partial K))&=\dim(P_{r-1}(e))N_e+\text{dim}(P_{\widehat{r}-1}(f))N_f+\dim(\mathbb{B}_k(\partial K)\cdot\boldsymbol{n}_{\partial K})-N_f,\\
			&=N_v+(k-1+\dim(P_{r-1}(e))N_e+(\text{dim}(P_{\widehat{r}-1}(f))+\dim(P_m(f))-1)N_f,
		\end{aligned}
	\end{equation}
	where the term $-N_f,$ follows from the compatibility identity:
	\begin{equation*}
		\int_f \nabla_{f}\times \boldsymbol{v} = \int_{\partial f} \boldsymbol{v}\cdot \boldsymbol{t}_{\partial f} \d S.
	\end{equation*}
	On the element $K$, we first define the enlarged space 
	\begin{equation}\label{EnlargedV}
		\begin{aligned}
			\widetilde{\boldsymbol{V}}_{r-1,k+1}(K):=\bigg\{\boldsymbol{v} \in \boldsymbol{V}(K): 
			&\left\{
			\begin{aligned}
				&(\nabla\times \boldsymbol{v})|_{\partial K}\in \boldsymbol{\mathbb{B}}_k(\partial K),\\
				&\boldsymbol{n}_{\partial K}\times (\boldsymbol{v}\times \boldsymbol{n}_{\partial K})|_{ \partial K} \in \mathbb{E}_{r-1,k+1}(\partial K),
			\end{aligned}
			\right.\\
			&\left\{ \begin{aligned}
				&\nabla \times \Delta (\nabla\times  \boldsymbol{v}) \in \nabla \times (\boldsymbol{x}_K \times \boldsymbol{P}_{k-1}(K)) , \\
				&\nabla \cdot  \boldsymbol{v}\in P_{r-2}(K)
			\end{aligned} \right. 
			\bigg\},
		\end{aligned}
	\end{equation}
	whose well-posedness follows directly from that of the $-$curl$\Delta$curl problem \eqref{curlDeltacurlProblem} by given data on the polyhedron $K$:
	\begin{align*}
		\boldsymbol{j}=\nabla\times \boldsymbol{f}\in \nabla\times(\boldsymbol{x}_K\times\boldsymbol{P}_{k-1}(K)), \quad j_d \in P_{r-2}(K), \\ 
		(\boldsymbol{h},\boldsymbol{g})\in \mathbb{E}_{r-1,k+1}(\partial K)\times \mathbb{B}_k(\partial K)\subset \boldsymbol{Y}(\partial K).
	\end{align*}
	Combining the dimensional results from \eqref{dimBk}, \eqref{dimEk}, and \eqref{dimxPk}, we obtain the dimension of $\widetilde{\boldsymbol{V}}_{r-1,k+1}(K)$ as:
	\begin{equation}\label{dim}
		\begin{aligned}
			\text{dim} (\widetilde{\boldsymbol{V}}_{r-1,k+1}(K) )
			&=
			\text{dim}(\mathbb{B}_k(\partial K))+ \text{dim}(\mathbb{E}_{r-1,k+1}(\partial K)) -\text{dim}(\mathbb{B}_k(\partial K)\cdot\boldsymbol{n}_{\partial K})\\
			&\quad + \text{dim}(\nabla\times(\boldsymbol{x}_K \times \boldsymbol{P}_{k-1}(K)))+\text{dim}(P_{r-2}(K))\\
			&= \dim(\mathbb{B}_k(\partial K))+ \dim(P_{r-1}(e))N_e+\dim(P_{\widehat{r}-1}(f)-1)N_f \\
			&\quad+ \dim(\boldsymbol{x}_K\times\boldsymbol{P}_{k-1}(K))+\dim(P_{r-2}(K))  \\
			&=3 N_v+(3 k-3+ \text{dim}(P_{r-1}(e))) N_{e}\\
			&\quad+\left(2 \text{dim}P_{k-2}(f)+\text{dim}P_{m}(f)+\text{dim}P_{\widehat{r}-1}(f)-1\right)N_{f}\\
			&\quad +3\text{dim}(P_k(K))-\text{dim}(P_{k+1}(K))+1+\text{dim}P_{r-2}(K),
		\end{aligned}
	\end{equation}
	where the term $-\text{dim}(\mathbb{B}_k(\partial K)\cdot\boldsymbol{n}_{\partial K})$ ensures that the compatibility condition \eqref{compatibility} holds, and the curl operator restricted to $\boldsymbol{x}_K\times \boldsymbol{P}_{k-1}(K)$ is an isomorphism.
	\par
	We define the final restricted space 
	\begin{equation}\label{FinalSpace}
		\begin{aligned}
			\boldsymbol{V}_{r-1,k+1}(K):=\{\v \in \widetilde{\boldsymbol{V}}_{r-1,k+1}(K): &(\nabla \times \boldsymbol{v}-\boldsymbol{\Pi}_{k}^{\nabla, K} \nabla \times \boldsymbol{v}, \boldsymbol{x}_K \times \widehat{\boldsymbol{p}}_{k-1})_K=0, \\
			&\quad \quad \quad  \forall \widehat{\boldsymbol{p}}_{k-1} \in \boldsymbol{P}_{k-1}(K)/\boldsymbol{P}_{k-3}(K)\}.
		\end{aligned}
	\end{equation}
	To satisfy $\nabla\times \boldsymbol{V}_{r-1,k+1}(K)\subset \boldsymbol{W}_k(K)$, the restriction in \eqref{FinalSpace} is designed as the curl analogue of \eqref{WFinal}.
	\begin{proposition}\label{DofProp}
		The dimension of $\boldsymbol{V}_{r-1,k+1}(K)$ is given by
		\begin{align*}
			\text{dim} (\boldsymbol{V}_{r-1,k+1}(K) )
			&=3 N_v+(3 k-3+ \text{dim}(P_{r-1}(e))) N_{e}\\
			&\quad+\left(2 \text{dim}(P_{k-2}(f))+\text{dim}(P_{m}(f))+\text{dim}(P_{\widehat{r}-1}(f))-1\right)N_{f}\\
			&\quad +3\text{dim}(P_{k-2}(K))-\text{dim}(P_{k-1}(K))+1+\text{dim}(P_{r-2}(K)),
		\end{align*}
		where $m=\max(0,k-2)$ and $\widehat{r}$ is defined in \eqref{rwidehat}.
		Furthermore,
		the following degrees of freedom
		\begin{flalign}\label{dofV1}
			&\	\bullet \mathbf{D}^1_{\boldsymbol{V}}: \text{the values of }  \nabla\times \boldsymbol{v} \text{ at the vertices of } K, &\\
			\label{dofV2}
			&\ \bullet \mathbf{D}^2_{\boldsymbol{V}}:\text{the values of } \nabla\times \boldsymbol{v} \text{ at k-1 distinct points of every edge of } K,&\\
			\label{dofV3}
			&\  \bullet \mathbf{D}^3_{\boldsymbol{V}}:  \text{the edge moments of } 
			\frac{1}{|e|} \int_{e} \boldsymbol{v} \cdot \boldsymbol{t}_e p_{r-1}\d e, \forall p_{r-1}\in P_{r-1}(e),&\\
			\label{dofV4}
			&\ \bullet \mathbf{D}^4_{\boldsymbol{V}}:\text{the  tangential face moments } \frac{1}{|f|}\int_f (\nabla\times \boldsymbol{v})_{\tau} \cdot \boldsymbol{p}_{k-2} \d f,  \forall \boldsymbol{p}_{k-2} \in \boldsymbol{P}_{k-2}(f), & \\
			\label{dofV5}
			&\ \bullet \mathbf{D}^5_{\boldsymbol{V}}: \text{the  face moments  } \frac{1}{|f|}\int_{f}  \boldsymbol{v}_{\tau} \cdot \overrightarrow{\nabla}_{f}\times p_{k-2} \d f, \forall p_{k-2} \in P_{k-2}(f), &\\
			\label{dofV6}
			&\ \bullet \mathbf{D}^6_{\boldsymbol{V}}: \text{the  face moments  } \frac{1}{|f|}\int_{f}  \boldsymbol{v}_{\tau}\cdot \boldsymbol{x}_f p_{\widehat{r}-1} \d f, \forall p_{\widehat{r}-1} \in P_{\widehat{r}-1}(f), &\\
			\label{dofV7}
			&\ \bullet \mathbf{D}^7_{\boldsymbol{V}}: \text{the volume moments } \int_{K} \nabla\times \boldsymbol{v}  \cdot\left(\boldsymbol{x}_K \times \boldsymbol{p}_{k-3}\right) \d K, \forall \boldsymbol{p}_{k-3} \in \boldsymbol{P}_{k-3}(K), &\\
			\label{dofV8}
			&\ \bullet \mathbf{D}^8_{\boldsymbol{V}}:\text{the volume moments } 
			\int_K \boldsymbol{v} \cdot \boldsymbol{x}_K p_{r-2} \d K, \forall  p_{r-2} \in 
			P_{r-2}(K)
		\end{flalign}
		are unisolvent in $\Vrk$.
	\end{proposition}
	\begin{proof}
		Using standard techniques for passing from enlarged to restricted virtual element spaces \cite{Ahmad2013},  it suffices to  show that  the degrees of freedoms \eqref{dofV1}-\eqref{dofV8}, along with the additional moments
		\begin{equation*}
			\widetilde{\mathbf{D}}^7_{\boldsymbol{V}}: \int_{K} \nabla\times \boldsymbol{v}  \cdot\left(\boldsymbol{x}_K \times \widehat{\boldsymbol{p}}_{k-1}\right) \d K, \, \forall \widehat{\boldsymbol{p}}_{k-1} \in \boldsymbol{P}_{k-1}(K)/\boldsymbol{P}_{k-3}(K)
		\end{equation*}
		are unisolvent in the enlarged space $\widetilde{\boldsymbol{V}}_{r-1,k+1}(K)$. 
		Note that $\widetilde{\mathbf{D}}^7_{\boldsymbol{V}} \cup  \mathbf{D}^7_{\boldsymbol{V}}$ are equivalent to prescribing the moments
		\begin{equation*}
			\int_{K} \nabla\times \boldsymbol{v}  \cdot\left(\boldsymbol{x}_K \times \boldsymbol{p}_{k-1}\right) \d K, \forall \boldsymbol{p}_{k-1} \in \boldsymbol{P}_{k-1}(K).
		\end{equation*}
		\par 
		The proof begins with a dimension count of the degrees of freedom:
		\begin{gather*}
			\# \mathbf{D}^1_{\boldsymbol{V}}= 3N_v,\quad \# \mathbf{D}^2_{\boldsymbol{V}}=3(k-1)N_e,\quad \# \mathbf{D}^3_{\boldsymbol{V}}=\text{dim}(P_{r-1}(e))N_e\\
			\# \mathbf{D}^4_{\boldsymbol{V}}=2\text{dim}(P_{k-2}(f))N_f,\quad \# \mathbf{D}^5_{\boldsymbol{V}}=(\text{dim}(P_{m}(f))-1)N_f,\quad \# \mathbf{D}^6_{\boldsymbol{V}}=\text{dim}(P_{\widehat{r}-1}(f))N_f,\\
			\# \mathbf{D}^7_{\boldsymbol{V}}+\# \widetilde{\mathbf{D}}^7_{\boldsymbol{V}}=3\text{dim}(P_{k}(K))-\text{dim}(P_{k+1}(K))+1,\quad \# \mathbf{D}^8_{\boldsymbol{V}}=\text{dim}(P_{r-2}(K)).
		\end{gather*}
		Observe that the above number of degrees of freedom equals the dimension of   $\widetilde{\boldsymbol{V}}_{r-1,k+1}(K)$ given in \eqref{dim}. 
		To establish unisolvence, we  show that if all the degrees of freedom \eqref{dofV1}-\eqref{dofV7} for a given $\v \in \widetilde{\boldsymbol{V}}_{r-1,k+1}(K)$ vanish,  then $\v$ must be identically zero. \par
		Imposing the vanishing conditions $\mathbf{D}^1_{\boldsymbol{V}}(\v) = \mathbf{D}^2_{\boldsymbol{V}}(\v) = \mathbf{D}^4_{\boldsymbol{V}}(\v) = 0$  yields $\mathbf{D}^1_{\boldsymbol{W}}(\nabla \times \v) = \mathbf{D}^2_{\boldsymbol{W}}(\nabla \times \v) = \mathbf{D}^3_{\boldsymbol{W}}(\nabla \times \v) = 0$.
		By setting $\mathbf{D}^3_{\boldsymbol{V}}(\v) = \mathbf{D}^5_{\boldsymbol{V}}(\v) = 0$, the face integral identity
		\begin{equation}\label{rotv0}
			\int_f \nabla \times \v \cdot \boldsymbol{n}_{f} q_m \d f = \int_f \v \cdot  \overrightarrow{\nabla}_{f}\times q_m \d f+\int_{\partial f} \v \cdot \boldsymbol{t}_{\partial f} q_m\d S=0,\quad \forall p_m\in P_m(K),
		\end{equation}
		implies that $\mathbf{D}^4_{\boldsymbol{W}}(\nabla\times \v)=0$.
		Then the unisolvence of the degrees of freedom \eqref{DW1}-\eqref{DW4}  for  $\widehat{\mathbb{B}}_k(K)$ and $\overline{\mathbb{B}}_{k}(f)$  leads to
		\begin{equation}\label{curlv0}
			\nabla\times \v =0 \quad \text{on } \partial K.
		\end{equation} 
		Recall the uniqueness of $\mathbb{E}_{r-1,k+1}(f)$ for $r\ge 2$ established in \cite{ZJQQuadCurl}, as well as the Serendipity property for $r=1$ introduced in \cite{VEMforlowMaxwell}. The identity \eqref{curlv0} and the vanishing degrees of freedom \eqref{dofV3} and \eqref{dofV6} imply that
		\begin{equation*}
			\nabla_{f}\times  \boldsymbol{v}_{\tau}=0,\quad \nabla\cdot \boldsymbol{v}_{\tau} =0,\quad  \boldsymbol{v}_{\tau}\cdot\boldsymbol{t}_{\partial f}=0 \quad \text{on } f \in \partial K,
		\end{equation*}
		which  further yields
		\begin{equation}\label{vt0}
			\boldsymbol{v}_{\tau}=0 \quad \text{on } f \in \partial K.
		\end{equation}
		On the element $K$,
		the homogeneous boundary conditions \eqref{curlv0} and \eqref{vt0} lead to 
		\begin{equation}
			\begin{aligned}
				|\nabla \times \boldsymbol{v}|_{1,K}^2&=\int_K \nabla  \nabla \times \boldsymbol{v}:\nabla  \nabla \times \boldsymbol{v} \d K\\
				&=-\int_K \nabla \times \boldsymbol{v} \cdot \Delta \nabla\times \v \d K+\int_{\partial K} \nabla \times\boldsymbol{v} \cdot \frac{\partial \nabla
					\times\boldsymbol{v}}{\partial \boldsymbol{n}_{\partial K}} \d S\\
				&=-\int_{K} \boldsymbol{v}\cdot  \nabla\times  \Delta \nabla\times \v \d K-\int_{\partial K} \boldsymbol{n}_{\partial K}\times  \boldsymbol{v} \cdot (  \Delta \nabla\times \v) \d S\\
				&=-\int_{K} \boldsymbol{v}\cdot  \nabla\times  \Delta \nabla\times \v \d K.
			\end{aligned}\label{halfnorm}.
		\end{equation}
		Combined  with the definition of \eqref{EnlargedV} and  $\mathbf{D}^{7}_{\boldsymbol{V}}(\boldsymbol{v})=\widetilde{\mathbf{D}}^7_{\boldsymbol{V}}(\boldsymbol{v})=0$, there exists $\boldsymbol{p}_{k-1}\in \boldsymbol{P}_{k-1}(K)$ such that 
		\begin{equation*}
			\begin{aligned}
				|\nabla \times \,\boldsymbol{v}|_{1,K}^2&=-\int_K \boldsymbol{v} \cdot \nabla\times (\boldsymbol{x}_K \times \boldsymbol{p}_{k-1}) \d K\\
				&=-\int_K  \nabla\times \boldsymbol{v}\cdot(\boldsymbol{x}_K \times \boldsymbol{p}_{k-1}) \d K=0.
			\end{aligned}
		\end{equation*}
		Then, applying the Poincaré inequality under the homogeneous boundary condition \eqref{curlv0} yields
		$$\nabla\times \v = 0 \quad \text{in } K.$$
		From the exactness of \eqref{StokesComplex},
		there exists a function $\phi \in H^1_0(K)$ such that $$\boldsymbol{v}= \nabla  \phi  \quad \text{in } K.$$ 
		Since $\nabla \cdot \v \in P_{r-2}(K)$ in $K$, we deduce that $\phi \in U_r(K)$ and that $\phi$ satisfies the vanishing degrees of freedom \eqref{DU1}-\eqref{DU3}. Finally, given $\mathbf{D}_{\boldsymbol{V}}^8(\v)=0$, we get
		\begin{equation*}
			\int_K \phi p_{r-2}\d K =\int_K \phi \nabla \cdot (\boldsymbol{x}_K p_{r-2}) \d K= -\int_K \nabla \phi \cdot (\boldsymbol{x}_K p_{r-2}) \d K =0, \
		\end{equation*}
		which implies $\mathbf{D}_{U}^4(\phi)=0$. Thus we have $\phi =0$, which completes the proof.
	\end{proof}
	\begin{remark}\label{PolyCompl}
		Unlike the $\boldsymbol{H}(\operatorname{grad-curl})$-conforming virtual element space of \cite{VEMforStokes}, based on the biharmonic problem \eqref{biharmonic problem}, our approach uses the $-\operatorname{curl}\Delta\operatorname{curl}$ problem \eqref{curlDeltacurlProblem}. This foundation naturally supports the divergence constraint $\nabla \cdot \boldsymbol{v} \in P_{r-2}(K)$, thus enabling the arbitrary-order polynomial approximation.
		Specifically, from the polynomial decomposition \eqref{Pkdc1}, it holds that
		\begin{equation*}
			\begin{aligned}
				\boldsymbol{P}_m(K) = \nabla P_{m+1}(K)\oplus\boldsymbol{x}\times \boldsymbol{P}_{m-1}(K) &\subset \nabla P_{r}(K)\oplus\boldsymbol{x}\times \boldsymbol{P}_{m-1}(K)  \subset \boldsymbol{V}_{r-1, k+1}(K),
			\end{aligned}
		\end{equation*}
		where $m$=min$\{r-1, k+1\}$. Then, taking $r=k$, $r=k+1$ and $r=k+2$, we have $\boldsymbol{P}_{k-1}(K)\subset\boldsymbol{V}_{k-1, k+1}(K)$, $\boldsymbol{P}_{k}(K)\subset\boldsymbol{V}_{k, k+1}(K)$ and $\boldsymbol{P}_{k+1}(K)\subset\boldsymbol{V}_{k+1, k+1}(K)$, respectively.
	\end{remark}
	\begin{proposition}\label{L2prop}
		The $L^2$- projection $ \boldsymbol{\Pi}_{l}^{0, K}: \boldsymbol{V}_{r-1, k+1}(K) \to \boldsymbol{P}_{l}(K)$ with $l=\min\{r-1,k-1\}$ is computable based on the degrees of freedom \eqref{dofV1}-\eqref{dofV8}.
	\end{proposition}
	\begin{proof}
		Using the polynomial decomposition \eqref{Pkdc2},  any  $\boldsymbol{q}_{l} \in \boldsymbol{P}_{l}(K)$ can be expressed as 
		\begin{align*}
			\boldsymbol{q}_{l}&=\nabla\times (\boldsymbol{x}_K\times \boldsymbol{p}_{l})+\boldsymbol{x}_Kq_{l-1}
			\\
			&=\nabla\times (\boldsymbol{x}_K\times\widehat{\boldsymbol{p}}_{l})+ \nabla \times (\boldsymbol{x}_K\times \boldsymbol{p}_{l-2})+\boldsymbol{x}_Kq_{l-1},
		\end{align*}
		where   polynomials $\boldsymbol{p}_{l}\in \boldsymbol{P}_{l}(K)$, $\widehat{\boldsymbol{p}}_{l} \in \boldsymbol{P}_{l}(K)/\boldsymbol{P}_{l-2}(K)$, $\boldsymbol{p}_{l-2} \in \boldsymbol{P}_{l-2}(K)$ and $ q_{l-1} \in P_{l-1}(K)$. 
		For any $\v\in\boldsymbol{V}_{r-1,k+1}(K)$, the projection $\boldsymbol{\Pi}_{l}^{0,K}$ satisfies
		\begin{equation}\label{ComL2Pro}
			\begin{aligned}
				\int_K\boldsymbol{\Pi}_{l}^{0,K} \boldsymbol{v} \cdot \boldsymbol{q}_{l} \d K &=\int_K \boldsymbol{v} \cdot \boldsymbol{q}_{l} \d K \\
				&=\int_K \boldsymbol{v} \cdot \nabla\times (\boldsymbol{x}_K\times\boldsymbol{p}_{l}) +\int_K \boldsymbol{v}\cdot \boldsymbol{x}_Kq_{l-1} \d K \\
				&=\int_K \nabla\times \v \cdot (\boldsymbol{x}_K\times\boldsymbol{p}_{l})\d K +\int_{\partial K}\boldsymbol{v } \times \boldsymbol{n}_{\partial K}\cdot(\boldsymbol{x}_K\times \boldsymbol{p}_{l}) \d S\\
				&\quad +\int_K \boldsymbol{v}\cdot \boldsymbol{x}_Kq_{l-1} \d K\\
				&=\int_K \boldsymbol{\Pi}_{l+1}^{\nabla,K}\nabla\times \boldsymbol{ v}\cdot(\boldsymbol{x}_K\times\widehat{\boldsymbol{p}}_{l})\d K+ \int_K \nabla\times \v \cdot (\boldsymbol{x}_K\times \boldsymbol{p}_{l-2})\d K \\
				&\quad+\int_K \boldsymbol{v}\cdot \boldsymbol{x}_Kq_{l-1} \d K+ \sum_f\int_f (\boldsymbol{n}_f\times(\boldsymbol{x}_K\times \boldsymbol{p}_{l}))_{\tau}\cdot \boldsymbol{\Pi}_{l+1}^{0,f}\boldsymbol{v}_{\tau} \d f.
			\end{aligned}
		\end{equation}
		The right-hand side of \eqref{ComL2Pro} can be computed as follows: the first term is computable via $\boldsymbol{\Pi}^{\nabla,K}_{l+1}$ in $\boldsymbol{W}_{k}(K)$; the second and third terms follow from the degrees of freedom \eqref{dofV7} and \eqref{dofV8}; and the last term, $\boldsymbol{\Pi}^{0,f}_{l+1}$, is obtained using the decomposition \eqref{Pkdcf}, the projection $\Pi_{l+1}^{\nabla,f}$ in $\overline{\mathbb{B}}_{k}(f)$, and the degrees of freedom \eqref{dofV3}, \eqref{dofV5}, and \eqref{dofV6}. The computation for the last term proceeds as follows:
		\begin{equation*}
			\begin{aligned}
				\int_f \boldsymbol{\Pi }_{l+1}^{0,f} \boldsymbol{v}_{\tau}\cdot \boldsymbol{p}_{l+1} \d f&=\int_f \boldsymbol{v}_{\tau}\cdot (\overrightarrow{\nabla}_{f}\times \widehat{p}_{l+2}+\overrightarrow{\nabla}_{f}\times p_m+\boldsymbol{x}_f p_{l})\d f\\
				&=\int_f \Pi_{l+1}^{\nabla,f} \nabla_{f}\times   \boldsymbol{v}_{\tau}\widehat{p}_{l+2} \d f+\int_f \boldsymbol{v}_{\tau}\cdot \overrightarrow{\nabla}_{f}\times p_m \d f-\int_{\partial f} \boldsymbol{v}\cdot\boldsymbol{t}_{\partial f} \widehat{p}_{l+2} \d S\\
				&\quad +\int_f \boldsymbol{v}\cdot \boldsymbol{x}_f p_{l}\d f, \forall \boldsymbol{p}_{l+1} \in \boldsymbol{P}_{l+1}(f),
			\end{aligned}
		\end{equation*}
		where $\widehat{p}_{l}\in P_{l}(f)/P_{m}(f)$, $p_m\in P_m(f)$, and  $p_{l}\in P_{l}(f)$ with $m=\operatorname{max}\{0,k-2\}$. Note that the final term vanishes for $l=0$ due to the definition \eqref{reducedE}.
		The proof is complete.
	\end{proof}
	Gluing the local space $\boldsymbol{V}_{r-1,k+1}(K)$ over all elements $K$ in $\mathcal{T}_h$ produces the global space
	\begin{equation}
		\boldsymbol{V}_{r-1,k+1}(\Omega):=\{ \boldsymbol{v} \in \boldsymbol{V}( \Omega):\boldsymbol{v}|_K \in \boldsymbol{V}_{r-1,k+1}(K),\forall K \in \mathcal{T}_h\}.
	\end{equation}
	We note that the global set of degrees of freedom, defined as the counterpart of \eqref{dofV1}–\eqref{dofV8}, guarantees the conforming property $\nabla\times \boldsymbol{V}_{r-1,k+1}(\Omega) \subset \boldsymbol{H}^1(\Omega)$.
	\subsection{Scalar $L^2$-conforming space}
	We end our construction with the rightmost space in the discrete complex \eqref{DisComplex}. The local space is defined as 
	\begin{equation*}
		Q_{k-1}(K):=P_{k-1}(K),
	\end{equation*}
	with the degrees of freedom 
	\begin{flalign*}
		&\ \bullet \mathbf{D}_{Q}:\text{the volume moments } 
		\int_K q p_{k-1}, \forall  p_{k-1} \in 
		P_{k-1}(K).&
	\end{flalign*}
	The global space $Q_{k-1}(\Omega)$ consists of discontinuous piecewise polynomials:
	\begin{equation*}
		Q_{k-1}(\Omega):=\{q\in L^2(\Omega): q|_K\in P_{k-1}(K), \forall K \in \mathcal{T}_h\}.
	\end{equation*}
	\subsection{The discrete complex}
	In this subsection, we establish the exactness of the discrete complex \eqref{DisComplex} and its connection to the continuous complex \eqref{ContiComplex}.  
	\begin{theorem}
		The discrete complex \eqref{DisComplex} is exact.
	\end{theorem}
	\begin{proof}
		From the perspective of the vector potential problem for the Stokes equations with non-homogeneous boundary conditions,  we specifically prove the exactness 
		\begin{equation}\label{shortcomplex}
			\nabla\times\boldsymbol{V}_{r-1,k+1}(\Omega)= \boldsymbol{W}_k(\Omega)\cap \operatorname{ker(\nabla\cdot)}.
		\end{equation}
		The proof for the remaining part of the sequence follows the same lines as in \cite[Theorem 4.1]{VEMforStokes}.
		\par 
		Since the restrictions in \eqref{WFinal} and \eqref{FinalSpace} are compatible with the curl operator, it suffices to establish the exactness between the local enlarged spaces $\widetilde{\boldsymbol{V}}_{r-1,k+1}(K)$ and $\widetilde{\boldsymbol{W}}_k(K)$; the global result \eqref{shortcomplex} then follows by patching the local spaces.
		\par 
		Observe that for any $\boldsymbol{u}^{K} \in \boldsymbol{W}_k(K) \cap \ker(\nabla\cdot)$, it is the unique solution to the Stokes problem \eqref{StokesProblem} on the local element $K$, with given data $$\boldsymbol{f} \in \boldsymbol{x}_K \times \boldsymbol{P}_{k-1}(K) \quad\text{and} \quad  \boldsymbol{g}_1 \in \mathbb{B}_k(\partial K).$$  
		Similarly, for any $\boldsymbol{\psi}^{K} \in \boldsymbol{V}_{r-1,k+1}(K)$, it is the unique solution to the $-\operatorname{curl} \Delta \operatorname{curl}$ problem \eqref{curlDeltacurlProblem} on $K$, with data $$\boldsymbol{j} \in \nabla \times (\boldsymbol{x}_K \times \boldsymbol{P}_{k-1}(K)), \quad  j_d \in P_{r-1}(K), $$ 
		$$(\boldsymbol{h}, \boldsymbol{g}_2) \in (\mathbb{E}_{r-1,k+1}(\partial K),\mathbb{B}_k(\partial K)) \subset \boldsymbol{Y}(\partial K).$$
		It therefore follows from  Remark \ref{NonHomEqu} and Theorem \ref{NonHomThem} that when setting $\boldsymbol{j} = \nabla \times \boldsymbol{f}$ and $\boldsymbol{g}_1 = \boldsymbol{g}_2$, we obtain
		\begin{equation*}
			\nabla \times \boldsymbol{\psi}^{K} = \boldsymbol{u}^{K},
		\end{equation*}
		which implies
		\begin{equation}
			\nabla \times \widetilde{\boldsymbol{V}}_{r-1,k+1}(K) = \widetilde{\boldsymbol{W}}_k(K) \cap \ker(\nabla \cdot).
		\end{equation}
	\end{proof}
	
	\begin{remark}
		The $\boldsymbol{H}(\operatorname{grad-curl})$-conforming virtual element spaces presented in this work generalize three families of $\boldsymbol{H}(\operatorname{grad-curl})$-conforming finite elements \cite{ZZMQuadCurl} to polyhedral meshes, which correspond to the specific parameter choices $r = k$, $k+1$, and $k+2$. In the lowest-order case $r=k=1$, the degrees of freedom for our virtual element spaces, defined for general polyhedra and illustrated in Figure \ref{fig:dofgradcurl}, precisely recover those of the discrete finite element complex from \cite{ZZMQuadCurl} when restricted to simplex meshes.
	\end{remark}
	\begin{figure}
		\centering
		\includegraphics[width=1.0\linewidth, trim=0 3cm 0 3cm, clip]{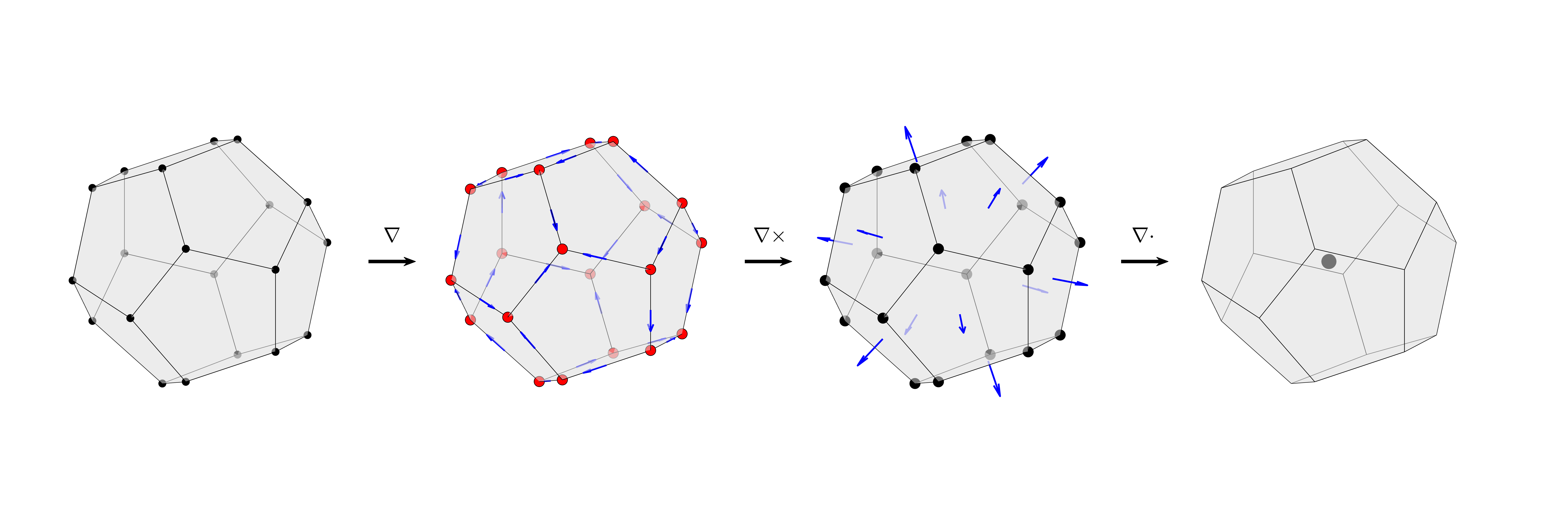}
		\caption{The lowest-order ($r = k = 1$) virtual element complex \eqref{DisComplex} on a polyhedral element.}
		\label{fig:dofgradcurl}
	\end{figure}
	\par
	For $s>\frac{1}{2}$, the following interpolation operators, determined  by the degrees of freedom of their virtual element spaces, are given by
	\begin{gather*}
		J_h:  H^{s+1}(\Omega) \to U_r(\Omega),\quad
		\boldsymbol{I}_h: \boldsymbol{V}^s(\Omega)\to \boldsymbol{V}_{r-1,k+1}(\Omega),\\
		\boldsymbol{J}_h: \boldsymbol{H}^{s+1}(\Omega) \to \boldsymbol{W}_k(\Omega), \quad 
		P_h:L_2(\Omega) \to Q_{k-1}(\Omega),
	\end{gather*}
	where
	$$
	\boldsymbol{V}^s(\Omega):= \{\v \in \boldsymbol{H}^s(\Omega): \nabla\times \v \in \boldsymbol{H}^{s+1}(\Omega)\}.
	$$
	The continuous embedding $H^{s+1}(\Omega) \hookrightarrow C^{0,s-\frac{1}{2}}(\Omega)$  and the trace theorem \eqref{Trace0} ensure that $J_h$, $\boldsymbol{J}_h$, and $\boldsymbol{I}_h$ are well-defined.  
	\begin{remark}
		The regularity result in Remark \ref{Regu} implies that for some $s > \frac{1}{2}$, the solution $\boldsymbol{\psi}$ to the grad-curl problem \eqref{VectorPotentialProblem} has the regularity $\boldsymbol{\psi} \in \boldsymbol{V}^s(\Omega)$, whence the interpolant $\boldsymbol{I}_h \boldsymbol{\psi}$ is well-defined.
	\end{remark}
	\begin{proposition}
		The last two rows of the following diagram commute.
		\begin{equation}\label{CommutDigram}
			\begin{tikzcd}
				\mathbb{R}\arrow[r, "\subset"] & H^1(\Omega) \arrow[r, "\nabla"] \arrow[d,"{\rotatebox{90}{$\subset$}}"] & \boldsymbol{V}( \Omega)  \arrow[r, "\nabla \times"] \arrow[d, "{\rotatebox{90}{$\subset$}}"] & \boldsymbol{H}^1( \Omega) \arrow[r, "\nabla\cdot"] \arrow[d,"{\rotatebox{90}{$\subset$}}" ] & L^2(\Omega) \arrow[r] \arrow[d, ] & 0 \\
				\mathbb{R} \arrow[r, "\subset"] & H^{s+1}(\Omega) \arrow[r, "\nabla"] \arrow[d, "J_h"] & \boldsymbol{V}^s(\Omega) \arrow[r, "\operatorname{\nabla\times}"] \arrow[d, "\boldsymbol{I}_h"] & \boldsymbol{H}^{s+1}( \Omega) \arrow[r, "\nabla\cdot"] \arrow[d, "\boldsymbol{J}_h"] & L^2(\Omega) \arrow[r] \arrow[d, "P_h"] & 0 
				\\
				\mathbb{R}\arrow[r, "\subset"] & U_r(\Omega) \arrow[r, "\nabla"] & \boldsymbol{V}_{r-1,k+1}(\Omega) \arrow[r, "\nabla \times"] & \boldsymbol{W}_{k}(\Omega) \arrow[r, "\nabla\cdot"] & Q_{k-1}(\Omega) \arrow[r] & 0 .
			\end{tikzcd}
		\end{equation}
	\end{proposition}
	
	\begin{proof}
		\par 
		We take the middle columns as an example.
		Let $\boldsymbol{v} \in \boldsymbol{V}^{s}(\Omega) $. The commutativity of $\boldsymbol{I}_h$ and $\boldsymbol{j}_{h}$  means that 
		\begin{equation}\label{commut}
			\nabla \times \boldsymbol{I}_h \boldsymbol{v} =\boldsymbol{j}_{h} \nabla \times \boldsymbol{v}.
		\end{equation}
		On each edge $e$ of $K$, the degrees of freedom \eqref{dofV1},  \eqref{dofV2}, \eqref{DW1}, and \eqref{DW2} yield 
		\begin{equation}\label{EdgeCom}
			\nabla \times \boldsymbol{I}_h \v|_e =\boldsymbol{J}_h \nabla\times \v |_e.	
		\end{equation}
		Combining this with the normal moments  \eqref{dofV5} and \eqref{DW4}, we have on each face $f \in \partial K$:
		\begin{equation}
			\begin{aligned}
				\int_f	\nabla\times \boldsymbol{I}_h \v \cdot \boldsymbol{n}_f q_m \d f
				&=\int_{\partial f} \nabla_{f}\times   \boldsymbol{I}_h \v q_m \d S
				= \int_{\partial f} \boldsymbol{I}_h \v \cdot \boldsymbol{t}_{\partial f} q_m \d S+ \int_f \boldsymbol{I}_h \v \cdot  \overrightarrow{\nabla}_{f}\times  q_m \d f\\
				&= \int_{\partial f}  \v \cdot \boldsymbol{t}_{\partial f} q_m \d S+ \int_f  \v \cdot  \overrightarrow{\nabla}_{f}\times  q_m \d f\\
				&=\int_f \nabla_{f}\times   \v  q_m \d f=\int_f \nabla \times \v \cdot \boldsymbol{n}_f q_m \d f\\
				&=\int_f (\boldsymbol{J}_h \nabla \times \v )\cdot \boldsymbol{n}_f q_m \d f,
				\quad \forall q_m \in P_m(f),
			\end{aligned}
		\end{equation}
		where $ m = \operatorname{max}\{0,k-2\}$.
		Moreover, the degrees of freedom of the tangential component \eqref{dofV4} and \eqref{DW3} imply that 
		\begin{equation}\label{TanMoment}
			\int_f(\nabla\times \boldsymbol{I}_h \v)_{\tau} q_{k-2} \d f  = \int_f (\boldsymbol{J}_h\nabla\times \v)_{\tau} q_{k-2} \d f, \quad \forall q_{k-2} \in P_{k-2}(f).
		\end{equation}
		It follows from the definitions of $\widehat{\mathbb{B}}_k(f)$ and $\overline{\mathbb{B}}_k(f)$, together with \eqref{EdgeCom}-\eqref{TanMoment}, that 
		\begin{equation}\label{commut1}
			(\nabla \times \boldsymbol{I}_h\boldsymbol{v} -\boldsymbol{j}_{h}\nabla \times\boldsymbol{v})|_{\partial K} =0. 
		\end{equation}
		On the element $K$, a straightforward verification shows
		\begin{equation}\label{commut2}
			\begin{aligned}
				(\nabla \times\boldsymbol{I}_h \boldsymbol{v}, \boldsymbol{x}_K \times \boldsymbol{p}_{k-3})_K&=  (\nabla\times \boldsymbol{v}, \boldsymbol{x}_K\times \boldsymbol{p}_{k-3})_K\\
				&=(\boldsymbol{J}_h\nabla\times \boldsymbol{v},\boldsymbol{x}_K\times \boldsymbol{p}_{k-3})_K
			\end{aligned}
		\end{equation}
		In view of \eqref{commut1} and \eqref{commut2},
		the unisolvence of the degrees of freedom for $\boldsymbol{W}_k(K)$ leads to \eqref{commut}.
		The commutativity of other columns can be proved in a similar way.
	\end{proof}
	\subsection{Interpolation error estimates}
	As indicated in Remark \ref{PolyCompl}, the choice of $r$ influences the polynomial compatibility. For simplicity, we restrict our attention to the case $r = k$ in the remainder of this paper, which minimizes the degrees of freedom; analogous arguments apply to other cases. This subsection establishes interpolation error estimates for the local space $\boldsymbol{V}_{k-1,k+1}(K)$.\par 
	We begin by introducing the $\boldsymbol{H}(\operatorname{curl})$-conforming virtual element space \cite{VEMforMaxwell,VEMforGener}$, $ with a slight modification of polynomial degrees. For $k\ge 2$, the local face space is defined as
	\begin{align*}
		\mathbb{E}^e_{k-1, k+1}(f)=\left\{ \boldsymbol{v} \in \boldsymbol{L}^2(f):  \nabla\cdot \boldsymbol{v}\in P_{k-1}(f), \nabla_{f}\times  \boldsymbol{v}\in P_{k}(f),\right. \\
		\left.(\boldsymbol{v}\cdot \boldsymbol{t}_{\partial f})\in P_{k-1}(e), \forall e\in \partial f\right\},
	\end{align*} 
	For $k=1$, it is given by 
	\begin{equation*}
		\begin{aligned}
			\mathbb{E}^e_{0,2}(f):=\left\{\boldsymbol{v} \in \boldsymbol{L}^2(f): \nabla \cdot  \boldsymbol{v}\in P_{0}(f) ,  \nabla_{f}\times   \boldsymbol{v} \in P_1(f),\left(\boldsymbol{v} \cdot \boldsymbol{t}_{\partial f}\right)|_e \in P_{0}(e), \right.\\
			\left.
			\forall e\in \partial f,\int_f \v \cdot \boldsymbol{x}_fp_0 =0, \forall p_0\in P_0(f)
			\right\}.
		\end{aligned}
	\end{equation*}
	The corresponding local element space is 
	\begin{equation}\label{Vke}
		\begin{aligned}
			\boldsymbol{V}^e_{k-1, k+1}(K):=\left\{\v \in  \boldsymbol{L}^2(K):\nabla \cdot \v \in P_{k-2}(K),\nabla \times \nabla \times  \v \in \boldsymbol{P}_{k-1}(K), \right. \\
			\left.  \v_{\tau} \in \mathbb{E}^e_{k-1,k+1}(f), \forall f \in \partial K \right\},
		\end{aligned}
	\end{equation}
	with the following degrees of freedom 
	\begin{flalign} \label{edgeDof1}
		&\	\bullet \mathbf{D}^1_{\boldsymbol{V}^e}: \text{on each edge } e \text{ of } K \text{ the moments } \frac{1}{|e|}\int_e \v\cdot \boldsymbol{t}_e p_{k-1} \d e, \\
		&\	\bullet \mathbf{D}^2_{\boldsymbol{V}^e}: \text{the face moments of } \frac{1}{|f|} \int_{f}  \v_{\tau}\cdot \overrightarrow{\nabla}_{f}\times  p_{k} \d f ,  \forall p_{k}\in P_{k}(f),\\
		&\ \bullet \mathbf{D}^3_{\boldsymbol{V}^e}: \text{the  face moments  } \frac{1}{|f|}\int_{f}  \boldsymbol{v}_{\tau}\cdot \boldsymbol{x}_f p_{\widehat{k}-1} \d f, \forall p_{\widehat{k}-1} \in P_{k-1}(f), &\\
		&\  \bullet \mathbf{D}^4_{\boldsymbol{V}^e}: \text{the volume moments of } \int_K \v \cdot \boldsymbol{x}_K p_{k-2} \d K, \forall p_{k-2}\in P_{k-2}(K)&\\
		&\	\bullet \mathbf{D}^5_{\boldsymbol{V}^e}: \text{the volume moments of } \frac{1}{|K|}\int_K \nabla \times  \v\cdot( \boldsymbol{x}_K\times\boldsymbol{p}_{k-1} ) \d K,\label{edgeDof5} \\
		\nonumber &\quad \quad \quad \quad \quad \quad \quad \quad \quad \quad \quad 
		\quad \quad \quad \quad \quad \quad \quad \, \, \, \forall \boldsymbol{p}_{k-1} \in \boldsymbol{P}_{k-1}(K).
	\end{flalign}
	Here, $\widehat{k}$ is defined as in \eqref{rwidehat}. According to \cite[Proposition 3.7]{VEMforMaxwell}, the $\boldsymbol{L}^2$-projection $\boldsymbol{\Pi}_{k-1}^{0,K}:\boldsymbol{V}^e_{k-1,k+1}(K) \to \boldsymbol{P}_{k-1}(K)$ is computable from the above degrees of freedom.
	As established  in \cite[Lemma 4.4]{VEMforGener}, the following auxiliary bound holds for $\boldsymbol{V}^e_{k-1, k+1}(K)$.
	Using an analogous argument, this bound extends to $\boldsymbol{V}_{k-1,k+1}(K)$ and their direct sum. For completeness, we present  a brief proof.
	\begin{lemma}\label{lowbound}
		For each $\boldsymbol{v} \in  \boldsymbol{V}^e_{k-1, k+1}(K)$, $ \boldsymbol{V}_{k-1, k+1}(K)$ or their sum, there holds
		\begin{align*}
			\|\boldsymbol{v}\|_{K} &\lesssim h_K\|\nabla \times \boldsymbol{v}\|_{K}+ \sup _{p_{k-2} \in P_{k-2}(K)} \frac{\int_{K} \boldsymbol{v} \cdot (\boldsymbol{x}_{K} p_{k-2}) \d K}{\left\|\boldsymbol{x}_{K}  p_{k-2}\right\|_{K}}.\\
			&\quad + \sum_{f\in \partial K}\left( h_f^{\frac{3}{2}}\|\nabla_{f}\times  \boldsymbol{v}_{\tau}\|_{f}+ h_f\|\v_{\tau} \cdot \boldsymbol{t}_{\partial f}\|_{\partial f} + \sup _{p_{k-1} \in P_{k-1}(f)} \frac{\int_{f} \boldsymbol{v}_{\tau} \cdot (\boldsymbol{x}_{f} p_{k-1}) \d f}{\left\|\boldsymbol{x}_{f}  p_{k-1}\right\|_{f}} \right).
		\end{align*}
	\end{lemma} 
	\begin{proof}
		It follows from \cite[Lemma 4.4]{VEMforGener} that the following Helmholtz decomposition holds:
		\begin{equation*}
			\boldsymbol{v} = \nabla\times \boldsymbol{\rho}+\nabla\phi,
		\end{equation*}
		where $\phi \in H^1_0(K)$ satisfies the Poisson equation
		\begin{equation*}
			\Delta\phi = \div \v \text{ in } H^{-1}(K),
		\end{equation*}
		and $\boldsymbol{\rho}\in \boldsymbol{H}(\curl;K)\cap\boldsymbol{H}(\div;K)$ satisfies weakly
		\begin{equation*}
			\begin{aligned}
				\nabla\times\nabla\times \boldsymbol{\rho}&=\nabla\times \boldsymbol{v},\quad \nabla\cdot \boldsymbol{\rho}=0 \quad \text{ in } K,\\
				(\nabla\times\boldsymbol{\rho})\times\boldsymbol{n}_{\partial K}&=\boldsymbol{v}\times \boldsymbol{n}_{\partial K},\quad \boldsymbol{\rho}\cdot \boldsymbol{n}_{\partial K}=0 \quad \text{ on }\partial K.
			\end{aligned}
		\end{equation*}
		Moreover,  the following identities hold:
		\begin{equation}\label{DecElement}
			(\nabla\times \boldsymbol{\rho},\nabla\phi)_K=0,\quad \|\boldsymbol{v}\|^2_K=\|\nabla\times \boldsymbol{\rho}\|_K^2+\|\nabla\phi\|_K^2,
		\end{equation}
		along with the estimates
		\begin{gather}
			\|\nabla\times \boldsymbol{\boldsymbol{\rho}}\|_K\lesssim h_K\|\nabla\times \boldsymbol{v}\|_K + h_K^{\frac{1}{2}}\|\boldsymbol{v}\times\boldsymbol{n}_{\partial K}\|_{\partial K} ,\\
			\|\nabla\phi\|_K^2\lesssim \left(h_K\|\nabla\times \boldsymbol{v}\|_K+ h_K^{\frac{1}{2}}\|\boldsymbol{v}\times\boldsymbol{n}_{\partial K}\|_{\partial K}+\sup _{p_{k-2} \in P_{k-2}(K)} \frac{\int_{K} \boldsymbol{v} \cdot (\boldsymbol{x}_{K} p_{k-2}) \d K}{\left\|\boldsymbol{x}_{K}  p_{k-2}\right\|_{K}}\right)\|\v\|_K.
		\end{gather}
		For the second inequality, we have used the fact that $\nabla \cdot \boldsymbol{v} \in P_{k-2}(K)$, which follows from the definitions of the spaces in \eqref{FinalSpace} and \eqref{Vke}.
		A similar Helmholtz decomposition on the face $f$ from \cite[Lemma 3.1]{VEMforGener} yields 
		\begin{equation}\label{FaceEsti}
			\|\boldsymbol{v}_{\tau}\|_f\lesssim h_f\|\nabla_f\times \v_{\tau}\|+ h_f^{\frac{1}{2}}\|\boldsymbol{v}_\tau\cdot \boldsymbol{t}_{\partial f}\|_{\partial f} + \sup _{p_{k-1} \in P_{k-1}(f)} \frac{\int_{f} \boldsymbol{v}_{\tau} \cdot (\boldsymbol{x}_{f} p_{k-1}) \d f}{\left\|\boldsymbol{x}_{f}  p_{k-1}\right\|_{f}},
		\end{equation}
		which is valid for $\boldsymbol{v}_{\tau}$ satisfying $\nabla_f \cdot \boldsymbol{v}_{\tau} \in P_{k-1}(f)$ and $\v\cdot \boldsymbol{t}_e \in P_{k-1}(e)$ on each edge $e\in \partial f$. Combing \eqref{DecElement}-\eqref{FaceEsti}, we have the desired result.
	\end{proof}
	An auxiliary interpolation operator $\boldsymbol{I}_h^e: \{\v\in \boldsymbol{H}^s(K): \nabla\times\v \in \boldsymbol{H}^s(K)\}  \to \boldsymbol{V}^e_{k-1,k+1}(K)$  with $s >\frac{1}{2}$, defined based on the degrees of freedom   \eqref{edgeDof1}-\eqref{edgeDof5},
	satisfies the  following interpolation error estimates  \cite[Theorem 4.5]{VEMforGener}. Here, $[s]$ denotes the greatest integer strictly less than $s$. For example, $[1]=0$.
	\begin{lemma}
		Let $\boldsymbol{v} \in \boldsymbol{H}^s(K)$ with $\nabla \times \boldsymbol{v} \in \boldsymbol{H}^{l}(K)$, where $\frac{1}{2} < s \leq k$ and $\frac{1}{2} < l \leq k+1$. Define $\widetilde{l} = \min\{l, [s]\}$. Then the following estimates hold:
		\begin{align}\label{Ihe}
			\|\boldsymbol{v}-\boldsymbol{I}_h^e\boldsymbol{v}\|_K &\lesssim h_K^s|\boldsymbol{v}|_{s, K} + h_K^{\widetilde{l}+1}|\nabla\times\v|_{\widetilde{l},K}+ h_K\|\nabla\times \v\|_K,  \\
			\label{CurlIhe}
			\|\nabla \times(\boldsymbol{v}-\boldsymbol{I}_h^e\boldsymbol{v})\|_{K}&\lesssim h_K^{l}|\nabla \times \boldsymbol{v}|_{l, K}.
		\end{align}
		The third term on the right-hand side of \eqref{Ihe} vanishes when $s \geq 1$. Furthermore, for any face $f \in \partial K$, 
		it holds
		\begin{equation}
			\label{RotIne}
			\|\nabla_{f}\times  (\v-\boldsymbol{I}_h^e\v)_{\tau}\|_f \lesssim h_f^{l-\frac{1}{2}}|\nabla_{f}\times  \v_{\tau}|_{l-\frac{1}{2},f}.
		\end{equation} 
	\end{lemma}
	We observe that when $s < 1$, $\widetilde{l} = \min\{l, 0\} = 0$, so that $h_K^{\widetilde{l}+1}|\nabla\times\boldsymbol{v}|_{\widetilde{l},K} = h_K\|\nabla\times\boldsymbol{v}\|_K$, thereby trivially allowing the omission of the third term in \eqref{Ihe} regardless of $s$.
	We also introduce the interpolation error estimates for $\boldsymbol{j}_{h}$, see \cite[Theorem 5]{JM2023} and \cite[Lemma 4.2]{VEMforMHD}.
	\begin{lemma}
		For every $\boldsymbol{v} \in \boldsymbol{H}^1_0(\Omega) \cap \boldsymbol{H}^{s}(\Omega)$ with $\frac{3}{2} < s \le k+1$, it holds that
		\begin{equation}\label{interpolation_J}
			\|\boldsymbol{v}-\boldsymbol{J}_h\boldsymbol{v}\|_K+h_K|\boldsymbol{v}-\boldsymbol{J}_h\boldsymbol{v}|_{1, K} \lesssim h_K^s|\boldsymbol{v}|_{s, K}.
		\end{equation}
	\end{lemma}
	The classical trace inequality on an element or face $G$ is given by
	\begin{equation}\label{H1trace}
		\|v\|_{\partial G} \lesssim h_G^{-\frac{1}{2}}\|v\|_G + h_G^{\frac{1}{2}}|v|_{1,G},\quad \forall v\in H^1(G).
	\end{equation}
	By restricting the boundary $\partial K$ to each face $f \in\partial K$, we generalize the trace inequality in \cite[Lemma 2.1]{VEMforSmallEdge} to functions in the higher-regularity space $H^{s+1}(K)$ with $s > \frac{1}{2}$:
	\begin{equation}\label{Hs1trace}
		|v|_{s+\frac{1}{2},f} \lesssim h_K^{[s]-s}|v|_{[s+1],K} + |v|_{s+1,K},\quad \forall v\in H^{s+1}(K).
	\end{equation}
	Building on the above preparation, we establish the interpolation error estimates for $\boldsymbol{I}_h$ as follows:
	\begin{theorem}
		If $\boldsymbol{v} \in \boldsymbol{H}^s(\Omega)$ with $\nabla \times \boldsymbol{v} \in \boldsymbol{H}^{s+1}(\Omega)$, where $\frac{1}{2} < s \le k$, then the following estimates hold:
		\begin{align}
			\label{Intpolation1}
			\|\boldsymbol{v}-\boldsymbol{I}_h\boldsymbol{v}\|_K&\lesssim h^s_K|\boldsymbol{v}|_{s, K}+ h_K^{[s]+1}|\nabla\times\v|_{[s],K}\\
			&\quad+h_K^{[s]+2}|\nabla\times \v|_{[s+1],K}+h^{s+2}_K|\nabla \times\boldsymbol{v}|_{s+1, K},\\
			\label{Intpolation2}
			\|\nabla \times(\boldsymbol{v}-\boldsymbol{I}_h\boldsymbol{v})\|_K&\lesssim h_K^{s+1}|\nabla \times\boldsymbol{v}|_{s+1, K},\\
			\label{Intpolation3}
			|\nabla \times(\boldsymbol{v}-\boldsymbol{I}_h\boldsymbol{v})|_{1, K}&\lesssim h^s_K|\nabla \times\boldsymbol{v}|_{s+1, K}.
		\end{align}
	\end{theorem}
	\begin{proof}
		Using the commutative property \eqref{commut} between $\boldsymbol{I}_h $ and $\boldsymbol{J}_h$, along with the interpolation error estimate \eqref{interpolation_J}, the results \eqref{Intpolation2} and \eqref{Intpolation3} follow directly.
		Applying the interpolation error estimate \eqref{Ihe}, we derive \eqref{Intpolation1} as follows:
		\begin{equation}\label{v-Ihv}
			\begin{aligned}
				\|\boldsymbol{v}-\boldsymbol{I}_h\boldsymbol{v}\|_K&\lesssim \|\boldsymbol{v}-\boldsymbol{I}_h^e\boldsymbol{v}\|_K+\|\boldsymbol{I}_h^e\boldsymbol{v}-\boldsymbol{I}_h\boldsymbol{v}\|_K\\
				&\lesssim h_K^s|\v|_{s,K}+h_K^{[s]+1}|\nabla\times \v|_{[s],K}+\|\boldsymbol{I}_h^e\boldsymbol{v}-\boldsymbol{I}_h\boldsymbol{v}\|_K.
			\end{aligned}
		\end{equation}
		We now turn to estimating the error $\boldsymbol{I}_h^e \boldsymbol{v}-\boldsymbol{I}_h\boldsymbol{v}$.
		From the definition of $\boldsymbol{I}_h^e $ and $\boldsymbol{I}_h$,  the following hold
		\begin{align*}
			(\boldsymbol{I}^e_h\v - \boldsymbol{I}_h\v)\cdot \boldsymbol{t}_{e} &= 0 \quad \text{on each edge } e \text{ of } K, \\
			\int_f (\boldsymbol{I}_h^e\boldsymbol{v}-\boldsymbol{I}_h\boldsymbol{v})_{\tau}\cdot (\boldsymbol{x}_f p_{k-1})\d f&=0, \quad \forall p_{k-1}\in P_{k-1}(f),\\
			\int_K (\boldsymbol{I}^e_h\v - \boldsymbol{I}_h\v)\cdot (\boldsymbol{x}_K p_{k-2}) \d K&= 0, \quad \forall p_{k-2} \in P_{k-2}(K).
		\end{align*} 
		Combining these  properties with Lemma \ref{lowbound}, we obtain
		\begin{align}\label{Ih-Ihe}
			\|\boldsymbol{I}_h^e\boldsymbol{v}-\boldsymbol{I}_h\boldsymbol{v}\|_{K} \lesssim h_K\|\nabla \times (\boldsymbol{I}_h^e\boldsymbol{v}-\boldsymbol{I}_h\boldsymbol{v})\|_{K} + \sum_{f\in \partial K} h_f^{\frac{3}{2}}\|\nabla_{f}\times  (\boldsymbol{I}_h^e\boldsymbol{v}-\boldsymbol{I}_h\boldsymbol{v})_{\tau}\|_{f}.
		\end{align}
		Using \eqref{CurlIhe} and \eqref{Intpolation2} to estimate the first term on the right-hand side of \eqref{Ih-Ihe} yields
		\begin{equation}\label{curlIh-Ihe}
			\begin{aligned}
				h_K\|\nabla \times (\boldsymbol{I}_h^e\boldsymbol{v}-\boldsymbol{I}_h\boldsymbol{v})\|_{K}&\lesssim h_K\|\nabla\times(\v-\boldsymbol{I}_h^e\v)\|_K+h_K\|\nabla\times(\boldsymbol{v}-\boldsymbol{I}_h\v)\|_K\\
				&\lesssim h_K^{s+2}|\nabla\times \v|_{s+1,K}.
			\end{aligned}
		\end{equation}
		For the last term in \eqref{Ih-Ihe}, using the interpolation error estimates \eqref{RotIne}, \eqref{Intpolation2}, \eqref{Intpolation3} and the trace inequality \eqref{Hs1trace}, we obtain
		\begin{equation}\label{RotIn-Ihe}
			\begin{aligned}
				\sum_{f\in \partial K} h_f^{\frac{3}{2}}\|\nabla_{f}\times  (\boldsymbol{I}_h^e\boldsymbol{v}-\boldsymbol{I}_h\boldsymbol{v})_{\tau}\|_{f}&\lesssim \sum_{f\in \partial K} h_f^{\frac{3}{2}}(\|\nabla_{f}\times  (\v-\boldsymbol{I}_h\v)_{\tau}\|_f+\|\nabla_{f}\times  (\v-\boldsymbol{I}_h^e\v)_{\tau}\|_f)\\
				&\lesssim
				h_K^{s-\frac{1}{2}}\|\nabla\times(\v-\boldsymbol{I}_h\v)\|_K+ h_K^{s+\frac{1}{2}}|\nabla\times(\v-\boldsymbol{I}_h\v)|_{1,K}\\
				&\quad +
				\sum_{f\in \partial K}
				h_f^{s+2}|\nabla_{f}\times  \v|_{s+\frac{1}{2},f} \\
				&\lesssim h_K^{[s]+2}|\nabla\times \v|_{[s+1],K} + h_K^{s+2}|\nabla\times \v|_{s+1,K}.
			\end{aligned}
		\end{equation}
		Thus, combining \eqref{v-Ihv} through \eqref{RotIn-Ihe}, the proof is complete.
	\end{proof} 
	\section{Discretization}
	\subsection{The discrete bilinear forms}
	For any $\boldsymbol{v}_h,\boldsymbol{w}_h \in \boldsymbol{V}_{k-1,k+1}(\Omega)$,
	we define the following bilinear form to discretize $(\nabla \nabla\times \boldsymbol{v}_h,\nabla\nabla\times \boldsymbol{w}_h)_K$:
	\begin{align*}
		a_h^K(\nabla\times\boldsymbol{v}_h, \nabla\times\boldsymbol{w}_h)&:=(\nabla\boldsymbol{\Pi}_k^{\nabla,K}\nabla \times\boldsymbol{v}_h, \nabla\boldsymbol{\Pi}_k^{\nabla,K}\nabla \times\boldsymbol{w}_h)_K\\
		&\quad+S_1^K((I-\boldsymbol{\Pi}_k^{\nabla,K})\nabla \times\boldsymbol{v}_h, (I-\boldsymbol{\Pi}_k^{\nabla,K})\nabla \times\boldsymbol{w}_h).
	\end{align*}
	The stabilization $S_1^K: \boldsymbol{W}_k(K)\times \boldsymbol{W}_k(K) \to \mathbb{R}$ of \cite{JM2023} is extended to $k\ge1$ by 
	\begin{align*}
		S_1^K(\boldsymbol{\xi}_h,\boldsymbol{\eta}_h)&:= h_K^{-2}(\boldsymbol{\Pi}^{\nabla,K}_{k}\boldsymbol{\xi}_h,\boldsymbol{\Pi}^{\nabla,K}_{k}\boldsymbol{\eta}_h)+ (\nabla\cdot \boldsymbol{\xi}_h,\nabla\cdot \boldsymbol{\eta}_h)_K\\
		&\quad+\sum_{f\in \partial_K} (h_f^{-1}(\boldsymbol{\Pi}^{0,f}_k\boldsymbol{\xi}_h,\boldsymbol{\Pi}^{0,f}_k\boldsymbol{\eta}_h)_f+(\boldsymbol{\xi}_h,\boldsymbol{\eta}_h)_{\partial f}),
	\end{align*}
	which satisfies the following bounds for
	\begin{align}
		\label{LowerBoundH1}
		\alpha_* |\boldsymbol{\xi}_h|_{1,K}^2 &\le S_1^K(\boldsymbol{\xi}_h,\boldsymbol{\xi}_h),\quad \forall \boldsymbol{\xi}_h \in \boldsymbol{W}_k(K),\\
		S_1^K(\boldsymbol{\xi}_h,\boldsymbol{\xi}_h)&\le \alpha^*|\boldsymbol{\xi}_h|_{1,K}^2, \quad \forall \boldsymbol{\xi}_h \in \operatorname{ker}(\boldsymbol{\Pi}^{\nabla,K}_k) \cap \boldsymbol{W}_k(K).
	\end{align}
	Here, the positive constants $\alpha_*$ and $\alpha^* $ are independent of $h_K$.
	By the standard argument \cite{VEM2013}, the local bilinear form $a_h^K(\cdot, \cdot)$ satisfies the following properties:\\
	$\bullet$ consistency: for all $\v_h\in \boldsymbol{V}_{k-1,k+1}(K)$ and  $\boldsymbol{q}_{k+1}\in \boldsymbol{P}_{k+1}(K)$, 
	\begin{equation}\label{ahConsist}
		a_h^K(\nabla\times\v_h,\nabla\times\boldsymbol{q}_{k+1}) = a^K(\nabla\times\v_h,\nabla\times\boldsymbol{q}_{k+1}),
	\end{equation}
	$\bullet$ stability: for all $\v_h\in \boldsymbol{V}_{k-1,k+1}(K)$, 
	\begin{equation}\label{ahStab}
		(\nabla\nabla\times\v_h,\nabla\nabla\times\v_h)_K \lesssim a_h^K(\nabla\times\boldsymbol{v}_h,\nabla\times\v_h)\lesssim  (\nabla\nabla\times\v_h,\nabla\nabla\times\v_h)_K.
	\end{equation}
	\par
	For the bilinear form  $(\v_h,\nabla q_h )_K$ with $\boldsymbol{v}_h\in\boldsymbol{V}_{k-1,k+1}(K)$ and  $q_h\in U_k(K)$,  we define the discrete bilinear form 
	\begin{equation*}
		b_h^K(\cdot, \cdot ) :\boldsymbol{V}_{k-1,k+1}(K) \times \boldsymbol{V}_{k-1,k+1}(K) \to \mathbb{R},
	\end{equation*}
	by
	\begin{equation}\label{bhk}
		b_h^K(\boldsymbol{v}_h, \boldsymbol{w}_h):=(\boldsymbol{\Pi}_{k-1}^{0,K}\boldsymbol{v}_h, \boldsymbol{\Pi}_{k-1}^{0,K}\boldsymbol{w}_h)_K+S_2^K((I-\boldsymbol{\Pi}^{0,K}_{k-1})\boldsymbol{v}_h, (I-\boldsymbol{\Pi}_{k-1}^{0,K})\boldsymbol{w}_h).
	\end{equation}
	The stabilization $S_2^K$, defined for all $\boldsymbol{\xi}_h,\boldsymbol{\eta}_h \in \boldsymbol{V}_{k-1,k+1}(K)$, is given by
	\begin{equation}\label{SK} 
		\begin{aligned}
			S_2^K(\boldsymbol{\xi}_h, \boldsymbol{\eta}_h):=h_K^2(\boldsymbol{\Pi}_{k}^{0, K}\nabla \times\boldsymbol{\xi}_h, \boldsymbol{\Pi}_{k}^{0, K}\nabla \times\boldsymbol{\eta}_h)_{ K}+\sum_{f\in \partial K}\left[h_f(\boldsymbol{\Pi}^{0,f}_{+,k-1}\boldsymbol{\xi}_{h,\tau},\boldsymbol{\Pi}^{0,f}_{+,k-1}\boldsymbol{\eta}_{h,\tau})_f\right.\\
			+h_f^{3}(\boldsymbol{\Pi}^{0, f}_k \nabla \times\boldsymbol{\xi}_h, \boldsymbol{\Pi}^{0, f}_k \nabla \times\boldsymbol{\eta}_h)_{f} 
			\left.+h_f^{4}(\nabla \times\boldsymbol{\xi}_h, \nabla \times\boldsymbol{\eta}_h)_{ \partial f} 
			+h_f(\boldsymbol{\xi}_h\cdot \boldsymbol{t}_{\partial f}, \boldsymbol{\eta}_h\cdot \boldsymbol{t}_{\partial f})_{ \partial f}\right],
		\end{aligned}
	\end{equation}
	where the orthogonal projector $\boldsymbol{\Pi}^{0,f}_{+,k-1}: \boldsymbol{L}^2(f)\to \boldsymbol{x}_fP_{k-1}(f)$, which can be computed from the degrees of freedom defined in \eqref{dofV6} for $k\ge2$, satisfies
	\begin{equation*}
		(\boldsymbol{\Pi}^{0,f}_{+,k-1}\v -\v , \boldsymbol{x}_f p_{k-1})_f=0, \quad \forall \v \in  \mathbb{E}_{k-1,k+1}(f),\quad \forall p_{k-1} \in P_{k-1}(f).
	\end{equation*}
	Note that for $k=1$ and $\v\in \mathbb{E}_{0,2}(f)$, definition \eqref{reducedE} gives $\boldsymbol{\Pi}^{0,f}_{+,k-1}\v=0$.
	For all $ \boldsymbol{v}_h \in \boldsymbol{V}_{k-1, k+1}(K)\cap \ker(\boldsymbol{\Pi}^{0, K}_{k-1})$, this stabilization satisfies the following stability bounds, with the proof deferred to the next subsection (see Theorem \ref{LemmaS2Stab}):
	\begin{equation}\label{Stab2Esti}
		\begin{aligned}
			\beta_*\|\boldsymbol{v}_h\|^2_{K}\le S_2^K(\boldsymbol{v}_h, \boldsymbol{v}_h)\le \beta^* \|\v_h\|_{h,\boldsymbol{V}(K)},
		\end{aligned}
	\end{equation}
	where the scaled norm $\|\v_h\|_{h,\boldsymbol{V}(K)}$ is defined by 
	\begin{equation*}
		\|\v_h\|^2_{h,\boldsymbol{V}(K)}:= \|\vh\|^2_K+h_K^2\|\nabla \times \v_h\|^2_K+h_K^4|\nabla \times\v_h|^2_{1,K}
	\end{equation*}
	and the positive constants $\beta_*$ and $\beta^* $ are independent of $h_K$.
	From the lower bound estimate of \eqref{Stab2Esti}, for any $\boldsymbol{v}_h,\boldsymbol{w}_h\in\boldsymbol{V}_{k-1,k+1}(K)$, we easily obtain the coercivity 
	\begin{equation}\label{bhCoer}
		(\boldsymbol{v}_h,\boldsymbol{v}_h)_K\lesssim b_h^K(\boldsymbol{v}_h,\boldsymbol{v}_h).
	\end{equation}
	Using the upper bound estimate of \eqref{Stab2Esti} and the polynomial inverse estimate, we have the continuity   
	\begin{equation}
		\begin{aligned}\label{bhConti}
			b_h^K(\boldsymbol{v}_h,\boldsymbol{w}_h) &\lesssim (\|\vh\|^2_K+h_K^2\|\nabla\times (I-\boldsymbol{\Pi}_{k-1}^{0,K})\vh\|^2_K+h_K^4|\nabla\times (I-\boldsymbol{\Pi}_{k-1}^{0,K})\v_h|^2_{1,K})^{\frac{1}{2}} \\
			&\quad\quad(\|\wh\|^2_K+h_K^2\|\nabla\times (I-\boldsymbol{\Pi}_{k-1}^{0,K})\w_h\|^2_K+h_K^4|\nabla\times (I-\boldsymbol{\Pi}_{k-1}^{0,K})\w_h|^2_{1,K})^{\frac{1}{2}}\\
			&\lesssim  \|\vh\|_{h,\boldsymbol{V}(K)} \|\wh\|_{h,\boldsymbol{V}(K)}.
		\end{aligned}
	\end{equation}
	The consistency for $b_h^K(\cdot,\cdot)$ is satisfied by
	\begin{equation}\label{bhConsist}
		b_h^K(\boldsymbol{w}_h,\boldsymbol{q}_{k-1})=(\boldsymbol{w}_h,\boldsymbol{q}_{k-1})_K,\quad \forall \boldsymbol{w}_h\in \boldsymbol{V}_{k-1,k+1}(K),\boldsymbol{q}_{k-1}\in\boldsymbol{P}_{k-1}(K).
	\end{equation}
	\par
	As usual,  the global bilinear forms $a_h(\cdot, \cdot) $ are   $b_h(\cdot, \cdot) $  are defined by 
	\begin{equation*}
		\begin{aligned}
			a_h(\nabla\times\boldsymbol{v}, \nabla\times\boldsymbol{w})&=\sum_{K\in \mathcal{T}_h} a_h^K(\nabla\times\boldsymbol{v}, \nabla\times\boldsymbol{w}),  \quad \forall \boldsymbol{v} , \w \in \boldsymbol{V}_{k-1, k+1}(\Omega),  \\
			b_h(\boldsymbol{v}, \nabla q) &=\sum_{K\in \mathcal{T}_h}b_h^K(\boldsymbol{v}, \nabla q),  \quad  \forall \boldsymbol{v}\in \boldsymbol{V}_{k-1, k+1}(\Omega), q \in U_{k}(\Omega).
		\end{aligned}
	\end{equation*}
	\subsection{The stability results}
	In this subsection, we concentrate on proving the stability of $S_2^K$ defined by \eqref{SK}.  
	We give the following scaled trace inequalities \cite{Buffa2002}:
	\begin{align}
		\label{HrotTrace}
		\|\boldsymbol{v}\cdot \boldsymbol{t}_{\partial f}\|_{-\frac{1}{2},\partial f}&\lesssim \|\v\|_{f} +h_f\|\nabla_{f}\times  \v\|_{f},\quad \forall \v\in\boldsymbol{H}({\operatorname{curl}};f),\\
		\label{HcurlTrace}
		\|\boldsymbol{v}\times\boldsymbol{n}_{\partial K}\|_{-\frac{1}{2},\partial K}&\lesssim \|\boldsymbol{v}\|_K+ h_K\|\nabla\times \v\|_K,\quad\forall \v \in \boldsymbol{H}(\operatorname{curl};K).
	\end{align}
	Note that for any given $\v\in\mathbb{E}^e_{r-1,k+1}(f)$, $\v\cdot\boldsymbol{t}_{\partial f}$ is a piecewise polynomial on the boundary $\partial f$. Using the polynomial inverse inequality and \eqref{HrotTrace}, we obtain the following trace inequality:
	\begin{equation}\label{TraceTangetial}
		\|\v \cdot \boldsymbol{t}_{\partial f}\|_{\partial f} \lesssim h_f^{-\frac{1}{2}}\|\v \cdot \boldsymbol{t}_{\partial f}\|_{-\frac{1}{2},\partial f}\lesssim h_f^{-\frac{1}{2}}\|\v\|_{f} +h_f^{\frac{1}{2}}\|\nabla_{f}\times  \v\|_{f},\quad \forall \v\in\mathbb{E}^e_{r-1,k+1}(f).
	\end{equation}
	Although $\mathbb{E}^e_{k-1,k+1}(f)$ is not a polynomial space, an inverse inequality nevertheless holds on it. This inequality is crucial for establishing the stability result on $\boldsymbol{V}^e_{k-1,k+1}(K)$; see \cite[Lemma 5.4]{VEMforGener}:
	\begin{lemma}\label{EeTrace}
		The following inverse inequality holds true:
		\begin{equation*}
			\|\boldsymbol{v}\|_f\lesssim h_f^{-\frac{1}{2}}\|\boldsymbol{v}\|_{-\frac{1}{2},f},\quad \forall \boldsymbol{v} \in \mathbb{E}^e_{k-1,k+1}(f),
		\end{equation*}
		which, together with \eqref{HcurlTrace}, yields
		\begin{equation*}
			\sum_{f\in\partial K}h_f^\frac{1}{2}\|\v_{\tau}\|_f\lesssim  \|\boldsymbol{v}\|_K +h_K\|\nabla\times \v\|_K, \quad \forall \v\in\boldsymbol{V}^e_{k-1,k+1}(K).
		\end{equation*}
	\end{lemma}
	Although we cannot extend the inverse inequality to $\mathbb{E}_{r-1,k+1}(f)$, Lemma \ref{EeTrace} enables us to generalize the trace inequality to the local space $\boldsymbol{V}_{k-1,k+1}(K)$ as follows.
	\begin{lemma}\label{VkTraceIneq}
		The following trace inequality holds:
		\begin{equation*}
			\sum_{f\in \partial K}h_f^{\frac{1}{2}}\|\v_{\tau}\|_f \lesssim \|\v\|_{h,\boldsymbol{V}(K)}, \quad \forall \boldsymbol{v}\in \boldsymbol{V}_{k-1,k+1}(K).
		\end{equation*}
		\begin{proof}
			For any given $\v\in \boldsymbol{V}_{r-1,k+1}(K)$, there exists a decomposition
			\begin{equation*}
				\v= \v^0+\v^e,
			\end{equation*}
			satisfying the following properties on each face $f\in \partial K$:
			\begin{gather*}
				\nabla_{f}\times  \v^0_{\tau} \in \overline{\mathbb{B}}_k(f)\cap L_0^2(f),\quad \nabla\cdot \boldsymbol{v}_{\tau}^0=0,\quad \boldsymbol{v}_{\tau}^0\cdot\boldsymbol{t}_{e}=0, \quad \forall e\in \partial f,\\
				\nabla_{f}\times  \v^e_{\tau} \in P_0(f),\quad \nabla\cdot \boldsymbol{v}_{\tau}^e\in P_{k-1}(f),\quad \boldsymbol{v}_{\tau}\cdot\boldsymbol{t}_{e}\in P_{k-1}(e), \quad \forall e\in \partial f.
			\end{gather*}
			In the face space $\mathbb{E}_{k-1,k+1}(f)$, the decomposition is uniquely determined due to the well-posedness of both curl-div systems \cite{VEMfordivcurl}. Note that $\boldsymbol{v}^0$ is divergence-free and tangential trace-free on each face $f$, while $\boldsymbol{v}^e_{\tau} \in \mathbb{E}^e_{k-1,k+1}(f)$. 
			Then, the two-dimensional Friedrichs inequality yields
			\begin{equation*}
				\|\v^0_{\tau}\|_f \lesssim h_f \|\nabla_{f}\times   \v_{\tau}^0\|_f, \quad \forall f\in\partial K.
			\end{equation*}
			Combining this with Lemma \ref{EeTrace} and the trace inequalities \eqref{TraceTangetial} and \eqref{H1trace}, we have
			\begin{align*}
				\sum_{f\in \partial K}h_f^{\frac{1}{2}}\|\v_{\tau}\|_f &\le   \sum_{f\in \partial K}h_f^{\frac{1}{2}}(\|\boldsymbol{v}^0_{\tau}\|_f+
				\|\boldsymbol{v}^e_{\tau}\|_f)  \\
				&\lesssim   \sum_{f\in \partial K}h_f^{\frac{3}{2}}\|\nabla_{f}\times   \boldsymbol{v}^0_{\tau}\|_f +
				\sum_{f\in \partial K}\| \boldsymbol{v}^e_{\tau}\|_{-\frac{1}{2},f}\\
				&\lesssim \| \boldsymbol{v}^e\|_K + h_K\|\nabla \times \boldsymbol{v}^e\|_K + h_K\|\nabla \times \boldsymbol{v}^0\|_K+h_K^{2}|\nabla\times \boldsymbol{v}^0|_{1,K}.
			\end{align*}
			The proof is complete.
		\end{proof}
	\end{lemma}
	\begin{theorem}\label{LemmaS2Stab}
		There exist two positive constants $\beta_*$ and $\beta^* $ independent of $h_K$ such that  
		\begin{equation}\label{stab2}
			\begin{aligned}
				\beta_*\|\boldsymbol{v}_h\|^2_{K}\le S_2^K(\boldsymbol{v}_h, \boldsymbol{v}_h)\le &\beta^* \|\v\|_{h,\boldsymbol{V}(K)}^2,\\
				&\forall \boldsymbol{v}_h \in \boldsymbol{V}_{k-1, k+1}(K)\cap \ker(\boldsymbol{\Pi}^{0, K}_{k-1}).
			\end{aligned}
		\end{equation}
	\end{theorem}
	\begin{proof}
		For any $\boldsymbol{v}_h \in \boldsymbol{V}_{k-1, k+1}(K)\cap \ker(\boldsymbol{\Pi}^{0, K}_{k-1})$, applying Lemma \ref{lowbound} obtains
		\begin{equation}\label{aux1}
			\|\boldsymbol{v}_h\|_{K} \lesssim h_K\|\nabla \times \boldsymbol{v}_h\|_{K}+
			\sum_{f\in \partial K}\left( h_f^{\frac{3}{2}}\|\nabla_{f}\times  \boldsymbol{v}_{h,\tau}\|_{f}+ h_f\|\v_{h,\tau}\cdot \boldsymbol{t}_{\partial f}\|_{\partial f} +\|\boldsymbol{\Pi}^0_{+,k-1}\boldsymbol{v}_{h,\tau}\|_f \right).
		\end{equation}
		Sine $ \nabla \times \boldsymbol{v}_h \in \boldsymbol{W}_k(K)$,  the stability estimates \eqref{LowerBoundH1} yields
		\begin{equation*}
			\begin{aligned}
				|\nabla \times \boldsymbol{v}_h|^2_{1, K}
				\lesssim h_K^{-2}\|\boldsymbol{\Pi}^{\nabla, K}_k \nabla \times\v_h\|^2_K +\sum_{f\in\partial K} (h_f^{-1}\|\boldsymbol{\Pi}^{0, f}_k\nabla \times\v_h\|^2_f +\|\nabla \times\v_h\|^2_{ \partial f}).
			\end{aligned}
		\end{equation*}
		Then we get the $L^2$-norm estimate
		\begin{equation}\label{aux2}
			\begin{aligned}
				\|\nabla\times \v_h\|_{K}&\le \|\nabla\times \v_h-\boldsymbol{\Pi}^{0, K}_{k}\nabla\times \v_h\|_K+\|\boldsymbol{\Pi}^{0, K}_{k}\nabla\times \v_h\|_K\\
				&\lesssim h_K|\nabla \times\v_h|_{1, K}+\|\boldsymbol{\Pi}^{0, K}_{k}\nabla\times \v_h\|_K\\
				&\lesssim \|\boldsymbol{\Pi}^{0, K}_k \nabla \times\v_h\|_K +\sum_{f\in\partial K} (h_f^{\frac{1}{2}}\|\boldsymbol{\Pi}^{0, f}_k\nabla \times\v_h\|_f +h_f\|\nabla \times\v_h\|_{ \partial f}).
			\end{aligned}
		\end{equation} 
		From \cite[Lemma 4.3]{ZJQQuadCurl}, we obtain the following estimate in $\mathbb{E}^e_{k-1,k+1}(f)$:
		\begin{equation*}
			\|\nabla_{f}\times  \v_{h,\tau}\|_f \lesssim \|\Pi^{0,f}_{k} \operatorname*{rot}\boldsymbol{v}_{h,\tau}\|_f + h_f^{\frac{1}{2}}\|\operatorname*{rot}\boldsymbol{v}_{h,\tau}\|_{\partial f}
		\end{equation*}
		which, together with \eqref{aux1} and  \eqref{aux2}, leads to 
		\begin{equation}\label{aux3}
			\begin{aligned}
				\|\boldsymbol{v}_h\|_K\lesssim h_K\|\boldsymbol{\Pi}^{0, K}_k \nabla \times\boldsymbol{v}_h\|_K+\sum_{f\in \partial K}\left(h_f^{\frac{3}{2}} \|\boldsymbol{\Pi}^{0, f}_k \nabla \times\boldsymbol{v}_h\|_f +h_f^{\frac{1}{2}}\|\boldsymbol{\Pi}^{0,f}_{+,k-1}\boldsymbol{v}_{h,\tau}\|_f \right.\\
				\left.+h_f^2 \|\nabla \times\boldsymbol{v}_{h}\|_{\partial f}+h_f\|\boldsymbol{v}_{h,\tau}\cdot \boldsymbol{t}_{\partial f}\|_{\partial f}\right).
			\end{aligned}
		\end{equation}
		This establishes the lower bound in \eqref{stab2}.
		\par 
		Next, we estimate the five terms on the right-hand side of \eqref{aux3} to derive the upper bound.
		From the stability of projection $\boldsymbol{\Pi}_{k}^{0,K}$ and  $\boldsymbol{\Pi}_{k}^{0,f}$, along with the trace inequality \eqref{H1trace}, we obtain
		\begin{equation}\label{aux5}
			\begin{aligned}
				\|\boldsymbol{\Pi}^{0,K}_{k}\nabla \times \v_h\|_K+\sum_{f\in\partial K}h_f^{\frac{3}{2}}\|\boldsymbol{\Pi}^{0, f}_k\nabla \times \v_h \|_f
				&\lesssim
				h_K\|\nabla \times\v_h\|_{ K}+ h_K^{\frac{3}{2}}\|\nabla \times\v_h\|_{\partial K}	\\
				&\lesssim h_K\|\nabla \times\v_h\|_{ K}+h_K^2|\nabla \times \vh|_{1, K}.
			\end{aligned}
		\end{equation}
		The trace inequality \eqref{H1trace} and inverse inequality in $\widehat{\mathbb{B}}_k(f)$ or $\overline{\mathbb{B}}_k(f)$ \cite{ChenLong2018}  lead to  
		\begin{equation}
			\begin{aligned}
				\sum_{f\in\partial K}h_f^2\|\nabla\times \vh\|_{\partial f} &\lesssim \sum_{f\in \partial K}h_f^2 (h_f^{-\frac{1}{2}}\|\nabla\times \vh\|_{ f}+h_f^{\frac{1}{2}}|\nabla\times \vh|_{1, f}) 
				\\ &\lesssim h_K^{\frac{3}{2}}\|\nabla \times\v_h\|_{\partial K} \lesssim h_K\|\nabla\times \vh\|_{K}+h_K^2|\nabla\times \vh|_{1, K}.
			\end{aligned}
		\end{equation}
		Finally, using  the trace inequality \eqref{TraceTangetial}  and \eqref{H1trace},
		along with Lemma \ref{VkTraceIneq}, we have
		\begin{equation}\label{PolyIn}
			\begin{aligned}
				\sum_{f\in\partial K}&(h_f^{\frac{1}{2}}\|\Pi^0_{+,k-1}\boldsymbol{v}_{h,\tau}\|_f+h_f\|\boldsymbol{v}_{h,\tau}\cdot \boldsymbol{t}_{\partial f}\|_{\partial f})\\
				&\lesssim \sum_{f\in\partial K}(h_f^{\frac{1}{2}}\|\boldsymbol{v}_{h,\tau}\|_{ f}+h_f^{\frac{3}{2}}\|\nabla_{f}\times  \boldsymbol{v}_{h,\tau}\|_{ f})\\
				&\lesssim \|\boldsymbol{v}_h\|_K+h_K\|\nabla\times \vh\|_{K}+h_K^2|\nabla\times \vh|_{1, K}. 
			\end{aligned}
		\end{equation}
		Thus, combining \eqref{aux5}-\eqref{PolyIn}, we conclude that the upper bound in \eqref{stab2} holds.
	\end{proof}

	\subsection{The discrete problem}
	By imposing homogeneous boundary conditions, we define the following discrete spaces:
	\begin{gather*}
		U_h:= U_{k}(\Omega)\cap H_0^1(\Omega), \quad\boldsymbol{V}_h:=\boldsymbol{V}_{k-1,k+1}(\Omega)\cap \boldsymbol{V}_0(\Omega), \\\boldsymbol{W}_h:=\boldsymbol{W}_k(\Omega)\cap \boldsymbol{H}_0^1(\Omega), \quad Q_h:=Q_{k-1}(\Omega)\cap L_0^2(\Omega).
	\end{gather*}
	These spaces form an exact discrete complex:
	\begin{equation}\label{homdisComplex}
		0 \stackrel{}{\longrightarrow} U_h\stackrel{\nabla}{\longrightarrow} \boldsymbol{V}_{h} \stackrel{\nabla \times }{\longrightarrow} \boldsymbol{W}_{h} \stackrel{\nabla \cdot }{\longrightarrow} Q_{h}\longrightarrow 0.
	\end{equation}
	The virtual element scheme for the grad-curl problem \eqref{VectorPotentialProblem} finds $(\boldsymbol{\psi}_h,  \lambda_h)\in \boldsymbol{V}_h  \times U_h$ such that
	\begin{equation}\label{disProblem}
		\left\{\begin{aligned}
			a_{h}\left(\nabla\times\boldsymbol{\psi}_{h}, \nabla\times\boldsymbol{\phi}_{h}\right)+b_{h}\left(\boldsymbol{\phi}_{h}, \nabla \lambda_{h}\right) & =\frac{1}{\nu}\left(\boldsymbol{f}_{h}, \nabla \times  \boldsymbol{\phi}_{h}\right), \quad \forall \boldsymbol{\phi}_{h} \in \boldsymbol{V}_h, \\
			b_{h}\left(\boldsymbol{\psi}_{h}, \nabla q_h \right) & =0, \quad \forall q_{h} \in U_h,
		\end{aligned}\right.
	\end{equation}
	where $\boldsymbol{f}_h|_K = \boldsymbol{\Pi }^{0,K}_{k} \boldsymbol{f}$ approximates $\boldsymbol{f}$ with the optimal convergence: 
	\begin{equation}\label{rightHandError}
		\|\boldsymbol{f}-\boldsymbol{f}_h\|\lesssim h^{s}\|\boldsymbol{f}\|_{s}, \quad s\le k.
	\end{equation}
	\par 
	We define the subspace
	\begin{equation*}
		\boldsymbol{X}_h =\{\boldsymbol{v}_h \in \boldsymbol{V}_{h}(\Omega): b_h(\boldsymbol{v}_h,\nabla q_h)=0,\, \forall q_h\in U_{h}(\Omega)\}
	\end{equation*}
	and present the well-posedness of the discrete problem as follows.
	\begin{theorem}\label{ExiUni}
		The discrete problem \eqref{disProblem} has a unique solution with $\lambda_h=0$.
	\end{theorem}
	\begin{proof}
		We first prove that the discrete Friedrichs inequality on $\boldsymbol{X}_h$: for any given $\v_h \in \boldsymbol{X}_h$, there exists a constant $\alpha $ such that 
		\begin{equation}\label{disFre}
			a_h(\nabla\times\boldsymbol{v}_h,\nabla\times\boldsymbol{v}_h)\ge \alpha \|\boldsymbol{v}_h\|_{\boldsymbol{V}(\Omega)}.
		\end{equation}		  
		Consider the unique solution $w\in H_0^1(\Omega)$ to the Poisson equation:
		\begin{equation*}
			\Delta w= \nabla \cdot \boldsymbol{v}_h \quad \text{in } H^{-1}(\Omega).
		\end{equation*}
		Define $\boldsymbol{z}=\boldsymbol{v}_h -\nabla w$. Then $\boldsymbol{z}$ satisfies weakly
		\begin{equation*}
			\nabla \times\boldsymbol{z} =\nabla \times\boldsymbol{ v}_h , \quad \nabla \cdot \boldsymbol{z}=0\quad \text{in } \Omega, \quad \boldsymbol{z}\times \boldsymbol{n}=0 \quad \text{on } \Gamma,
		\end{equation*}
		which implies $ \boldsymbol{z}\in \boldsymbol{X}(\Omega)$.
		By the Friedrichs inequality \eqref{FreIneq},  it holds for $s>\frac{1}{2}$
		\begin{equation}\label{aux6}
			\|\boldsymbol{z}\|_{s}\lesssim  \|\nabla \times\boldsymbol{z}\|=\|\nabla\times \v_h\|.
		\end{equation}
		Hence, $\boldsymbol{I}_h \boldsymbol{z}$ is well-defined, and the commutativity between $\boldsymbol{I}_h$ and $\boldsymbol{J}_h$ in diagram \eqref{CommutDigram} establishes that
		\begin{equation*}
			\nabla\times (\boldsymbol{I}_h\boldsymbol{z} -\boldsymbol{v}_h) = \boldsymbol{J}_h \nabla\times \boldsymbol{z} -\nabla\times \v_h=\boldsymbol{J}_h \nabla\times\boldsymbol{v}_h -\nabla\times \v_h =0.
		\end{equation*}
		Combined with the exactness of discrete complex \eqref{homdisComplex}, there exists $q_h \in U_h$ satisfying
		\begin{equation}\label{Ihz}
			\nabla q_h = \boldsymbol{v}_h- \boldsymbol{I}_h \boldsymbol{z}.
		\end{equation}
		Applying \eqref{aux6} and the interpolation error estimate \eqref{Intpolation1} yields
		\begin{equation}\label{auxIh}
			\|\boldsymbol{I}_h\boldsymbol{z}\| \le \|\boldsymbol{z}\| +\|\boldsymbol{z} -\boldsymbol{I}_h \boldsymbol{z}\|\lesssim  \|\nabla \times\boldsymbol{v}_h\|_1.
		\end{equation}
		Using \eqref{Ihz} and the fact that $\boldsymbol{v}_h \in \boldsymbol{X}_h$, we obtain 
		\begin{equation*}
			b_h(\boldsymbol{v}_h, \boldsymbol{v}_h)=b_h( \boldsymbol{v}_h,\boldsymbol{I}_h\boldsymbol{z}+\nabla q_h)=b_h(\boldsymbol{v}_h, \boldsymbol{I}_h\boldsymbol{z}).
		\end{equation*}
		The coercivity \eqref{bhCoer} and continuity \eqref{bhConti} lead to
		\begin{equation*}
			\|\boldsymbol{v}_h\|^2\lesssim b_h(\boldsymbol{v}_h, \boldsymbol{v}_h)=b_h(\boldsymbol{v}_h, \boldsymbol{I}_h\boldsymbol{z})\lesssim  \|\boldsymbol{v}_h\|_{\boldsymbol{V}(\Omega)}\|\boldsymbol{I}_h\boldsymbol{z}\|_{\boldsymbol{V}(\Omega)},
		\end{equation*}
		which, combined with   \eqref{auxIh}, yields 
		\begin{equation*}
			\|\boldsymbol{v}_h\|^2\lesssim \|\boldsymbol{v}_h\|_{\boldsymbol{V}(\Omega)}\|\nabla \times\boldsymbol{v}_h\|_1.
		\end{equation*}
		This inequality, together with the stability \eqref{ahStab}, establishes the discrete Friedrichs inequality  \eqref{disFre}.
		\par
		We now prove the discrete inf-sup condition:
		there exists a positive constant $\beta $ independent of $h$ such that 
		\begin{equation}\label{disLBB}
			\sup _{\boldsymbol{\phi}_{h} \in \boldsymbol{V}_h /\{0\}} \frac{b_h\left(\boldsymbol{\phi}_h, \nabla  q_h\right)}{\left\|\boldsymbol{\phi}_{h}\right\|_{\boldsymbol{V}(\Omega)}} \geq \beta\|\nabla  q_{h}\|,  \quad \forall q_{h} \in U_h.
		\end{equation}
		To prove this, take $\boldsymbol{\phi}_h =\nabla q_h$ for any $q_h\in U_h$ and use the coercivity \eqref{bhCoer}, which  gives
		\begin{equation*}
			b_h(\boldsymbol{\phi}_h,\nabla q_h)= b_h(\nabla q_h, \nabla q_h) \gtrsim \|\nabla q_h\|^2.
		\end{equation*}
		Combined with $\|\boldsymbol{\phi}_h\|_{\boldsymbol{V}(\Omega)} = \|\nabla q_h\|$, the discrete inf-sup condition \eqref{disLBB} holds.
		\par 
		In conclusion,  the discrete problem \eqref{disProblem} is well-posed.
		Furthermore, by choosing
		$\v_h = \nabla \lambda_h$ in the first equation of \eqref{disProblem} and applying the coercivity \eqref{bhCoer},  we deduce that $\lambda_h =0$.
	\end{proof}
	\begin{remark}\label{disEqv}
		We introduce
		the discrete velocity-pressure pair $(\bu_h, p_h) \in \boldsymbol{W}_h\times Q_h$  as the unique solution to the virtual element discretization of the Stokes problem \eqref{StokesProblem} from \cite{VEMforStokes}: 
		\begin{equation}\label{DisStokesproblem}
			\begin{aligned}
				\nu a_h(\bu_h,\v_h) +(\nabla\cdot \boldsymbol{v}_h, p_h) &= (\boldsymbol{f}_h, \boldsymbol{v}_h), \quad\forall \v_h \in \boldsymbol{W}_h\\
				(\nabla\cdot \bu_h,q_h)&=0. \quad \forall q_h \in U_h.
			\end{aligned}
		\end{equation}
		Let $(\boldsymbol{\psi}_h,\lambda_h)\in \boldsymbol{V}_h\times U_h$ be the unique solution of \eqref{disProblem}. The exactness of the discrete sequence \eqref{homdisComplex} implies the relations:
		\begin{equation*}
			\nabla \times \boldsymbol{X}_h = \boldsymbol{W}_h\cap \ker(\nabla \cdot), \text{ and  } \nabla\cdot \boldsymbol{W}_h= Q_h.
		\end{equation*}
		It then follows that the discrete Stokes problem \eqref{DisStokesproblem} is equivalent to finding $\bu_h \in \boldsymbol{W}_h \cap \ker(\nabla \cdot)$ satisfying
		\begin{equation*}
			\nu a_h(\bu_h, \nabla\times \boldsymbol{\phi}_h)=(\boldsymbol{f}_h,\nabla\times \boldsymbol{\phi}_h), \quad \forall \boldsymbol{\phi}_h \in \boldsymbol{X}_h.
		\end{equation*}
		This formulation is identical to the discrete grad-curl problem \eqref{disProblem} with $\lambda_h=0$, which implies that $\boldsymbol{\psi}_h$ acts as the discrete vector potential for the velocity $\bu_h$ in \eqref{DisStokesproblem}, in the sense that
		$$ \nabla\times \boldsymbol{\psi}_h =\bu_h.$$
	\end{remark}	
	\begin{remark}
		The discrete scheme \eqref{disProblem} with the right-hands $(\boldsymbol{j}_h,\boldsymbol{\phi}_h)$ can also be applied to the quad-curl problem \eqref{primalForm},
		where $\boldsymbol{j}_h|_K = \boldsymbol{\Pi}_{k-1}^{0,K}\boldsymbol{j}$ approximates $\boldsymbol{j}$ with optimal convergence:
		\begin{equation*}
			\|\boldsymbol{j}-\boldsymbol{j}_h\| \lesssim h^s\|\boldsymbol{j}\|_s, \quad s\le k.
		\end{equation*}
		While the convergence results (Theorems \ref{TheoremConv1} and \ref{allNormError}) in the next subsection are presented for the grad-curl problem, a parallel analysis yields analogous results for the quad-curl problem.
	\end{remark}

	\subsection{Convergence analysis}
	The standard Dupont–Scott theory \cite{Polyesti} provides the following local approximation results.
	\begin{lemma}
		For all $\v \in \boldsymbol{H}^s(\Omega)$ satisfying $\nabla \times\v \in \boldsymbol{H}^{s+1}(\Omega)$,
		there exist  $\v^{\pi}_{k-1}\in \boldsymbol{P}^{\text{dc}}_{k-1}(\Omega)$ and $ \v^{\pi}_{k+1} \in \boldsymbol{P}^{\text{dc}}_{k+1}(\Omega)$ with $0<s\le k$ such that 
		\begin{align}
			\label{polyappro}
			|\v -\v^{\pi}_{k-1}|_{m, K}&\lesssim h_K^{l-m}|\v|_{l, K}, \quad 0\le m\le l\le s, \\
			\label{divpolyappro}
			|\nabla \times(\v-\v^{\pi}_{k+1})|_{m, K}&\lesssim h_K^{l-m}|\nabla \times\v|_{l, K}, \quad 0\le m \le  l\le s+1,
		\end{align}
		where  $\boldsymbol{P}_k^{\text{dc}}(\Omega) = \left\{\boldsymbol{v}\in \boldsymbol{L}^2(\Omega); \boldsymbol{v}|_K \in \boldsymbol{P}_k(K), \forall K \in \mathcal{T}_h \right\}$.
	\end{lemma}
	\begin{theorem}\label{TheoremConv1}
		Suppose that $(\boldsymbol{\psi}, \lambda) \in \boldsymbol{V}_0(\Omega) \times H^1_0(\Omega) $ is the solution of the  problem \eqref{VectorPotentialProblem} with  $\lambda=0$, 
		and  $(\boldsymbol{\psi}_h, \lambda_h) \in \boldsymbol{V}_h\times U_h$  the solution of the discrete scheme \eqref{disProblem} with $\lambda_h=0$.  There holds 
		\begin{equation*}
			\begin{aligned}
				\|\boldsymbol{\psi}-\boldsymbol{\psi}_h\|_{\boldsymbol{V}(\Omega)}&\lesssim  \inf _{\boldsymbol{z}_{h} \in \boldsymbol{X}_{h}}\left\|\boldsymbol{\psi}-\boldsymbol{z}_{h}\right\|_{\boldsymbol{V}(\Omega)} \\
				\nonumber
				& \quad +\inf _{\boldsymbol{v}^{\pi}_{k+1} \in \boldsymbol{P}_{k+1}^{\text{dc}}(\Omega)}\left|\nabla \times\left(\boldsymbol{\psi}-\boldsymbol{v}^{\pi}_{k+1}\right)\right|_1+\frac{1}{\nu}\left\|\boldsymbol{f}-\boldsymbol{f}_{h}\right\|. 
			\end{aligned}
		\end{equation*}
	\end{theorem}
	\begin{proof}
		For any $\boldsymbol{z}_h\in \boldsymbol{X}_h $ and $\boldsymbol{v}^{\pi}_{k+1} \in \boldsymbol{P}^{dc}_{k+1}(\Omega)$, setting $\boldsymbol{\delta}_h = \boldsymbol{z}_h -\boldsymbol{u}_h \in \boldsymbol{X}_h$, we obtain  
		\begin{equation*}
			\begin{aligned}
				\|\boldsymbol{z}_h& -\boldsymbol{\psi}_h\|^2_{\boldsymbol{V}(\Omega)}\\
				&\lesssim  a_h(\boldsymbol{z}_h-\boldsymbol{\psi}_h, \boldsymbol{\delta}_h)
				\quad (\text{use the  coercivity } \eqref{ahStab})\\
				&=\sum_{K\in \mathcal{T}_h}(a_h^K(\boldsymbol{z}_h-\boldsymbol{v}^{\pi}_{k+1}, \boldsymbol{\delta}_h)+a_h^K(\boldsymbol{v}^{\pi}_{k+1}, \boldsymbol{\delta}_h)) \nonumber \\
				&\quad -\frac{1}{\nu}(\boldsymbol{f}_{h}, \nabla \times\boldsymbol{\delta}_h) \quad (\text{use } \eqref{disProblem} \text{ with }\lambda_h=0)\\
				&=\sum_{K\in \mathcal{T}_h}\left(a_h^K(\boldsymbol{z}_h-\boldsymbol{v}^{\pi}_{k+1}, \boldsymbol{\delta}_h)+(\nabla\nabla \times\boldsymbol{v}^{\pi}_{k+1}, \nabla\nabla \times\boldsymbol{\delta}_h)_K\right) \nonumber\\
				&\quad-\frac{1}{\nu}(\boldsymbol{f}_{h}, \nabla \times\boldsymbol{\delta}_h)\quad (\text{use the consistency } \eqref{ahConsist})\\
				&=\sum_{K\in \mathcal{T}_h}\left(a_h^K(\boldsymbol{z}_h-\boldsymbol{v}^{\pi}_{k+1}, \boldsymbol{\delta}_h)+(\nabla\nabla \times(\boldsymbol{v}^{\pi}_{k+1}-\boldsymbol{\psi}), \nabla\nabla \times\boldsymbol{\delta}_h)_K\right)\nonumber \\
				&\quad+\frac{1}{\nu}(\boldsymbol{f}-\boldsymbol{f}_h, \nabla \times\boldsymbol{\delta}_h) \quad (\text{use } \eqref{VectorPotentialProblem} \text{ with }\lambda=0).
			\end{aligned}
		\end{equation*}
		It follows from coercivity \eqref{ahStab} that 
		\begin{equation*}
			\|\boldsymbol{z}_h-\boldsymbol{\psi}_h\|_{\boldsymbol{V}(\Omega)}\lesssim (|\nabla \times(\boldsymbol{\psi-z}_h)|_1+|\nabla \times(\boldsymbol{\psi}-\boldsymbol{v}^{\pi}_{k+1})|_1+\frac{1}{\nu}\|\boldsymbol{f}-\boldsymbol{f}_h\|).
		\end{equation*} 
		By using the triangle inequality,  we get the desired result
		\begin{equation*}
			\|\boldsymbol{\psi} -\boldsymbol{\psi}_h\|_{\boldsymbol{V}(\Omega)}\lesssim (\|\boldsymbol{\psi-z}_h\|_{\boldsymbol{V}(\Omega)}+|\nabla\times(\boldsymbol{\psi}-\boldsymbol{v}^{\pi}_{k+1})|_1+\frac{1}{\nu}\|\boldsymbol{f}-\boldsymbol{f}_h\|).
		\end{equation*}
		Since $\boldsymbol{z}_h$ and $\boldsymbol{v}^{\pi}_{k+1}$ are arbitrary, the proof is complete.
	\end{proof}
	\begin{theorem}\label{allNormError}
		Suppose that $(\boldsymbol{\psi},  \lambda) \in \boldsymbol{V}_0(\Omega) \times H^1_0(\Omega) $ is the solution of the  problem \eqref{VectorPotentialProblem} with  $\lambda=0$, 
		and let  $(\boldsymbol{\psi}_h, \lambda_h) \in \boldsymbol{V}_h\times U_h$ be the solution of the discrete scheme \eqref{disProblem} with  $\lambda_h=0$.
		Then for $\boldsymbol{f}\in \boldsymbol{H}^{s}(\Omega)$ and $\boldsymbol{\psi} \in \boldsymbol{H}^{s}(\Omega)$ with $\nabla\times\boldsymbol{\psi} \in \boldsymbol{H}^{s+1}(\Omega)$, $\frac{1}{2}<s\le k$, we have 
		\begin{align}\label{GradCurlNormConv}
			\|\boldsymbol{\psi}-\boldsymbol{\psi}_h\|_{\boldsymbol{V}(\Omega)} \lesssim h^{s}(\|\boldsymbol{\psi}\|_{s}+\|\nabla\times\boldsymbol{\psi}\|_{s+1}+\frac{1}{\nu}\|\boldsymbol{f}\|_s).
		\end{align}
	\end{theorem}
	\begin{proof}
		From the coercivity \eqref{bhCoer}, we have the following orthogonal decomposition with
		respect to the discrete inner products $b_h(\cdot,\cdot)$
		\begin{equation*}
			\boldsymbol{V}_h =\nabla U_h \oplus^\bot \boldsymbol{X}_h.
		\end{equation*}
		Hence, there exists $\phi_h \in U_h$ such that $\boldsymbol{I}_h \boldsymbol{\psi} - \nabla \phi_h \in \boldsymbol{X}_h$, and it follows that
		\begin{equation}\label{curlLBB}
			|\phi_h|_1\lesssim \sup_{q_h \in U_k(\Omega)/\{0\} }\frac{b_h(\nabla \phi_h , \nabla q_h)}{\|\nabla q_h\|} = \sup_{q_h \in U_k(\Omega)/ \{0\} }\frac{b_h(\boldsymbol{I}_h \boldsymbol{\psi}, \nabla q_h)}{\|\nabla q_h\|}.
		\end{equation}
		Let $\boldsymbol{\psi}^{\pi}_{k-1} \in \boldsymbol{P}_{k-1}^{dc}(\Omega)$ be the approximation to $\boldsymbol{\psi}$ satisfying \eqref{polyappro}.
		Using the consistency \eqref{bhConsist}, we obtain
		\begin{equation*}
			\begin{aligned}
				b_h(\boldsymbol{I}_h\boldsymbol{\psi}, \nabla \phi_h)=(\boldsymbol{\psi}^{\pi}_{k-1}-\boldsymbol{\psi}, \nabla \phi_h)+b_h(\boldsymbol{I}_h\boldsymbol{\psi}-\boldsymbol{\psi}^{\pi}_{k-1}, \nabla \phi_h).
			\end{aligned}
		\end{equation*}
		For the first term on the rignt-hand of the above equation, it holds
		\begin{equation}\label{aux7}
			(\boldsymbol{\psi}^{\pi}_{k-1}-\boldsymbol{\psi}, \nabla \phi_h)\lesssim h^s\|\boldsymbol{\psi}\|_s|\phi_h|_1.
		\end{equation}
		For the second one, according to the continuity \eqref{bhConti},  definition of the scaled norm $\|\cdot\|_{h,\boldsymbol{V}(K)}$, interpolation errors \eqref{Intpolation1}, \eqref{Intpolation2}, \eqref{Intpolation3}, we get
		\begin{equation*}
			\begin{aligned}
				b_h^K(&\boldsymbol{I}_h\boldsymbol{\psi}-\boldsymbol{\psi}^{\pi}_{k-1},  \nabla \phi_h)\nonumber\\
				&\lesssim \|\boldsymbol{I}_h\boldsymbol{\psi}-\boldsymbol{\psi}^{\pi}_{k-1}\|_{h,\boldsymbol{V}(K)}\|\nabla\phi_h\|_{h,\boldsymbol{V}(K)}\\
				&\lesssim (\|\boldsymbol{\psi}-\boldsymbol{I}_h\boldsymbol{\psi}\|_{h,\boldsymbol{V}(K)}+\|\boldsymbol{\psi}- \boldsymbol{\psi}^{\pi}_{k-1}\|_{h,\boldsymbol{V}(K)})|\phi_h|_{1,K}\\
				&\lesssim\left(\|\boldsymbol{\psi}-\boldsymbol{I}_h\boldsymbol{\psi}\|_K+h_K\|\nabla \times (\boldsymbol{\psi}-\boldsymbol{I}_h\boldsymbol{\psi})\|_K+h_K^2|\nabla \times (\boldsymbol{\psi}-\boldsymbol{I}_h\boldsymbol{\psi})|_{1,K}\right.\\
				&\quad+ \left. \|\boldsymbol{\psi}-\boldsymbol{\psi}_{k-1}^{\pi}\|_K+h_K|\boldsymbol{\psi}-\boldsymbol{\psi}_{k-1}^{\pi}|_{1,K}+h_K^2|\boldsymbol{\psi}-\boldsymbol{\psi}_{k-1}^{\pi}|_{2,K} \right)|\phi_h|_{1,K}\\
				&\lesssim h_K^s(\|\boldsymbol{\psi}\|_{s,K}+\|\nabla \times \boldsymbol{\psi}\|_{s+1,K})|\phi_h|_{1,K},
			\end{aligned}
		\end{equation*}
		which,  together with \eqref{aux7} and \eqref{curlLBB}, yields
		\begin{equation*}
			| \phi_h|_1\lesssim h^s(\|\boldsymbol{\psi}\|_s+\|\nabla \times\boldsymbol{\psi}\|_{s+1}).
		\end{equation*}
		Then we have
		\begin{equation*}
			\begin{aligned}
				\inf _{\boldsymbol{z}_{h} \in \boldsymbol{Z}_{h}}\left\|\boldsymbol{\psi}-\boldsymbol{z}_{h}\right\|_{\boldsymbol{V}(\Omega)}&\le \|\boldsymbol{\psi}-\boldsymbol{I}_h\boldsymbol{\psi}\|_{\boldsymbol{V}(\Omega)}+| \phi_h|_1\\
				&\lesssim  h^s(\|\boldsymbol{\psi}\|_s+\|\nabla\times\boldsymbol{\psi}\|_{s+1}),
			\end{aligned}
		\end{equation*}
		which, combined with  \eqref{TheoremConv1}, implies \eqref{GradCurlNormConv}.
	\end{proof}

	Since $\nabla \times \boldsymbol{\psi}_h$ is the discrete solution to \eqref{DisStokesproblem} by Remark \ref{Regu}, the convergence result from \cite{VEMforStokes,daVeiga2017} yields the following estimates:
	\begin{theorem}\label{divH1est}
		Under the assumptions of Theorem \ref{allNormError},  we have
		\begin{align}\label{theor6}
			|\nabla \times(\boldsymbol{\psi}-\boldsymbol{\psi}_h)|_{1}&\lesssim  h^s(\|\nabla \times\boldsymbol{\psi}\|_{s+1}+\frac{1}{\nu}\|\boldsymbol{f}\|_s).\\
			\|\nabla \times(\boldsymbol{\psi}-\boldsymbol{\psi}_h)\|&\lesssim h^{s+1} (\|\nabla \times\boldsymbol{\psi}\|_{s+1}+\frac{1}{\nu}\|\boldsymbol{f}\|_s).
		\end{align}
	\end{theorem}
	\subsection{A pressure-decoupled, symmetric positive definite, reduced system}
	Having established in Remark \ref{disEqv} that the discrete grad–curl problem \eqref{disProblem} is the vector potential formulation of the discrete Stokes problem \eqref{DisStokesproblem}, we briefly explore the potential advantages of this approach for solving the Stokes problem \eqref{StokesProblem}, relative to the conventional velocity-pressure solver. 
	\par 
	Let $\boldsymbol{\psi}_h$ and $\bu_h$ be the discrete solutions of \eqref{disProblem} and \eqref{DisStokesproblem}, respectively.
	Since we are primarily interested $\nabla\times \boldsymbol{\psi}_h$, we can eliminate degrees of freedom that do not affect it, allowing us to introduce the following reduced spaces.
	Define 
	\begin{equation}
		\widehat{U}_h:=U_1(\Omega)\cap H_0^1(\Omega) \text{ and } \widehat{\boldsymbol{V}}_h:=\boldsymbol{V}_{0,k+1}(\Omega)\cap \boldsymbol{V}_0(\Omega).
	\end{equation}        
	In addition, we recall the reduced velocity space $\widehat{\boldsymbol{W}}_h$ and pressure space $\widehat{Q}_h $ introduced in \cite{VEMforStokes}. These spaces are constructed by requiring the divergence of the local space $\boldsymbol{W}_k(K)$  to lie in $P_0(K)$, and then gluing these local spaces to form $\widehat{\boldsymbol{W}}_h$,  equipped with the degrees of freedom $\mathbf{D}_{\boldsymbol{W}}$ except the divergence term \eqref{DW6}. The corresponding pressure space $\widehat{Q}_h$ is taken to be the piecewise constant space. In what follows, the discrete problems \eqref{disProblem} and \eqref{DisStokesproblem}  are discussed within these reduced spaces. 
	This yields the unique solutions $(\widehat{\boldsymbol{\psi}}_h, \widehat{\lambda}_h) \in \widehat{\boldsymbol{V}}_h \times \widehat{U}_h$ and $(\widehat{\bu}_h, \widehat{p}_h) \in \widehat{\boldsymbol{W}}_h \times \widehat{Q}_h$.
	\par 
	\begin{remark}\label{reducedRemark}
		The following reduced complex 
		\begin{equation}\label{ReducedComplex}
			0 \stackrel{}{\longrightarrow} \widehat{U}_h\stackrel{\nabla}{\longrightarrow} \widehat{\boldsymbol{V}}_{h} \stackrel{\nabla \times }{\longrightarrow} \widehat{\boldsymbol{W}}_{h} \stackrel{\nabla \cdot }{\longrightarrow} \widehat{Q}_{h}\longrightarrow 0
		\end{equation}
		is exact.
		The exactness $\nabla\cdot \widehat{\boldsymbol{W}}_h=\widehat{Q}_h$ ensures that the divergence $\nabla\cdot \widehat{\boldsymbol{u}}_h$ is exactly zero, whence we obtain $\widehat{\boldsymbol{u}}_h=\boldsymbol{u}_h$. In addition, the identity $\nabla\times \widehat{\boldsymbol{V}}_h =\nabla\times \boldsymbol{V}_h$ leads to $\nabla\times \widehat{\boldsymbol{\psi}}_h =\nabla\times \boldsymbol{\psi}_h$.  It then follows from Remark \ref{disEqv} that
		\[
		\nabla \times \widehat{\boldsymbol{\psi}}_h = \nabla \times \boldsymbol{\psi}_h = \bu_h = \widehat{\bu}_h.
		\]  
	\end{remark}
	Let $N_V$, $N_E$, $N_F$, and $N$ denote the number of vertices, edges, faces, and elements in the mesh $\mathcal{T}_h$, respectively. Recalling the local space dimensions from Section 3, and noting that vanishing boundary conditions are omitted, the following dimensional results hold for the global reduced spaces:
	\begin{align*}
		\dim(\widehat{U}_h)&=N_V,\\ 
		\dim(\widehat{\boldsymbol{V}}_h)&=3N_V+(3k-2)N_E+(2\dim(P_{k-2}(f)+\dim(P_m(f)-1)N_F+ \\
		&\quad+(3\dim(P_{k-2}(K))+\dim(P_{k-1}(K))+1)N,\\
		\dim(\widehat{\boldsymbol{W}}_h)&=3N_V + 3(k-1)N_E+ (2\text{dim}(P_{k-2}(f))+\text{dim}(P_m(f)))N_F\\
		&\quad + (3\text{dim}(P_{k-2}(K))-\text{dim}(P_{k-1}(K))+1)N,\\
		\dim(\widehat{Q}_h)&=N.
	\end{align*}
	\par
	Now we introduce the following mixed formulation:
	\begin{equation}\label{EqudisProblem}
		\left\{\begin{aligned}
			a_{h}\left(\nabla\times\widehat{\boldsymbol{\psi}}_h, \nabla\times\widehat{\boldsymbol{\phi}}_{h}\right)+b_{h}\left(\widehat{\boldsymbol{\phi}}_{h}, \nabla \widehat{\lambda}_h\right) & =\frac{1}{\nu}\left(\boldsymbol{f}_{h}, \nabla \times  \widehat{\boldsymbol{\phi}}_{h}\right), \quad \forall \boldsymbol{\phi}_{h} \in \widehat{\boldsymbol{V}}_h, \\
			b_{h}\left(\widehat{\boldsymbol{\psi}}_h, \nabla \widehat{q}_h \right)-\widetilde{b}_h(\nabla \widehat{\lambda}_h,\nabla \widehat{q}_h ) & =0, \quad \forall  \widehat{q}_h \in \widehat{U}_h,
		\end{aligned}\right.
	\end{equation}
	where $\widetilde{b}_h(\nabla \widehat{\lambda}_h,\nabla \widehat{q}_h )$ is defined by extracting only the diagonal entries of the matrix representing $b_h(\nabla \widehat{\lambda}_h,\nabla \widehat{q}_h)$.
	\begin{remark}
		The well-posedness of \eqref{EqudisProblem} is inherited from the stability of the primal scheme \eqref{disProblem}.  Furthermore, since $\widehat{\lambda}_h$ is identically zero, the pair $(\widehat{\boldsymbol{\psi}}_h, 0)$ is also a solution to \eqref{EqudisProblem}. Consequently, the problems \eqref{EqudisProblem} and \eqref{disProblem} are equivalent.
	\end{remark}
	The corresponding matrix formulation, analogous to the Maxwell system in \cite{Chen2018Pre}, reads:
	\begin{equation}
		\begin{pmatrix}
			E^\mathrm{T}AE & B_{\lambda}^\mathrm{T} \\
			B_{\lambda} & -\widetilde{M}
		\end{pmatrix}
		\begin{pmatrix}
			\widehat{\boldsymbol{\psi}}_h \\
			\widehat{\lambda}_h
		\end{pmatrix}
		=
		\begin{pmatrix}
			E^\mathrm{T}F \\
			0
		\end{pmatrix}.
	\end{equation}
	Here, $\widehat{\boldsymbol{\psi}}_h$ and $\widehat{\lambda}_h$ denote the vector representations of the corresponding virtual element solutions, $E$ is the matrix of the discrete curl operator, and $F$ is the vector with entries $\frac{1}{\nu}(\boldsymbol{f}_h, \boldsymbol{v}_h)$ for all $\boldsymbol{v}_h \in \widehat{\boldsymbol{W}}_h$. The diagonality of $\widetilde{M}$ permits an efficient implementation of the Schur complement system:
	\begin{equation}\label{VectPotenDis}
		(E^\mathrm{T}AE + B_{\lambda}^\mathrm{T} {\widetilde{M}}^{-1} B_{\lambda})\boldsymbol{\psi}_h = E^\mathrm{T}F.
	\end{equation}
	This system matrix can be interpreted as a discrete counterpart of the operator $-\nabla\times \Delta\nabla\times + \nabla\nabla\cdot$.
	
	Consequently, our final positive-definite system \eqref{VectPotenDis} has $\dim(\widehat{\boldsymbol{V}}_h) = N_E + \dim(\widehat{\boldsymbol{W}}_h) - N_F$ degrees of freedom. In contrast, the discrete velocity-pressure scheme
	\begin{equation}\label{VpDis}
		\begin{pmatrix}
			A & B_p^\mathrm{T} \\
			B_p & 0
		\end{pmatrix}
		\begin{pmatrix}
			\widehat{\boldsymbol{u}}_h \\
			\widehat{p}_h
		\end{pmatrix}
		=
		\begin{pmatrix}
			F \\
			0
		\end{pmatrix}
	\end{equation}
	requires $\dim(\widehat{\boldsymbol{W}}_h) + N$ degrees of freedom. Transformation from the saddle-point problem \eqref{VpDis} to the positive-definite system \eqref{VectPotenDis} is achieved at the cost of an additional stabilization term, leading to a net change of $N + N_F -N_E$ in the number of degrees of freedom. In particular, when $N_E < N + N_F$, our system exhibits fewer degrees of freedom while maintaining comparable accuracy, as established in Remark \ref{reducedRemark}. Furthermore, the discrete vector potential formulation is inherently pressure-decoupled, allowing for an independent pressure solver if required.
	\begin{remark}
		This paper preliminary investigates the grad-curl problem as a vector potential formulation for the Stokes system from a theoretical perspective. Future work will focus on practical computations, including a comparative analysis with the velocity-pressure formulation in FEM and VEM, and the development of efficient preconditioners tailored to the resulting symmetric positive-definite system.
	\end{remark}
	\section{Numerical experiments}
	In this section, we present some numerical results to demonstrate the theoretical results of the virtual element method \eqref{disProblem} in the lowest-order $r=k=1$.
	As in \cite{HuangXHQuadCurl}, we consider the grad-curl problem \eqref{VectorPotentialProblem} with the source form $\boldsymbol{f}$ and the dynamic viscosity $\nu=1$ on a unit cube $\Omega = [0,1]^3$ such that 
	\begin{equation*}
		\boldsymbol{\psi}(x,y,z) = \nabla\times 
		\begin{pmatrix}
			0 \\
			0 \\
			\sin^3(\pi x)\sin^3(\pi y)\sin^3(\pi z)
		\end{pmatrix}.
	\end{equation*}
	
	\par
	We solve the grad-curl problem by using the C++ library Vem++ \cite{VemCode}.
	Two kinds of meshes with different mesh sizes $h$ borrowed from \cite[Fig. 1]{VEMforStokes} are considered as follows.\par
	$\bullet $ Cube: structured meshes consisting of cubes; see Fig. \ref{mesh}(a);\par
	$\bullet $ Voro: Voronoi tessellations optimized by the Lloyd algorithm; see Fig. \ref{mesh}(b);
	\begin{figure}[htbp]
		\centering
		\subfigure[Cube]{
			\begin{minipage}[t]{0.4\textwidth}
				\centering
				\includegraphics[scale=0.4]{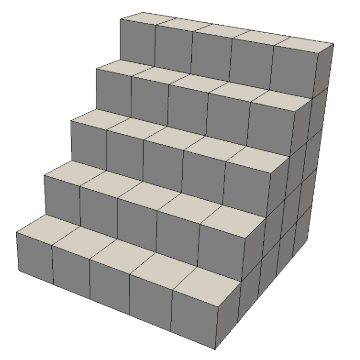}
			\end{minipage}
		}
		\subfigure[Voro]{
			\begin{minipage}[t]{0.4\textwidth}
				\centering
				\includegraphics[scale=0.4]{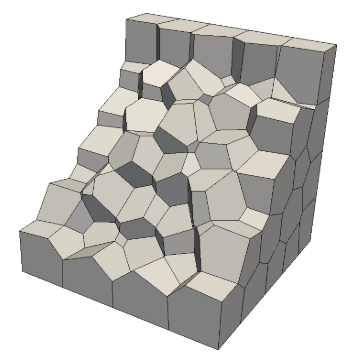}
			\end{minipage}
		}
		\centering
		\caption{Two representatives of the two families of meshes}
		\label{mesh}
	\end{figure}
	\\
	For the computation of the virtual element solution $\boldsymbol{\psi}_h$, since the error $\boldsymbol{\psi}-\boldsymbol{\psi}_h$ is not directly measured for VEM, we define two discrete error norms based on the equivalent bilinear forms \(b_h(\cdot,\cdot)\) and \(a_h(\cdot,\cdot)\) as follows:
	\begin{equation*}
		\begin{aligned}
			\|\boldsymbol{e}\|_h&=\sqrt{b_h(\boldsymbol{I}_h\boldsymbol{\psi}-\boldsymbol{\psi}_h, \boldsymbol{I}_h\boldsymbol{\psi}-\boldsymbol{\psi}_h)},\\
			|\nabla\times \boldsymbol{e}|_{1,h}&
			=\sqrt{a_h\left(\nabla\times(\boldsymbol{I}_h\boldsymbol{\psi}-\boldsymbol{\psi}_h), \nabla\times(\boldsymbol{I}_h\boldsymbol{\psi}-\boldsymbol{\psi}_h)\right)}.
		\end{aligned}
	\end{equation*}
	In fact, from the coercivity, continuity, polynomial approximation, and virtual element interpolation estimates, the computable error here scales like the ``exact one''.
	\par 
	We show the convergence results in Table. 1 and Table. 2 for the lowest order of $r= k = 1$ on the two meshes.
	We find that the convergence orders for the errors $\|\boldsymbol{e}\|_h$ and $|\nabla\times \boldsymbol{e}|_{1,h}$ are at least $O(h)$, which is consistent with the theoretical results.
	\begin{table}[ht]
		\centering
		\caption{Computed errors and rate of convergence with $r = k = 1$ on cube meshes}
		\begin{tabular}{rrrrrrr}
			\toprule
			Ndof &	\( h \) & \( \|\boldsymbol{e}\|_h \) & Rates & \( |\nabla\times \boldsymbol{e}|_{1,h} \) & Rates \\
			\midrule
			1040&	0.433012  & 1.730953E+01 &       & 7.771685E+01 &          \\
			6588&	0.216506  & 3.814650E+00 & 2.1819 & 4.550961E+01 & 0.7721 & \\
			20488&	0.144337  & 1.336897E+00 & 2.5859 & 3.044336E+01 & 0.9916 & \\
			46580&	0.108253  & 6.127802E$-$01 & 2.7117 & 2.284338E+01 & 0.9984  \\
			88704&	0.086602  & 3.316490E$-$01 & 2.7512 & 1.829537E+01 & 0.9949  \\
			\bottomrule 
		\end{tabular}
		\label{tab:Cube}
	\end{table}
	
	\begin{table}[ht]
		\centering
		\caption{Computed errors and rate of convergence with $r = k = 1$ on voro meshes}
		\begin{tabular}{rrrrrrr}
			\toprule
			Ndof &	\( h \) & \( \|\boldsymbol{e}\|_h \) & Rates & \( |\nabla\times \boldsymbol{e}|_{1,h} \) & Rates \\
			\midrule
			1028&	0.568225  & 3.475652E+01 &       & 9.499865E+01 &          \\
			4619&	0.318715  & 9.689208E+00 & 2.2091 & 5.898858E+01 & 0.8241 & \\
			39535&	0.153136  & 1.584513E+00 & 2.4704 & 2.599093E+01 & 1.1182 & \\
			80211&	0.120167  & 8.188194E$-$01 & 2.7230 & 1.970061E+01 & 1.1429  \\
			162179&	0.094650  & 4.342532E$-$01 & 2.6571 & 1.524933E+01 & 1.0730  \\
			\bottomrule
		\end{tabular}
		\label{tab:Voro}
	\end{table}
	
	\bibliographystyle{amsplain}
	\bibliography{GradCurl_reference}
	
\end{document}